\documentclass[11pt]{article}
\usepackage{amsthm}
\usepackage{amsmath}
\usepackage{amssymb}
\usepackage{amsfonts}
\usepackage{amsbsy}

\usepackage{multirow}

\usepackage{graphicx}
\usepackage{blindtext}

\usepackage{caption}
\usepackage{subcaption}

\usepackage{color}
\usepackage{url}
\usepackage{epstopdf}
\usepackage{pdfsync}
\usepackage[colorlinks]{hyperref}
\usepackage{ulem}

\usepackage{accents}
\newcommand{\ubar}[1]{\underaccent{\bar}{#1}}

\hypersetup{
    linkcolor=blue,
}

\usepackage[all]{xy}

\usepackage{algorithm, algorithmic}
\renewcommand{\algorithmiccomment}[1]{\bgroup\hfill//~#1\egroup}

 \numberwithin{equation}{section}

\newlength{\fixboxwidth}
\def\h{h}
\def\z{\zeta}
\def\dt{\Delta t}

\def\V{\mathfrak{V}}
\def\epsilon{\varepsilon}

\def\epsilon{\varepsilon}
\def\db{\mathbf{d}}

\def\R{\mathbb{R}}

\def\N{\mathcal{N}}

\def\Z{\mathcal{Z}}
\def\E{\mathbb{E}}

\def\Z{\mathcal{Z}}

\def\q{r}

\def\I{\mathcal{I}}
\def\J{\mathcal{J}}
\def\L{\mathcal{L}}

\def\W{\mathfrak{W}}

\def\T{\mathcal{T}}

\def\<{\big\langle}
\def\>{\big\rangle}
\def\Img{\operatorname{Im}}
\def\Inv{\operatorname{Inv}}
\def\Ker{\operatorname{Ker}}

\def\Span{\operatorname{span}}

\def\diiv{\operatorname{div}}

\def\supp{{\operatorname{support}}}

\def\diag{\operatorname{diag}}

\def\loc{{\rm loc}}

\def\err{{\rm err}}

\def\app{{\mathrm{ap}}}

\definecolor{red}{rgb}{0.9, 0, 0}

\newtheorem{Theorem}{Theorem}[section]

\newtheorem{Lemma}[Theorem]{Lemma}

\newtheorem{Remark}[Theorem]{Remark}
\newtheorem{Example}{Example}[section]
\newtheorem{Definition}[Theorem]{Definition}
\newtheorem{Construction}[Theorem]{Construction}

\begin{document}
\title{Gamblets for opening the complexity-bottleneck of implicit schemes
for  hyperbolic and parabolic ODEs/PDEs with rough coefficients}

\date{\today}

\author{Houman Owhadi\footnote{California Institute of Technology, Computing \& Mathematical Sciences , MC 9-94 Pasadena, CA 91125, owhadi@caltech.edu} and Lei Zhang\footnote{Shanghai Jiao Tong University,
School of Mathematical Sciences, Institute of Natural Sciences, and Ministry of Education Key Laboratory of Scientific and Engineering Computing (MOE-LSC), Shanghai, 200240, China, lzhang2012@sjtu.edu.cn}}

\maketitle

\begin{abstract}
Implicit schemes are popular methods for the integration of  time dependent PDEs such as hyperbolic and parabolic PDEs. However the necessity to solve corresponding linear systems at each time step constitutes a complexity  bottleneck in their application to PDEs with rough coefficients. We present a generalization of gamblets introduced in
 \cite{OwhadiMultigrid:2015} enabling the  resolution of these implicit systems in near-linear complexity and provide rigorous a-priori error bounds on the resulting numerical approximations of  hyperbolic and parabolic PDEs. These generalized gamblets
induce a multiresolution decomposition of the solution space that is adapted to both the underlying (hyperbolic and parabolic) PDE (and the system of ODEs resulting from  space discretization) and to the time-steps of the numerical scheme.
\end{abstract}


\section{Introduction}\label{secdt}
Implicit schemes are popular and powerful methods for the integration of time dependent PDEs such as hyperbolic and parabolic PDEs \cite{LentVandewalle:2005, Hochbruck:2015a, Hochbruck:2015b, BoomZingg:2015}. However the necessity to solve corresponding linear systems at each time step constitutes a complexity  bottleneck in their application to PDEs with rough coefficients.

Although multigrid methods \cite{Fedorenko:1961, Brandt:1973, Hackbusch:1978}  have been successfully generalized to  time dependent equations \cite{Lubich:1987, Vandewalle:1993, LentVandewalle:2005, Yavneh:2006, Erlangga:2006, Wieners:2010, Hochbruck:2015a}, their convergence rate can be severely affected by the lack of regularity of the coefficients \cite{wan2000}. While some degree of robustness can be achieved with  algebraic multigrid \cite{RugeStuben1987}, multilevel finite element splitting \cite{Yserentant1986}, hierarchical basis multigrid \cite{BankYserentant88}, multilevel preconditioning \cite{Vassilevski89}, stabilized hierarchical basis methods \cite{Vassilevski1997sirev} and energy minimization \cite{Mandel1999, wan2000, Xu2004}, the design of multigrid/multiresolution methods that are provably robust with respect to rough ($L^\infty$) coefficients was an open problem of practical importance \cite{BraWu09} addressed in \cite{OwhadiMultigrid:2015} with the introduction of gamblets  (in $\mathcal{O}(N\ln^{3d}N)$ complexity for the first solve and $\mathcal{O}(N\ln^{d+1}N)$ for subsequent solves  to achieve grid-size accuracy in $H^1$-norm for elliptic problems).
Numerical evidence suggests the robustness of low rank matrix decomposition based methods such as the Fast Multipole Method \cite{GreengardRokhlin:1987, YingBiros:2004a}, Hierarchical matrices \cite{HackbuschGrasedyck:2002, Bebendorf:2008} and Hierarchical Interpolative Factorization \cite{HoYing:2016} and while this robustness  can be proven rigorously for  Hierarchical matrices \cite{ Bebendorf:2008} (the complexity of Hierarchical matrices is $\mathcal{O}(N\ln^{2d+8}N)$  to achieve grid-size accuracy in $L^2$-norm for elliptic problems \cite{ Bebendorf:2008}) one may wonder if it is possible to rigorously lower this known complexity bound and achieve (at the same time) a meaningful multi-resolution decomposition of the solution space for time dependent problems. Although
classical wavelet based methods \cite{Beylkin:1995, Beylkin:1998,  DorobantuEngquist1998} enable a multi-resolution decomposition of the solution space their performance is also affected by the regularity of coefficients because they are not adapted to the underlying PDEs.

In section \ref{secgamblets} we present a generalization of gamblets introduced in \cite{OwhadiMultigrid:2015} and apply them in sections \ref{secwavepde} and \ref{secparabolic} to the implicit schemes for hyperbolic and parabolic PDEs with rough coefficients. As in \cite{OwhadiMultigrid:2015} these generalized gamblets (1) are elementary solutions of hierarchical information games associated with the process of computing with partial information and limited resources, (2) have a natural Bayesian interpretation under the mixed strategy emerging from the game theoretic formulation, (3) induce a  multi-resolution decomposition of the solution space that is adapted to the space-time numerical discretization of the underlying PDE and propagate the solution independently (at each time-step) in each sub-band of the decomposition. The complexity of pre-computing generalized gamblets  is $N\ln^{3d}N$ and that of propagating the solution is $N\ln^{d+1}N$ (at each time step, to achieve grid-size accuracy in energy norm).
Although real valued gamblets are sufficient for first and second order implicit schemes, higher order implicit schemes may require complex valued gamblets. These complex valued gamblets are introduced  and their application to higher order schemes is illustrated in
Section \ref{seccomplexgamblets}. Observing that the multiresolution decomposition induced by gamblets has properties that are similar to an eigenspace decomposition, we introduce, in Section \ref{secmultitimestep}, a multi-time-step scheme for solving parabolic PDEs (with rough coefficients) in $\mathcal{O}(N \ln^{3d+1} N)$ complexity.

Gamblets are derived from a Game Theoretic approach to Numerical Analysis  \cite{OwhadiMultigrid:2015, gamblet17} which could be seen as decision theory approach to numerical analysis \cite{Wald:1945, OwhadiScovel:2015w}. We  refer to the information based complexity literature for an understanding of the natural connection between the notions of computing with partial/priced information and numerical analysis (we refer in particular to \cite{Woniakowski1986, Packel1987, Traub1988, Nemirovsky1992, Woniakowski2009, Ritter2000, Novak2010}). Although  statistical approaches to numerical analysis   \cite{Diaconis:1988, Poincare:1896, Suldin1959, Larkin1972, Sard1963, Kimeldorf70, Shaw:1988, Hagan:1991, Hagan:1992}  have, in the past, received little attention, perhaps due to the counterintuitive nature of the process of randomizing a {\it known} function, the possibilities offered by combining numerical uncertainties/errors with  model uncertainties/errors appear to be stimulating their reemergence  \cite{ChkrebtiiCampbell2015, schober2014nips, Owhadi:2014,Hennig2015, Hennig2015b, Briol2015, Conrad2015, raissi2017inferring, perdikaris2016multifidelity, gamblet17, SchaeferSullivanOwhadi17}. We refer in particular to \cite{Skilling1992, schober2014nips, ChkrebtiiCampbell2015} for ODEs and to \cite{Owhadi:2014,OwhadiMultigrid:2015, gamblet17, Cockayneetal2016} for PDEs. Here the game theoretic approach of \cite{OwhadiMultigrid:2015} is applied to both PDEs and the system of ODEs resulting from their discretization. The multiscale nature of the underlying PDEs results in the stiffness of the corresponding ODEs (these ODEs are non only stiff \cite{TaoOwma10, TaoOwma11, TaoOwma11b} they are also characterized by a large range/continuum of time scales \cite{Owhadi:2003, Owhadi:2004, BenarousOwhadi:2003}). Although it is natural to integrate such ODEs by an eigenspace decomposition when the dimension of the system of ODEs is small, the cost of such an approach is in general prohibitive.
It is to some degree surprising that  gamblets have properties that are similar to eigenfunctions, or more precisely Wannier basis functions \cite{Wannier1937,Marzari2012} (i.e. linear combinations of eigenfunctions concentrated around a given eigenvalue that are also concentrated in space), while preserving the near-linear complexity of the integration.

Since (see  \cite{OwhadiMultigrid:2015}) Gamblets are also natural basis functions for numerical homogenization
 \cite{WhHo87, BaOs:1983, HoWu:1997, EEngquist:2003, OwZh:2007a,  EfGiHouEw:2006,BeOw:2010, BaLip10, DesDonOw:2012, Sym12, Wang:2012, OwZh:2011,  MaPe:2012, OwhadiZhangBerlyand:2014, HouLiu2015, Peterseim2016mug} they can also be employed to achieve sub-linear complexity under sufficient regularity of source terms and initial conditions (see \cite{OwZh06c, OwZh:2007b, OwZh:2011} and Remark \ref{rmksublinear}).

 We also refer to \cite{gamblet17} for a generalization of gamblets to arbitrary continuous linear bijections on Banach spaces.
 As discussed in \cite{gamblet17} gamblets also provide a solution to the problem of identifying
  operator adapted wavelets  \cite{cohen1992biorthogonal,  Bacrywav92,  engquist1994fast, carnicer1996local,  Liandrat01,  Dahmenwav05, Beylkinwav08, sweldens1998lifting, VassilevskiWang1997a,  Sudarshanwav05} satisfying three essential properties (see \cite{stevenson2009adaptive, Sudarshanwav05} for an overview): (a) scale-orthogonality
(with respect to the  operator scalar product to ensure block-diagonal stiffness matrices) (b) local support (or rapid decay) of the wavelets (to
ensures that the individual blocks are sparse) and (c) Riesz stability in the energy norm
(to ensure that the blocks are well-conditioned).

\section{Gamblets}\label{secgamblets}
We will, in this section, present a generalization of the gamblets introduced in \cite{OwhadiMultigrid:2015}. Since the proofs of the results presented in this section are similar to those given in \cite{OwhadiMultigrid:2015} we will refer the reader to \cite{OwhadiMultigrid:2015} and to \cite{gamblet17} for these proofs.

\subsection{The PDE}
Let $\zeta > 0$. Consider the PDE
\begin{equation}\label{eqnscalarzeta}
\begin{cases}
    \frac{4}{\z^2}\mu(x) u(x)-\diiv \big(a(x)  \nabla u(x)\big)=g(x) \quad  x \in \Omega; \\
    u=0 \quad \text{on}\quad \partial \Omega,
    \end{cases}
\end{equation}
where $\Omega$ is a bounded domain in $\R^d$ (of arbitrary dimension $d\in \mathbb{N}^*$) with piecewise Lipschitz boundary,
$a$ is a symmetric, uniformly elliptic $d\times d$ matrix with entries in $L^\infty(\Omega)$ and such that for all $x\in \Omega$ and $l\in \R^d$,
\begin{equation}\label{eqaineq}
\lambda_{\min}(a) |l|^2 \leq l^T a(x) l \leq \lambda_{\max}(a)|l|^2,
\end{equation}
and $\mu\in L^\infty(\Omega)$ with for all $x\in \Omega$,
\begin{equation}\label{eqineqmu}
\mu_{\min} \leq \mu(x) \leq \mu_{\max}\,.
\end{equation}

One purpose of gamblets is to compute the solution of \eqref{eqnscalarzeta} (or its finite-element solution) as fast as possible to a given accuracy.

  \begin{figure}[h!]
	\begin{center}
			\includegraphics[width=0.7\textwidth]{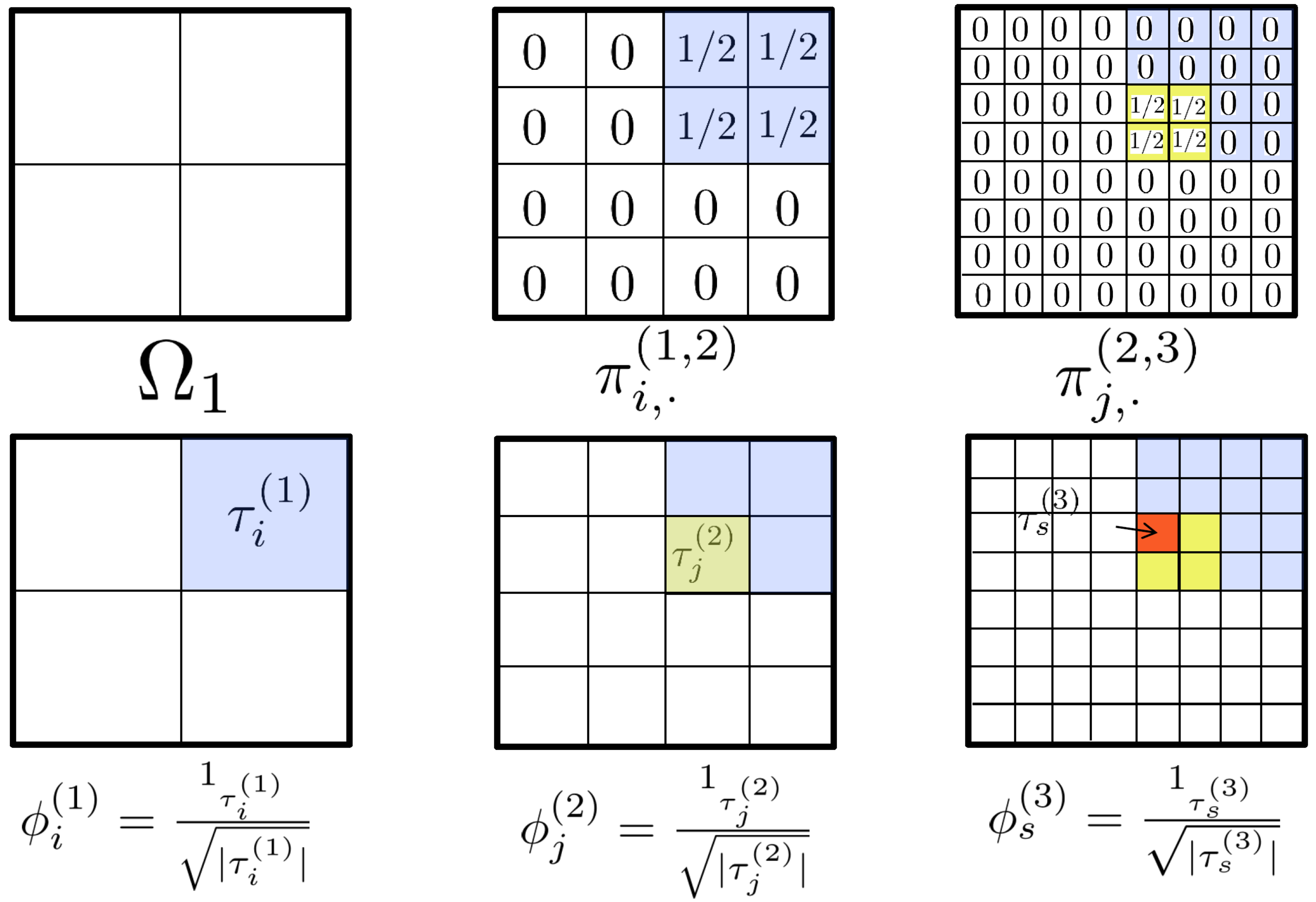}
		\caption{ $\Omega=(0,1)^2$. $\Omega_k$ corresponds to a uniform partition of $\Omega$ into $2^{-k}\times 2^{-k}$ squares. The bottom row shows the support of $\phi_i^{(1)}, \phi_j^{(2)}$ and $\phi_s^{(3)}$. Note that  $j^{(1)}=s^{(1)}=i$ and $s^{(2)}=j$. The top row shows the entries of $\pi^{(1,2)}_{i,\cdot}$ and $\pi^{(2,3)}_{j,\cdot}$.
}\label{figphipi}
	\end{center}
\end{figure}

\subsection{The hierarchy of measurement functions}

We will now introduce a hierarchy of measurement functions that will be used to characterize the process of computing of hierarchies of levels of complexity. We will need the following hierarchy of labels.
\begin{Definition}\label{defindextree}
We say that $\I^{(\q)}$ is an index tree of depth $\q$   if it is the finite set of $\q$-tuples of the form $i=(i_1,\ldots,i_\q)$. For $1\leq k \leq \q$ and $i=(i_1,\ldots,i_\q)\in \I^{(\q)}$,  write
$i^{(k)}:=(i_1,\ldots,i_k)$  and $\I^{(k)}:=\{i^{(k)}\,:\, i\in \I^{(\q)}\}$. For $1<s<k$ and a $k$-tuple of the form $i=(i_1,\ldots,i_k)$ we write $i^{(s)}:=(i_1,\ldots,i_s)$.
\end{Definition}

Write $I^{(k)}$ for the $\I^{(k)}\times \I^{(k)}$ identity matrix.
\begin{Construction}\label{constpiQQ}
For $k\in \{1,\ldots,q-1\}$ let $\pi^{(k,k+1)}$ be a $\I^{(k)}\times \I^{(k+1)}$ matrix such that $\pi^{(k,k+1)}(\pi^{(k,k+1)})^T=I^{(k)}$ and
 $\pi^{(k,k+1)}_{i,j}=0$ for $j^{(k)}\not=i$ (we say that $\pi^{(k,k+1)}$ is cellular).
\end{Construction}

 Let $(\phi_i^{(\q)})_{i\in \I^{(\q)}}$ be orthonormal elements of $L^2(\Omega)$ and, for $k\in \{1,\ldots,\q-1\}$ and $i\in \I^{(k)}$ define $\phi_i^{(k)}$ via induction by
\begin{equation}\label{eqndkjndnfj}
\phi_i^{(k)}=\sum_{j\in \I^{(k+1)}} \pi^{(k,k+1)}_j \phi_j^{(k+1)}
\end{equation}

We will refer to the elements $\phi_i^{(k)}$ as measurement functions. Through this paper we use
Haar wavelets or approximations thereof (Construction \ref{consphi}) as prototypical measurement functions. We refer the reader to
\cite{gamblet17} for a comprehensive description of the framework.

 \begin{Construction}\label{consphi}
Let $H, \delta\in (0,1)$. For $k\in \N^*$, let $\Omega_{k}$ be a nested  partition of $\Omega$ into subsets $(\tau^{(k)}_i)_{i\in \I^{(k)}}$ such that (1) each $\tau_i^{(k)}$ is contained in a ball of radius $H$ and contains a ball of radius $\delta H$ and (2) $|\tau^{(k)}_i|=|\tau^{(k)}_j|$ ($|\tau^{(k)}_i|$ is the volume of $\tau^{(k)}_i$). Let
 $\phi_i^{(k)}=\frac{1_{\tau^{(k)}_i}}{\sqrt{|\tau^{(k)}_i|}}$ where  $1_{\tau^{(k)}_i}$ is the indicator function of $\tau^{(k)}_i$.  Observe that the nesting matrices $\pi^{(k,k+1)}$ are cellular and orthonormal (in the sense that $\pi^{(k,k+1)}(\pi^{(k,k+1)})^T=I^{(k)}$ where $I^{(k)}$ is the $\I^{(k)}\times \I^{(k)}$ identity matrix).
 \end{Construction}

 \begin{Example}
 For our running example we will consider $\Omega=(0,1)^2$ illustrated in Figure \ref{figphipi} (taken from \cite{gamblet17}).
 Using the Construction \ref{consphi} we select
  $\Omega_{k}$ to be a regular grid partition of $\Omega$ into $2^{-k}\times 2^{-k}$ squares $\tau^{(k)}_i$  and  $\phi_i^{(k)}=\frac{1_{\tau^{(k)}_i}}{\sqrt{|\tau^{(k)}_i|}}$.
 \end{Example}

\subsection{The hierarchy of games}
Gamblets are then identified by turning the process of computing with limited resources and partial information as that of playing hierarchies of games defined as follows. We have two players I and II. Player I chooses
the right hand side $g$ of \eqref{eqnscalarzeta} in $H^{-1}(\Omega)$ and does not show it to Player II.
Starting with $k=1$, Player II sees  $(\int_{\Omega} u \phi_i^{(k)})_{i\in \I^{(k)}}$ and must  predict $u$ and $(\int_{\Omega} u \phi_i^{(k+1)})_{i\in \I^{(k+1)}}$. Once Player II has made his choice, he gets a loss, sees $(\int_{\Omega} u \phi_i^{(k+1)})_{i\in \I^{(k+1)}}$ and must predict $u$ and $(\int_{\Omega} u \phi_i^{(k+2)})_{i\in \I^{(k+2)}}$.
In this adversarial game Player I tries to maximize the loss of Player II and Player II tries to minimize it.
Optimal strategies are identified by lifting this deterministic minimax problem to a minimax over measures \cite[Sec.~5]{gamblet17}. In other words, Player I must play at random and Player II must look for an optimal strategy in the Bayesian class of strategies  by considering the SPDE
\begin{equation}\label{eqnscalarzetaspde}
\begin{cases}
    \frac{4}{\z^2}\mu v-\diiv \big(a  \nabla v\big)=\xi \quad  x \in \Omega; \\
    u=0 \quad \text{on}\quad \partial \Omega,
    \end{cases}
\end{equation}
where the right hand side of \eqref{eqnscalarzeta} has been replaced by a random field $\xi$ and the  bet of Player II at step $k$ is the expectation of the solution of the SPDE \eqref{eqnscalarzetaspde} conditioned on measurements of the solution of the deterministic PDE
\eqref{eqnscalarzeta}, i.e.
\begin{equation}\label{eqddjgd}
u^{(k),\z}(x):=\E\big[v(x)\big| \int_{\Omega} v(y)\phi_i^{(k)} (y)\,dy=\int_{\Omega} u(y)\phi_i^{(k)} (y)\,dy,\,i\in \I^{(k)} \big]\,.
\end{equation}
Note that the sequence of approximations \eqref{eqddjgd} is a martingale under the filtration formed by the measurements $(\int_{\Omega} u \phi_i^{(k)})_{i\in \I^{(k)}}$.

\subsection{$\z$-Gamblets}
If the loss of Player II is measured using relative error in the energy norm $\|w\|_\z^2:=\frac{4}{\z^2}\int_{\Omega}  w^2 \mu+\int_{\Omega} (\nabla w)^T a \nabla w$ associated with the operator scalar product
\begin{equation}\label{eqsp}
\<w_1,w_2\>_\z:=\frac{4}{\z^2}\int_{\Omega}  w_1 w_2 \mu+\int_{\Omega} (\nabla w_1)^T a \nabla w_2\,,
\end{equation}
then \cite[Sec.~5]{gamblet17} the optimal strategy of Player II is to select the distribution of $\xi$ as that of a centered Gaussian field with covariance operator $\L=\frac{4}{\z^2}\mu \cdot-\diiv \big(a  \nabla \cdot\big)$. This simply means that if $f\in H^1_0(\Omega)$ then $\int_{\Omega} f \xi$ is a centered Gaussian random variable of variance $\|f\|_\z^2$. Under that choice $u^{(k),\z}$ can be written as a linear combination of the measurements, i.e.
\begin{equation}
u^{(k),\z}(x)=\sum_{i\in \I^{(k)}}\psi_i^{(k),\z}(x)\int_{\Omega} u(y)\phi_i^{(k)}(y)\,dy
\end{equation}
and the coefficients $\psi_i^{(k),\z}$ are deterministic functions and elementary gambles (namely, \textit{$\zeta$-gamblets} or \textit{gamblets}) forming a basis for Player II's strategy ($\psi_i^{(k),\z}(x)$ is the best bet of Player II on the value of $u(x)$ given the information that $\int_{\Omega} u \phi_j^{(k)}=\delta_{i,j}$ for $j\in \I^{(k)}$).
As shown in \cite{OwhadiMultigrid:2015} (see also \cite{Owhadi:2014, gamblet17}),  gamblets are  optimal recovery splines  \cite{micchelli1977survey} characterized by
optimal variational and recovery properties.
\begin{Theorem}
It holds true that (1) for $i\in \I^{(k)}$,
\begin{equation}
\psi_i^{(k),\z}=\sum_{j\in \I^{(k)}}\Theta_{i,j}^{(k),-1} \L^{-1}\phi_j^{(k)}
\end{equation}
where $\Theta^{(k),-1}$ is the inverse of the Gramian matrix $\Theta^{(k)}_{i,j}:=\int_{\Omega}\phi_i^{(k)}\L^{-1}\phi_j^{(k)}$
(2) for $w\in \R^{\I^{(k)}}$, $\sum_{i\in \I^{(k)}} w_i \psi_i^{(k),\z}$ is the minimizer of $\|\psi\|_\z$ over all functions $\psi\in H^1_0(\Omega)$ such that $\int_{\Omega}\psi \phi_i^{(k)}=w_i$ for $i\in \I^{(k)}$  and (3) $u^{(k),\z}$ is the minimizer of $\|u-\psi\|_\z$ over all functions $\psi$ in $\Span\{\L^{-1}\phi_i^{(k)}\mid i\in \I^{(k)}\}$.
\end{Theorem}

 Furthermore  $\psi_i^{(k),\z}$ decays exponentially fast away from the support of $\phi_i^{(k)}$ and this exponential decay can be used to localize the nested computation of gamblets. To simplify the presentation,  we will  write $C$ any constant that depends only on $d, \Omega, \lambda_{\min}(a), \lambda_{\max}(a), \mu_{\min},\mu_{\max}, \delta$ but not on $\z$ nor $H$  (e.g., $2 C \z H^2\lambda_{\max}(a)$ will  be written $C \z H^2$).

 \begin{Theorem}\label{thmodhehiudhehdswws}
Let  $\phi_i^{(k)}$ be as in Construction \ref{consphi}. Let $\Omega_{i,n}^{(k)}$ be the union of subsets $\tau_j^{(k)}$ that are at distance at most $n H$ from $\tau_i^{(k)}$. Let
$\psi_i^{(k),\z,n}$ be the minimizer of $\|\psi\|_\z$ over all functions $\psi\in H^1_0(\Omega_{i,n}^{(k)})$ such that $\int_{\Omega}\psi \phi_j^{(k)}=\delta_{i,j}$ for $j\in \I^{(k)}$. We have $\|\psi_i^{(k),\z}-\psi_i^{(k),\z,n}\|_\z \leq C \|\psi_i^{(k),\z,0}\|_\z e^{- C^{-1}n}$
\end{Theorem}

\begin{Remark}
The optimal prior is Gaussian because \cite[Sec.~5]{gamblet17} of the linearity of the PDE and the quadratic nature of the loss function. For non linear PDEs or non quadratic loss functions, although optimal priors (which may not be Gaussian) could in principle be numerically approximated, such approximations could be severely impacted by stability issues as discussed in
 \cite{OSS:2013, OwhadiScovel:2013, owhadiBayesiansirev2013, osqr16}.
\end{Remark}

\begin{figure}[H]
\centering
\includegraphics[scale = 0.35]{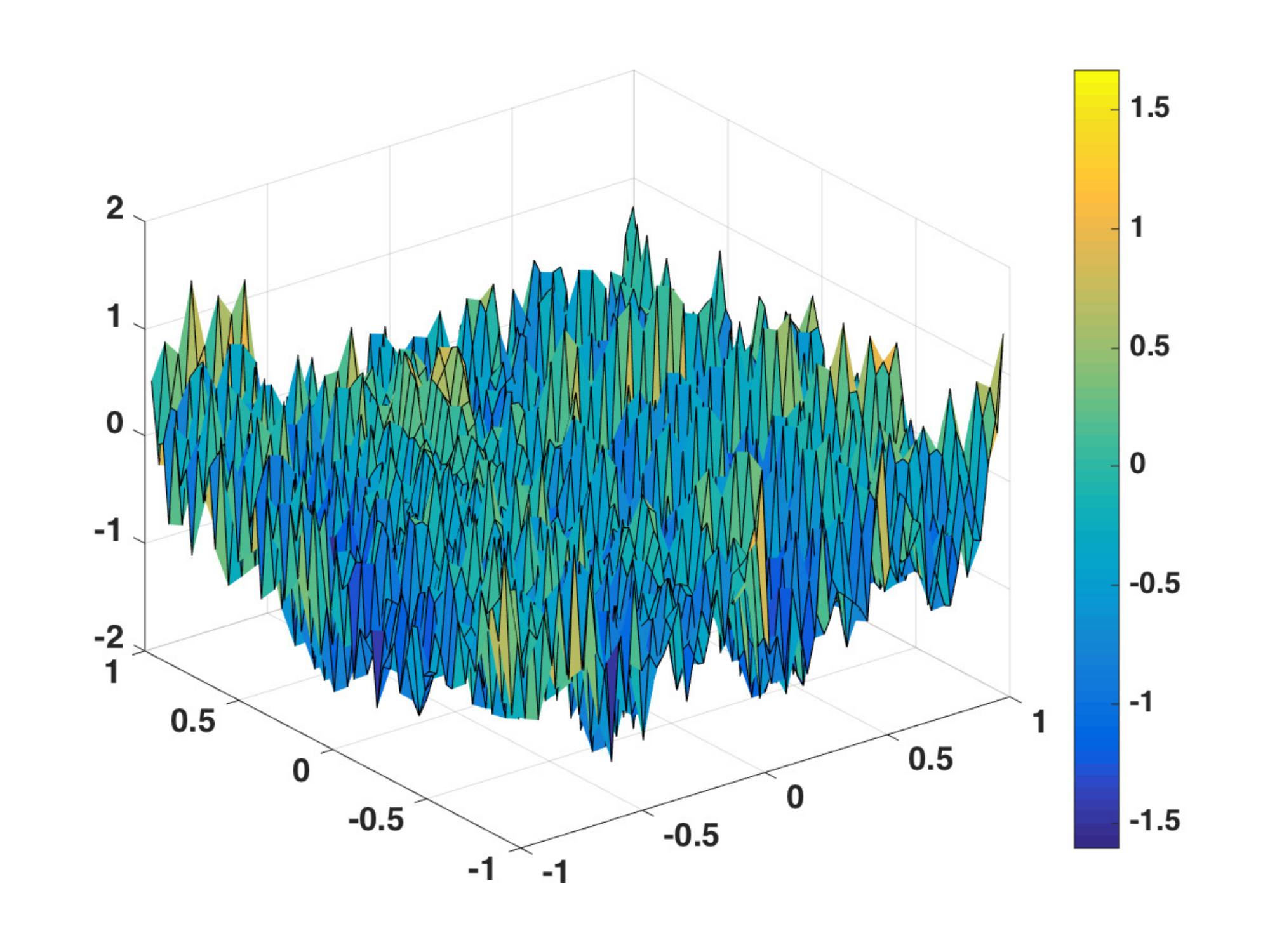}
\caption{$a(x)$ in log scale.}
\label{fig:coef}
\end{figure}

\begin{Example}\label{egadef}
For our numerical examples/illustrations, $d=2$, $\Omega=(0,1)^2$ and $\T_h$ is a square grid of mesh size $h=(1+2^{q})^{-1}$ with $\q=6$ and $64\times 64$ interior nodes, $a$ is piecewise constant on each square of $\T_h$ and given by
\begin{equation}\label{ega}
	a(x) = \Pi_{k=1}^\q \big(1+0.5\cos(2^k\pi(\frac{i}{2^\q+1}+\frac{j}{2^\q+1}))\big)
				\big(1+0.5\sin(2^k\pi(\frac{j}{2^\q+1}-3\frac{i}{2^\q+1}))\big)
\end{equation}
for $\displaystyle x\in[\frac{i}{2^\q+1}, \frac{i+1}{2^\q+1})\times[\frac{j}{2^\q+1}, \frac{j+1}{2^\q+1})$ as illustrated in  Figure \ref{fig:coef}.
We use
continuous bilinear nodal basis elements $\varphi_i$ spanned by $\displaystyle\{1,x_1,x_2,x_1 x_2\}$ in each square of $\T_h$.
Figure \ref{figgambletszeta} then provides an illustration of gamblets for various values of $\zeta$.
Note that the generalized gamblet $\psi_i^{(k),\z}$ can be seen as a non-linear interpolation between a re-scaling of the measurement function $\phi_i^{(k)}$ ($\z=0$) and the gamblets introduced in \cite{OwhadiMultigrid:2015} ($\z=\infty$).
\end{Example}

\begin{figure}[tp]
	\begin{center}
 		\includegraphics[height=6.6cm,width=\textwidth]{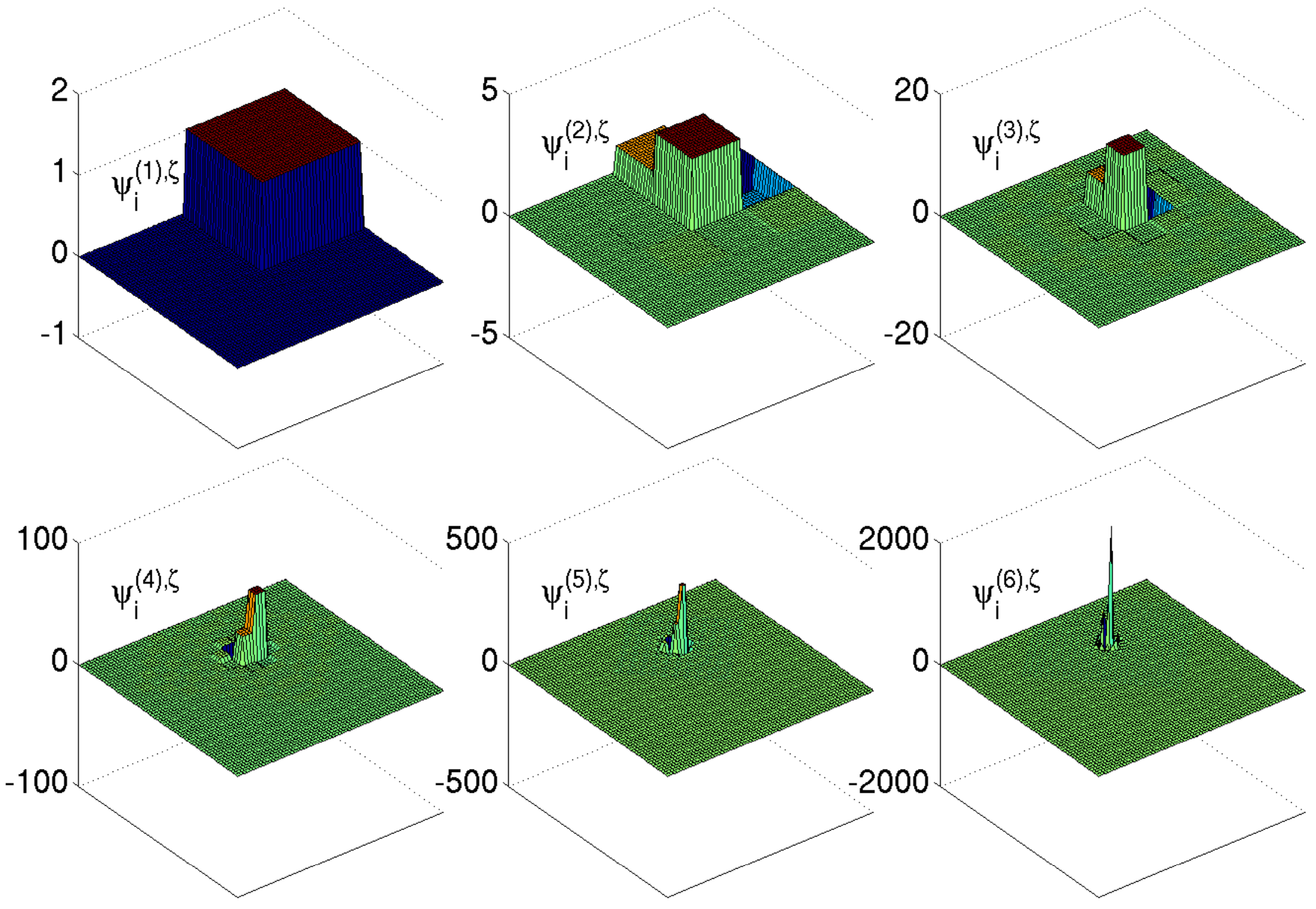} (a)\\
		\includegraphics[height=6.6cm,width=\textwidth]{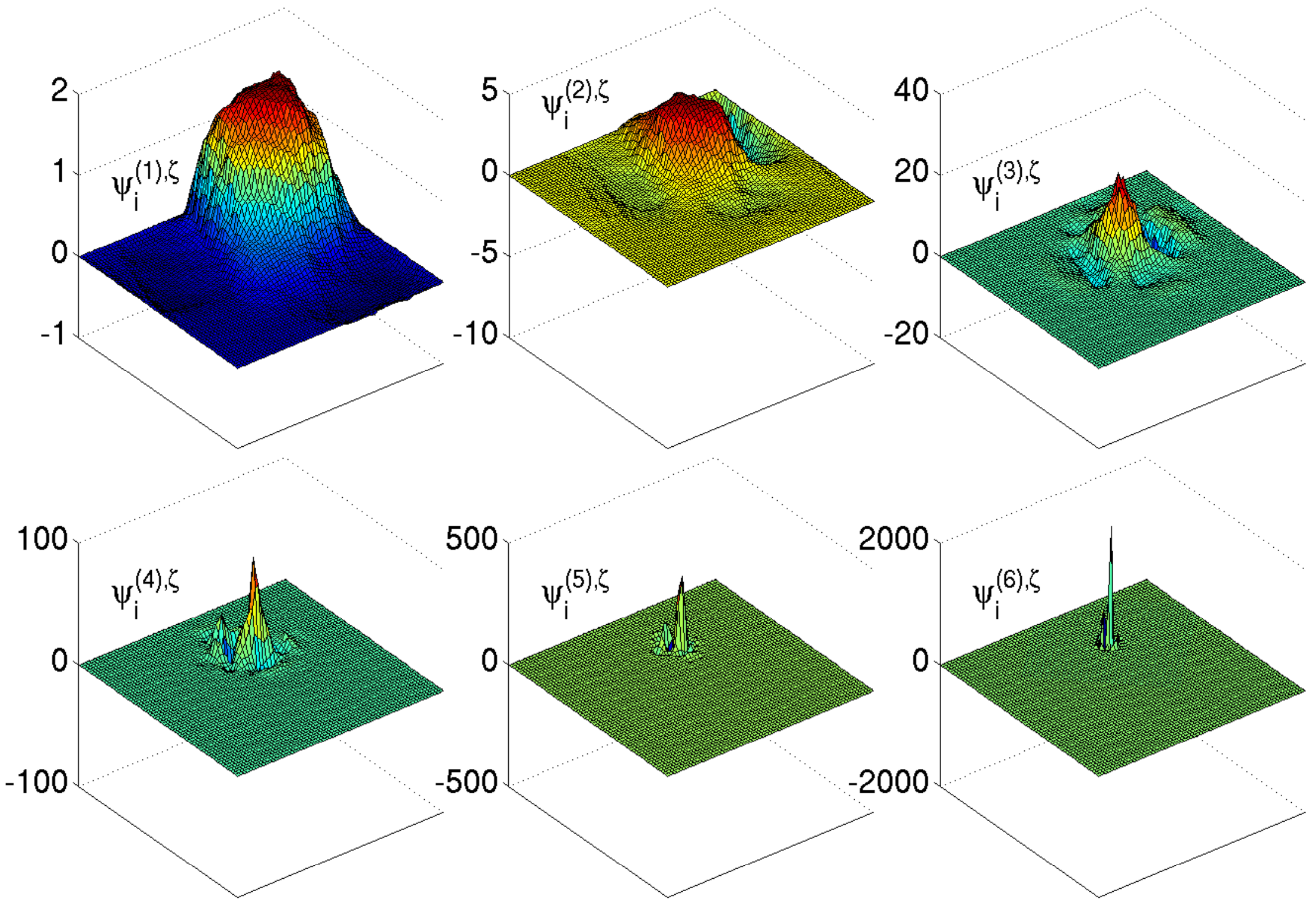} (b)\\
		\includegraphics[height=6.6cm,width=\textwidth]{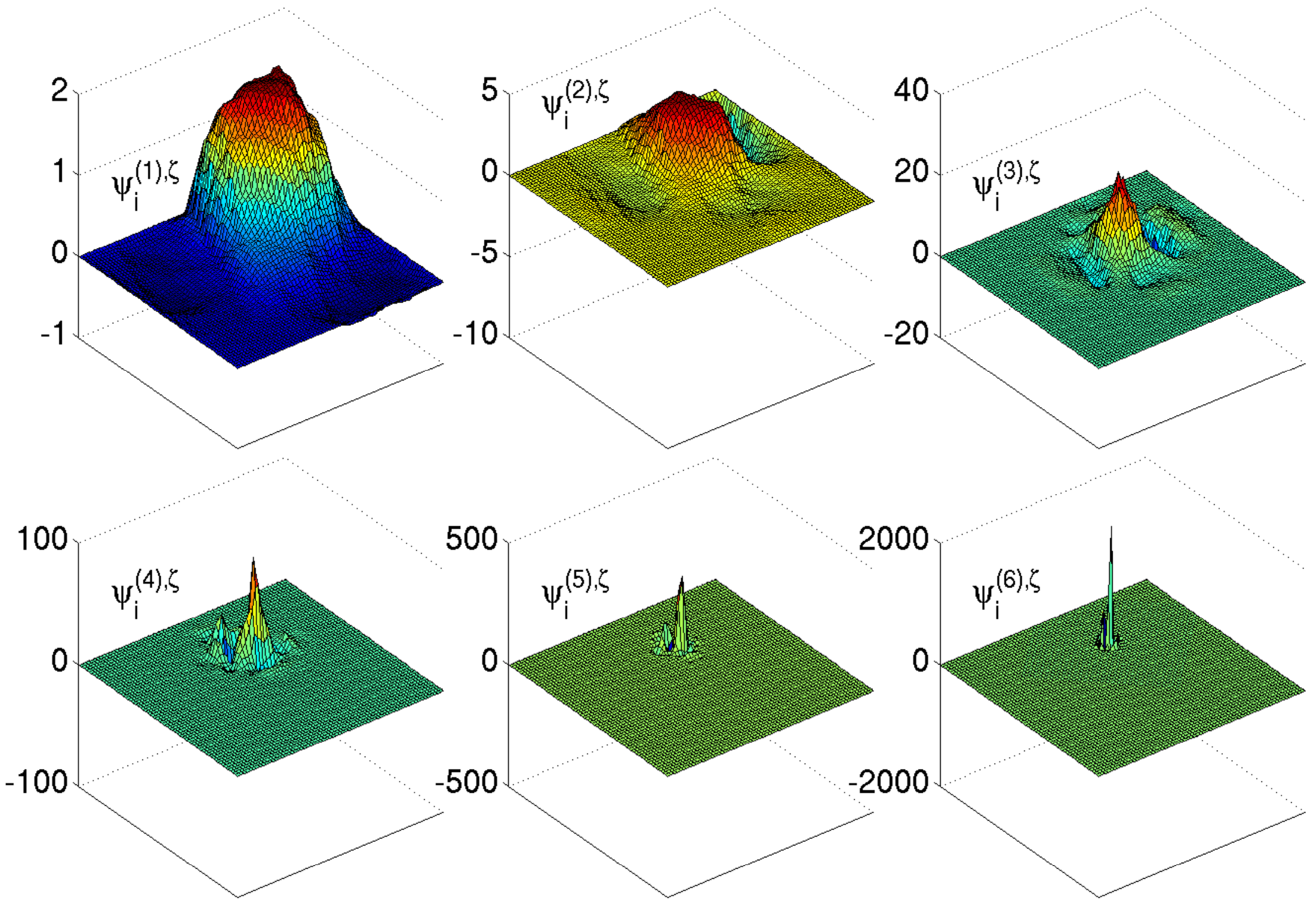} (c)
 		\caption{$\psi_i^{(k), \z}$ for (a) $\z=10^{-6}$ (b) $\z=1$ (c) $\z=10^6$}
 		\label{figgambletszeta}
 	\end{center}
 \end{figure}

\subsection{Multiresolution decomposition}
The nesting of the measurements function implies that of the gamblets, i.e. writing (for  $k\in \{1,\ldots,\q\}$)
$\V^{(k),\z}:=\operatorname{span}\{\psi^{(k),\z}_i \mid i\in \I^{(k)}\}$ we have (for  $k\in \{1,\ldots,\q-1\}$) $\V^{(k),\z}\subset \V^{(k+1),\z}$ and $\psi^{(k),\z}_i(x) =\sum_{j\in  \mathcal{I}_{k+1}} R^{(k),\z}_{i,j}\psi^{(k+1),\z}_j(x) $ where $R^{(k),\z}$ is the so called restriction/prolongation operator whose entry $R^{(k),\z}_{i,j}$ can be identified as,  $R^{(k),\z}_{i,j}=
\mathbb{E}\big[\int_{\Omega} v(y)\phi^{(k+1)}_j(y)\,dy\big| \int_{\Omega} v(y)\phi^{(k)}_l(y)\,dy=\delta_{i,l},\,l\in \mathcal{I}^{(k)} \big]$, i.e. the best bet of Player II on the value of $\int_{\Omega}u \phi_j^{(k+1),\z}$ given the information that $\int_{\Omega}u \phi_l^{(k),\z}=\delta_{i,l}$.
With the identification of the restriction/prolongation operator one can use gamblets to couple scales in a multigrid algorithm but here we will instead use gamblets to induce a multiresolution decomposition of the solution space via orthogonalization process akin to the one used with wavelets.

\begin{Definition}
For $k\in \{2,\ldots,\q\}$ let $\J^{(k)}$ be a finite set of $k$-tuples of the form $j=(j_1,\ldots,j_k)$
such that $\{j^{(k-1)}\mid j\in \J^{(k)}\}=\I^{(k-1)}$ and for $i\in \I^{(k-1)}$, $\operatorname{Card}\{ j\in \J^{(k)}\mid j^{(k-1)}=i\}=\operatorname{Card}\{ s\in \I^{(k)}\mid s^{(k-1)}=i\}-1$.
\end{Definition}

\begin{Definition}\label{defwk}
Let  $W^{(k)}$  be a $\J^{(k)}\times \I^{(k)}$ matrix such that: (1) $\Img(W^{(k),T})=\Ker(\pi^{(k-1,k)})$, (2)  $W^{(k)}W^{(k),T}=J^{(k)}$ where $J^{(k)}$ is the $\J^{(k)}\times \J^{(k)}$ identity matrix, (3)  $W^{(k)}_{j,i}=0$ for $(j,i)\in \J^{(k)}\times \I^{(k)}$ with  $j^{(k-1)}\not=i^{(k-1)}$.
\end{Definition}

When measurement functions are in Construction \ref{consphi} then
 an example of $W^{(k)}$ is provided in Construction \ref{const2} based on the following lemma.

\begin{Lemma}\label{const0}
Let $U^{(n)}$ be the sequence of $n\times n$ matrices defined (1) for $n=2$ by $U^{(2)}_{1,\cdot}=(1,-1)$ and $U^{(2)}_{2,\cdot}=(1,1)$ and (2) iteratively for $n\geq 2$ by $U^{(n+1)}_{i,j}=U^{(n)}_{i,j}$ for $1\leq i,j \leq n$, $U^{(n+1)}_{n+1,j}=1$ for $1\leq j \leq n+1$, $U^{(n+1)}_{i,n+1}=0$ for $1\leq i\leq n-1$ and $U^{(n+1)}_{n,n+1}=-n$. Then for $n\geq 2$, the rows of $U^{(n)}$ are orthogonal, $U^{(n)}_{n,j}=1$ for $1\leq j \leq n$ and we write $\bar{U}^{(n)}$ the corresponding orthonormal matrix obtained by renormalizing the rows of $U^{(n)}$.
\end{Lemma}

Note that another possible choice for $U^{(n)}$ (than the one described in Lemma \ref{const0}) is the discrete cosine transformation matrix.

\begin{Construction}\label{const2}
For $k\in \{2,\ldots,\q\}$, let $W^{(k)}$  be a $\J^{(k)}\times \I^{(k)}$ matrix such that: (1) $W^{(k)}_{j,i}=0$ for $(j,i)\in \J^{(k)}\times \I^{(k)}$ with  $j^{(k-1)}\not=i^{(k-1)}$, (2) for
$s\in \I^{(k-1)}$ and $t\in \{1,\ldots,n-1\}$ and $t' \in \{1,\ldots,n\}$, $W^{(k)}_{(s,t),(s,t')}=\bar{U}^{(n)}_{t,t'}$ (where $\bar{U}^{(n)}$ is defined in Lemma \ref{const0} and $n=\operatorname{Card}\{i\in \I^{(k)}\mid i^{(k-1)}=s\}$).
\end{Construction}

For $k\in \{2,\ldots,\q\}$ and $i\in \J^{(k)}$ let
\begin{equation}
\chi_i^{(k),\z}=\sum_{j\in \I^{(k)}}W_{i,j}^{(k)}\psi_j^{(k),\z}
\end{equation}
and
\begin{equation}
\W^{(k),\z}:=\Span\{\chi_i^{(k),\z} \mid i\in \I^{(k)}\}
\end{equation}
For  $k\in \{2,\ldots,\q\}$, write  $\W^{(k),\z}:=\operatorname{span}\{\chi^{(k),\z}_i \mid i\in \J^{(k)}\}$. Write $\oplus_\z$ the orthogonal direct sum with respect to the scalar product $\<\cdot,\cdot\>_\z$. The following theorem shows that $\W^{(k),\z}$ is the orthogonal complement of $\V^{(k),\z}$ in $\V^{(k-1),\z}$ and this induces a multiresolution decomposition of the solution space.
\begin{Theorem}\label{thmgugyug0}
 It holds true that for $k\in \{2,\ldots,\q\}$, $\V^{(k),\z}=\V^{(k-1),\z} \oplus_\z \W^{(k),\z}$ and, in particular
\begin{equation}\label{eqdedhhiuhe3}
\V^{(\q),\z}=\W^{(1),\z}\oplus_\z \W^{(2),\z} \oplus_\z  \cdots \oplus_\z \W^{(\q),\z},
\end{equation}
where $\W^{(1),\z} = \V^{(1),\z}$. Furthermore, $u^{(1)}$ is the finite-element solution of \eqref{eqnscalarzetaspde} in $\V^{(1),\z}$ and for $k\in \{2,\ldots,\q\}$, $u^{(k),\z}-u^{(k-1),\z}$ is the finite element solution of \eqref{eqnscalarzetaspde} in $\W^{(k),\z}$.
\end{Theorem}

Note that since the spaces $\W^{(k),\z}$ for $k\in\{1,\ldots,\q\}$ are orthogonal with each other, the corresponding finite-element subband solutions $u^{(1),\z}$ and $u^{(k),\z}-u^{(k-1),\z}$ for $k\in\{2,\ldots,\q\}$ can be computed independently. Figure \ref{figmultiresu} provides an illustration of the subband solutions for the solution $u(x)$ of equation \eqref{eqnscalarzeta} with $g(x)=\sin(\pi x_1)\cos(\pi x_2)$.

\begin{figure}[H]
\centering
\includegraphics [width=6cm]{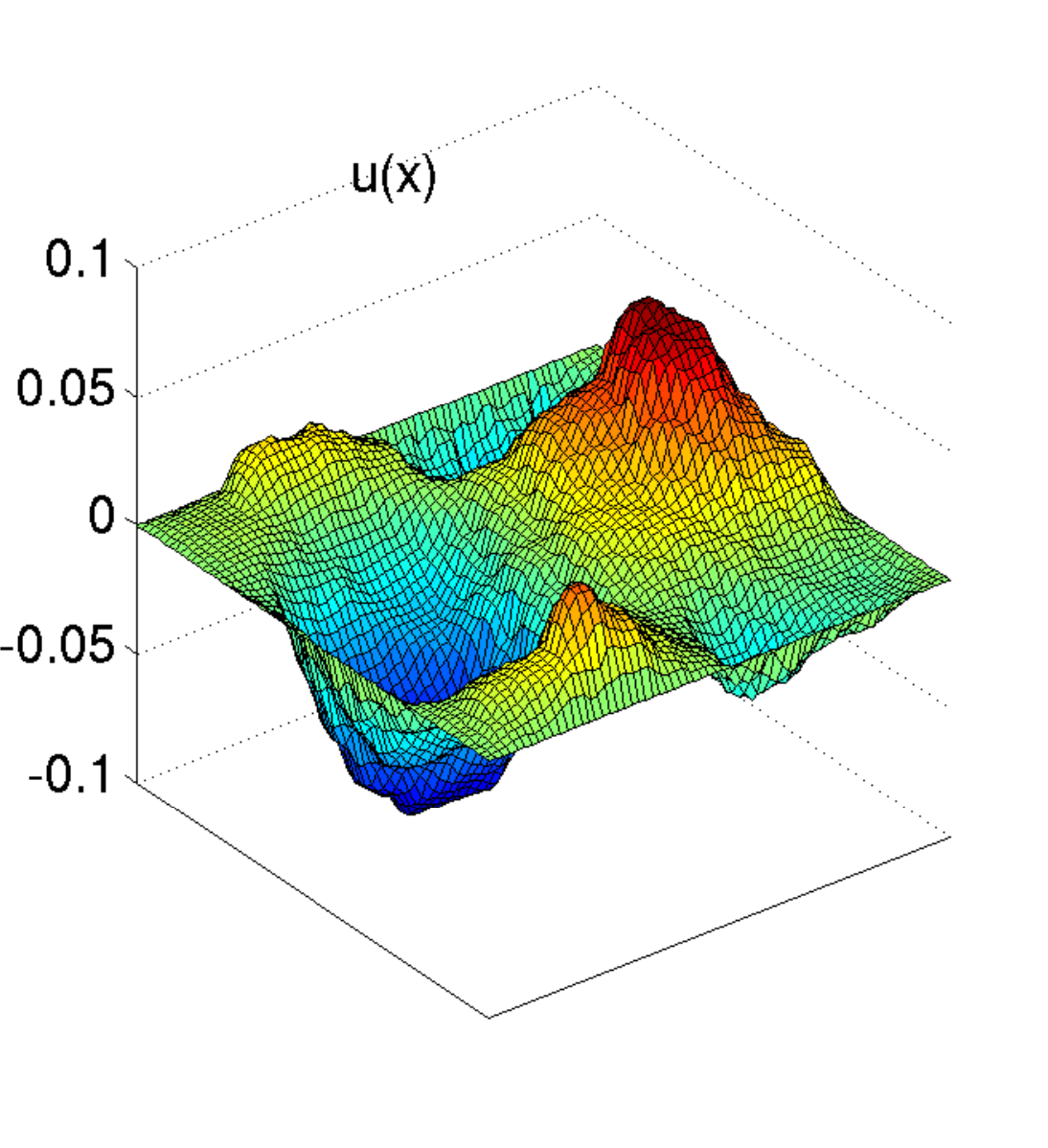}
\caption{Solution $u(x)$ of equation \eqref{eqnscalarzeta} with $g(x)=\sin(\pi x_1)\cos(\pi x_2)$.}
\label{fig:ucoef}
\end{figure}

\begin{figure}[tp]
	\begin{center}
 		\includegraphics[height=6.6cm,width=\textwidth]{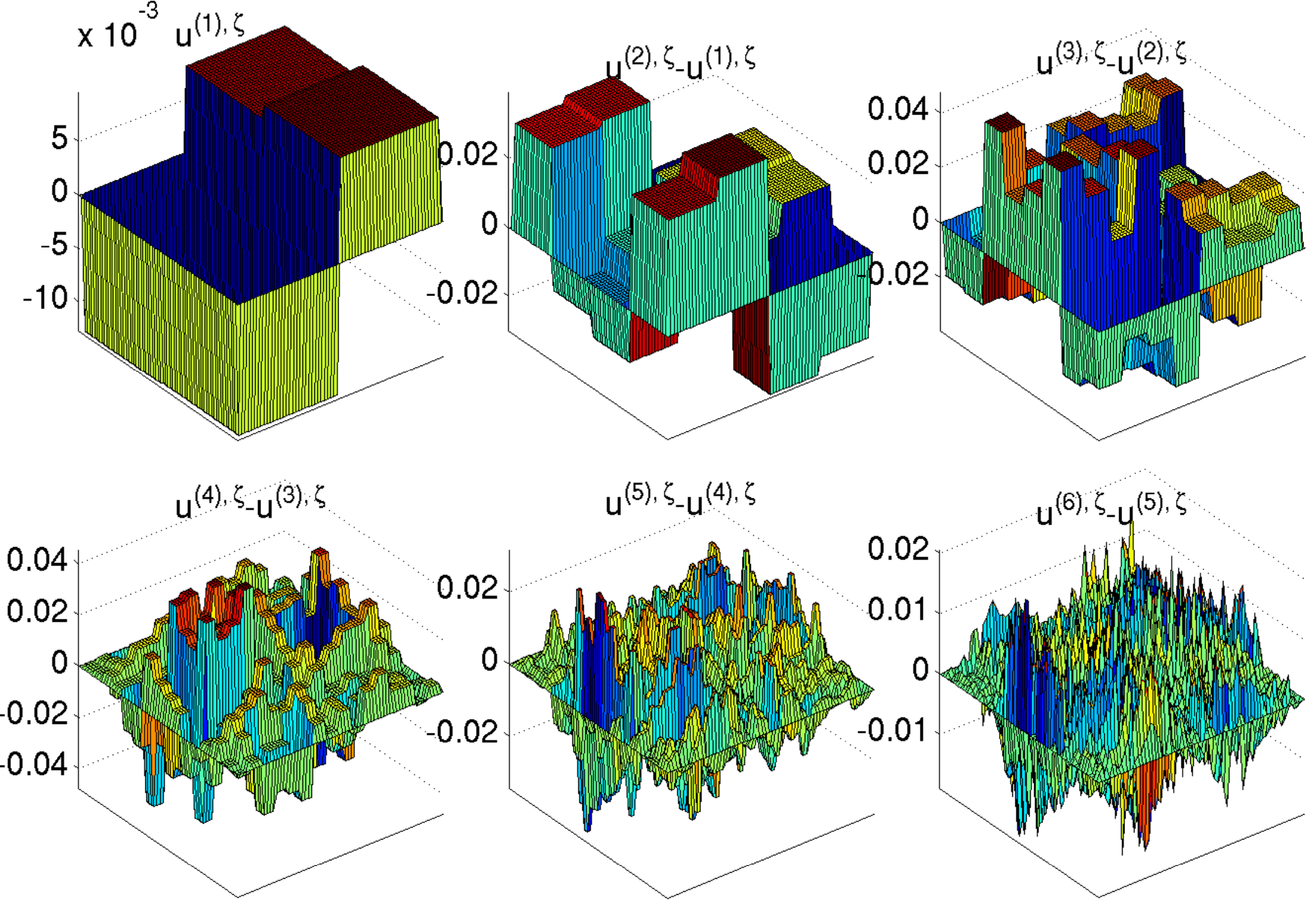} (a)\\
		\includegraphics[height=6.6cm,width=\textwidth]{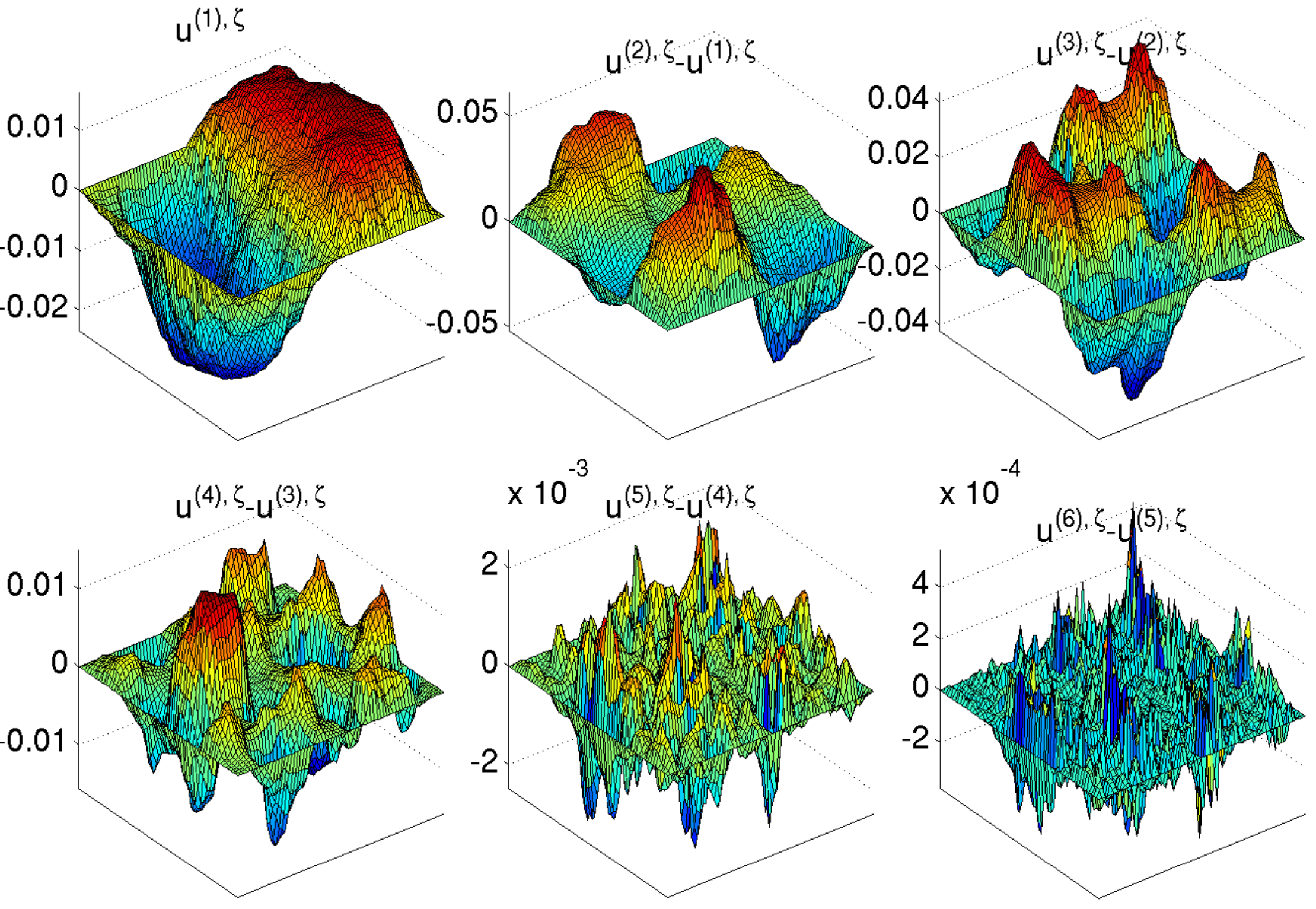} (b)\\
		\includegraphics[height=6.6cm,width=\textwidth]{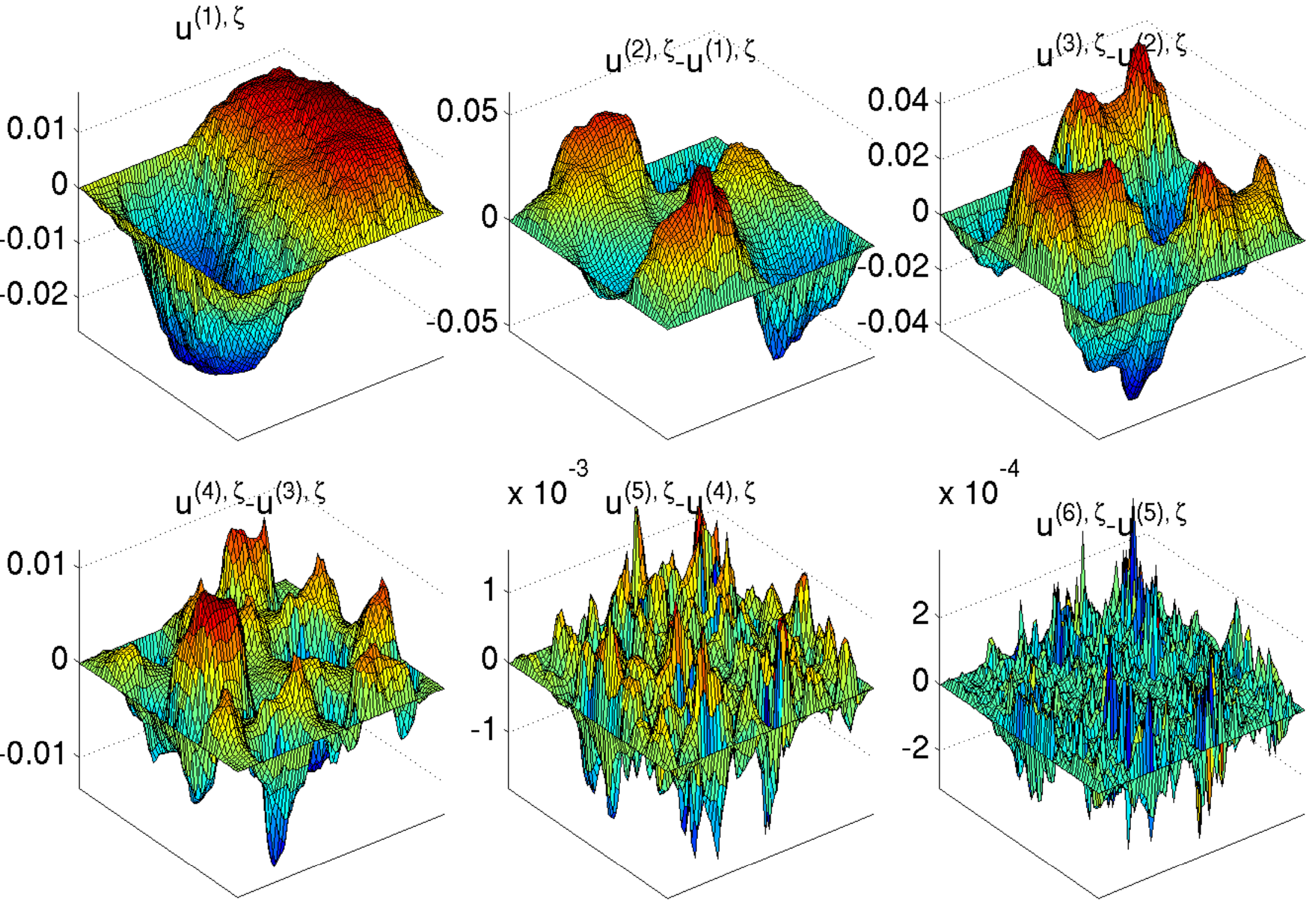} (c)
		\caption{Multiresolution decomposition of $u$ for $\z=10^{-6}$, $\z=1$ and $\z=10^6$}
		\label{figmultiresu}
	\end{center}
\end{figure}

\begin{Remark}
An analogy could be made between gamblets and
Lax Pairs \cite{lax1968integrals} where the solution space is also decomposed
in a way that involves the dynamics of the PDE itself. Here the pairs $\phi_i^{(k)}$ and $\phi_i^{(k),\z}$ form a biorthogonal system \cite{dieudonne1953biorthogonal} in the sense that $\int_{\Omega}\phi_i^{(k)}\psi_j^{(k),\z}=\delta_{i,j}$ for $i,j\in \I^{(k)}$ and the $\<\cdot,\cdot\>_\z$-orthogonal projection of $u\in H^1_0(\Omega)$ onto $\V^{(k),\z}$ is $\sum_{i\in \I^{(k)}} \psi_i^{(k)} \int_{\Omega}\phi_i^{(k)}u $. As discussed in \cite{gamblet17} gamblets are also optimal recovery splines in the sense of Micchelli and Rivlin
 \cite{micchelli1977survey} and as a consequence have optimal recovery properties  \cite{gamblet17}. 
\end{Remark} 

\subsection{Uniformly bounded condition numbers}
Let $A^{(k),\z}$ and $B^{(k),\z}$ be the stiffness matrices of finite element approximation of the \eqref{eqnscalarzetaspde} in $\V^{(k),\z}$ and $\W^{(k),\z}$, i.e.  $A^{(k),\z}_{i,j}:=\<\psi_i^{(k)},\psi_j^{(k)}\>_\z$ for $k\in \{1,\ldots,\q\}$ and $i,j\in \I^{(k)}$ and $B^{(k),\z}_{i,j}:=\<\chi_i^{(k),\z},\chi_j^{(k),\z}\>_\z$ for  $k\in \{2,\ldots,\q\}$ and $i,j\in \J^{(k)}$.
The following theorem shows that the multiresolution decomposition of Theorem \eqref{thmgugyug0} looks like an eigenspace decomposition in the sense that those subspaces are orthogonal with respect to the scalar product $\<\cdot,\cdot\>$ and the condition numbers of the matrices $B^{(k),\z}$ are uniformly bounded (the subspaces are not orthogonal in $L^2(\Omega)$ so \eqref{eqdedhhiuhe3} is not an exact eigenspace decomposition).
For a given matrix $M$, write $\operatorname{Cond}(M):=\sqrt{\lambda_{\max}(M^T M)}/\sqrt{\lambda_{\min}(M^T M)}$ its  condition number.
If $M$ is symmetric write $\lambda_{\min}(M)$ and $\lambda_{\max}(M)$ its minimal and maximal eigenvalues.

\begin{Theorem}\label{thmodhehiudhehd}
Let the $\phi_i^{(k)}$ be as in Construction \ref{consphi}.
For $\z\in (0,\infty]$, $\operatorname{Cond}(A^{(1),\z})\leq C H^{-2}$, and
 $\operatorname{Cond}(B^{(k),\z})\leq C H^{-2}$ for $k\in \{2,\ldots,\q\}$. Furthermore, for $\z=\infty$ and $k\in \{1,\ldots,\q\}$,
 $\displaystyle \frac{1}{C}  \leq \lambda_{\min}(A^{(k),\infty})$  and $\lambda_{\max}(A^{(k),\infty})\leq C H^{-2k}$. For $\z=\infty$ and $k\in \{2,\ldots,\q\}$,
 $\displaystyle \frac{1}{C} H^{-2 (k-1) }   \leq \lambda_{\min}(B^{(k),\infty})$  and $\lambda_{\max}(B^{(k),\infty})\leq  C H^{-2k}$.
\end{Theorem}
\begin{figure}[H]
\centering
\includegraphics [width=6cm]{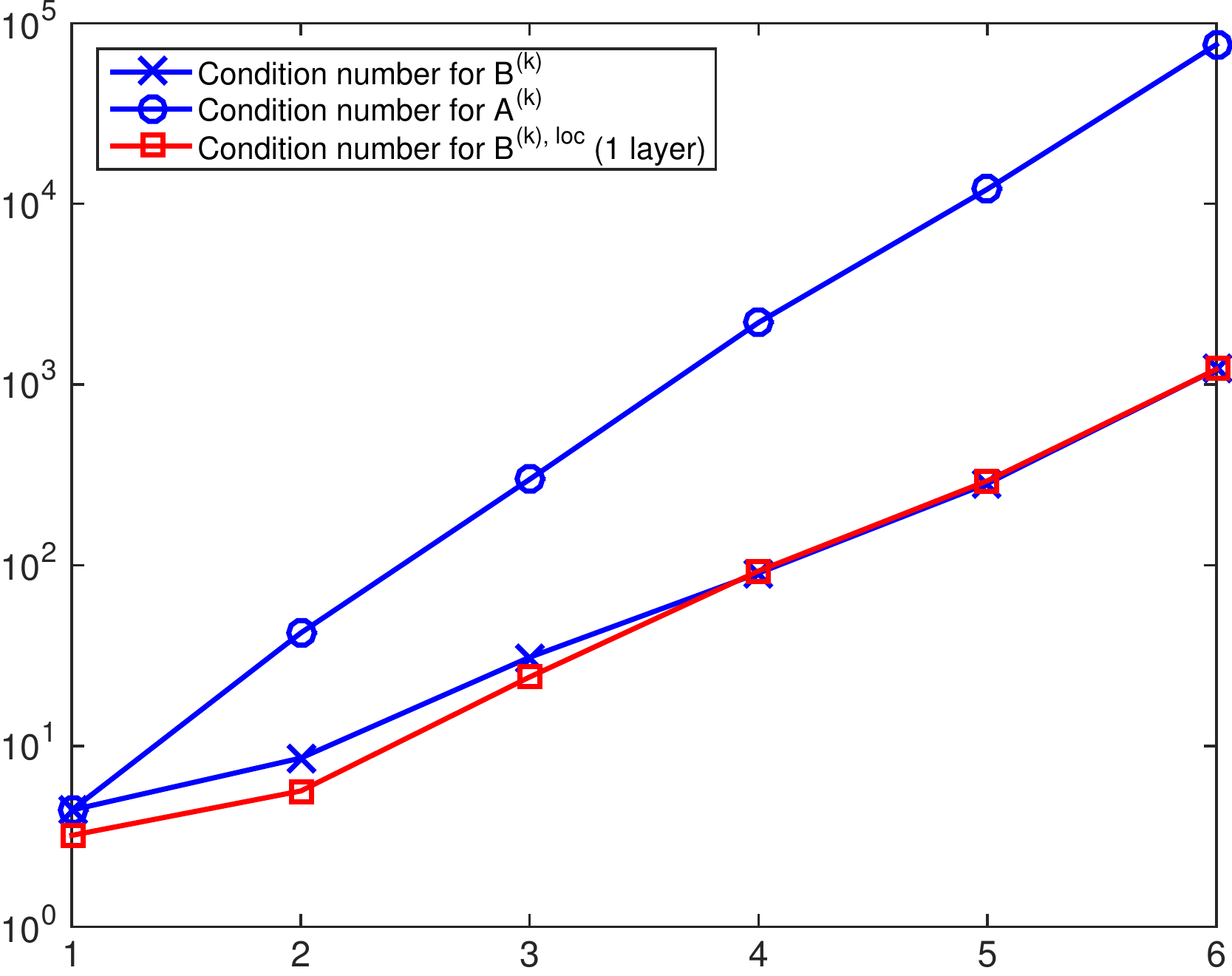}
\caption{Condition numbers of $B^{(k)}$ ($\zeta=\infty$) for $k=1,\ldots,6$ and $a(x)$ defined as in \eqref{ega}. }
\label{fig:condnum}
\end{figure}

\begin{figure}[H]
\centering
\includegraphics [width=6cm]{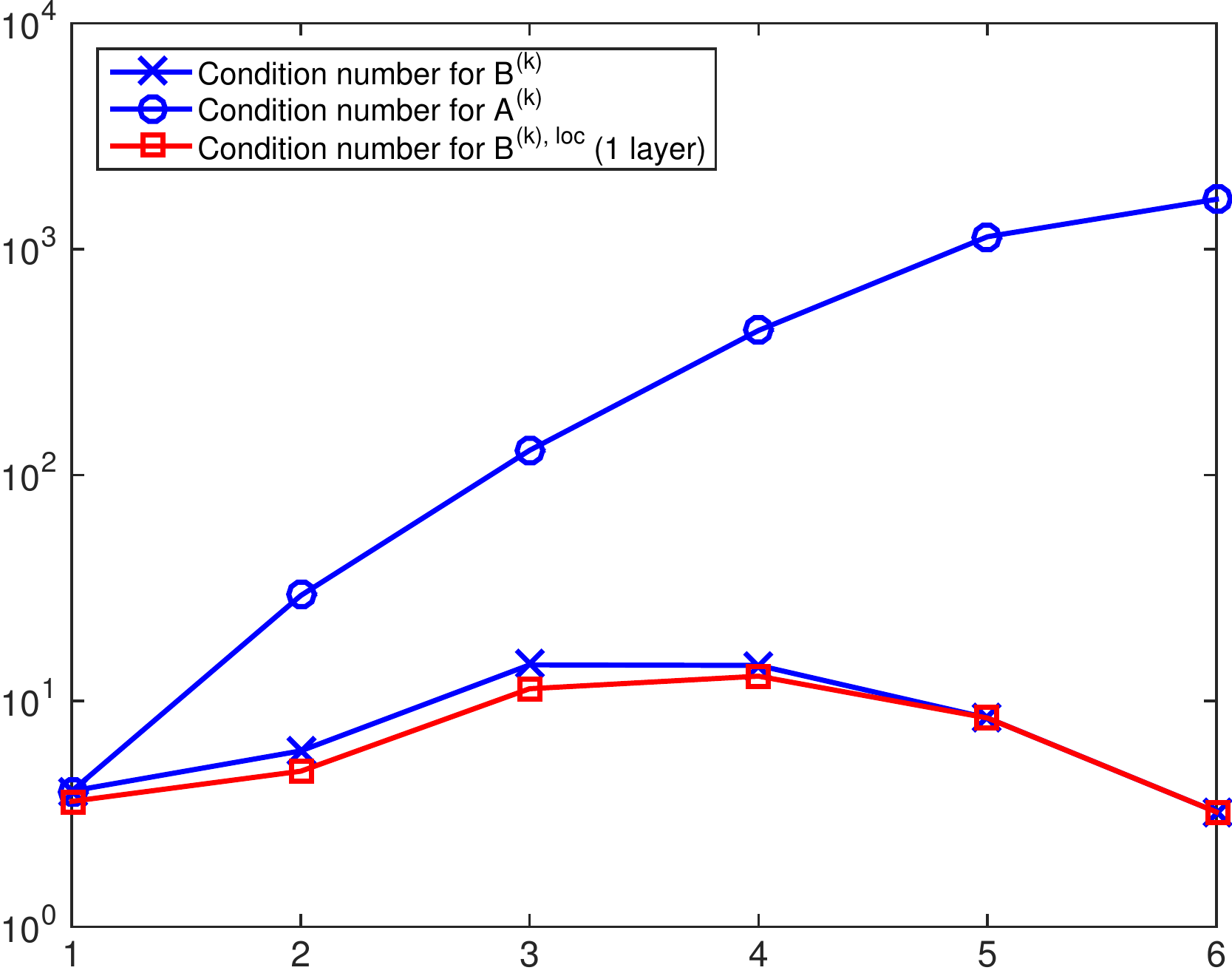}
\caption{Condition numbers of $B^{(k)}$ ($\zeta=\infty$) for $k=1,\ldots,6$ and $a(x)=I_d$ (the Laplacian). }
\label{fig:condnum_Lap}
\end{figure}

\begin{Example}
To simplify notations we will write $A^{(1),\z}$ as $B^{(1),\z}$ and omit the superscript $\z$ when $\z=\infty$. Figures \ref{fig:condnum} and \ref{fig:condnum_Lap} provide the condition numbers of $A^{(k)}$ and $B^{(k)}$ for $k=1,\ldots,6$ for $a(x)$ defined as in \eqref{ega} and $a(x)=I_d$ (the Laplacian). Note that these condition numbers do depend on the contrast of $a$.
Figures \ref{fig:eigs} and \ref{fig:eigs_Lap} illustrate the ranges of the eigenvalues of the PDE in $\V$ and in each subband $\W^{(k)}$, for $a(x)$ defined as in \eqref{ega} and $a(x)=I_d$ (the Laplacian), i.e. the figures are illustrations of the intervals $\big[\inf_{\psi \in \V} \frac{\|\psi\|_a}{\|\psi\|_{L^2(\Omega)}},\sup_{\psi \in \V} \frac{\|\psi\|_a}{\|\psi\|_{L^2(\Omega)}} \big]$
and $\big[\inf_{\psi \in \W^{(k)}} \frac{\|\psi\|_a}{\|\psi\|_{L^2(\Omega)}},\sup_{\psi \in \W^{(k)}} \frac{\|\psi\|_a}{\|\psi\|_{L^2(\Omega)}} \big]$ where we write $\|\psi\|_a:=\|\psi\|_\z$ for $\z=\infty$. Note that the eigenvalues of $B^{(k)}$ cover only subintervals of spectrum of the discretized operator, which corresponds to a multi-resolution decomposition.
\end{Example}

\begin{figure}[H]
\centering
\includegraphics [width=6cm]{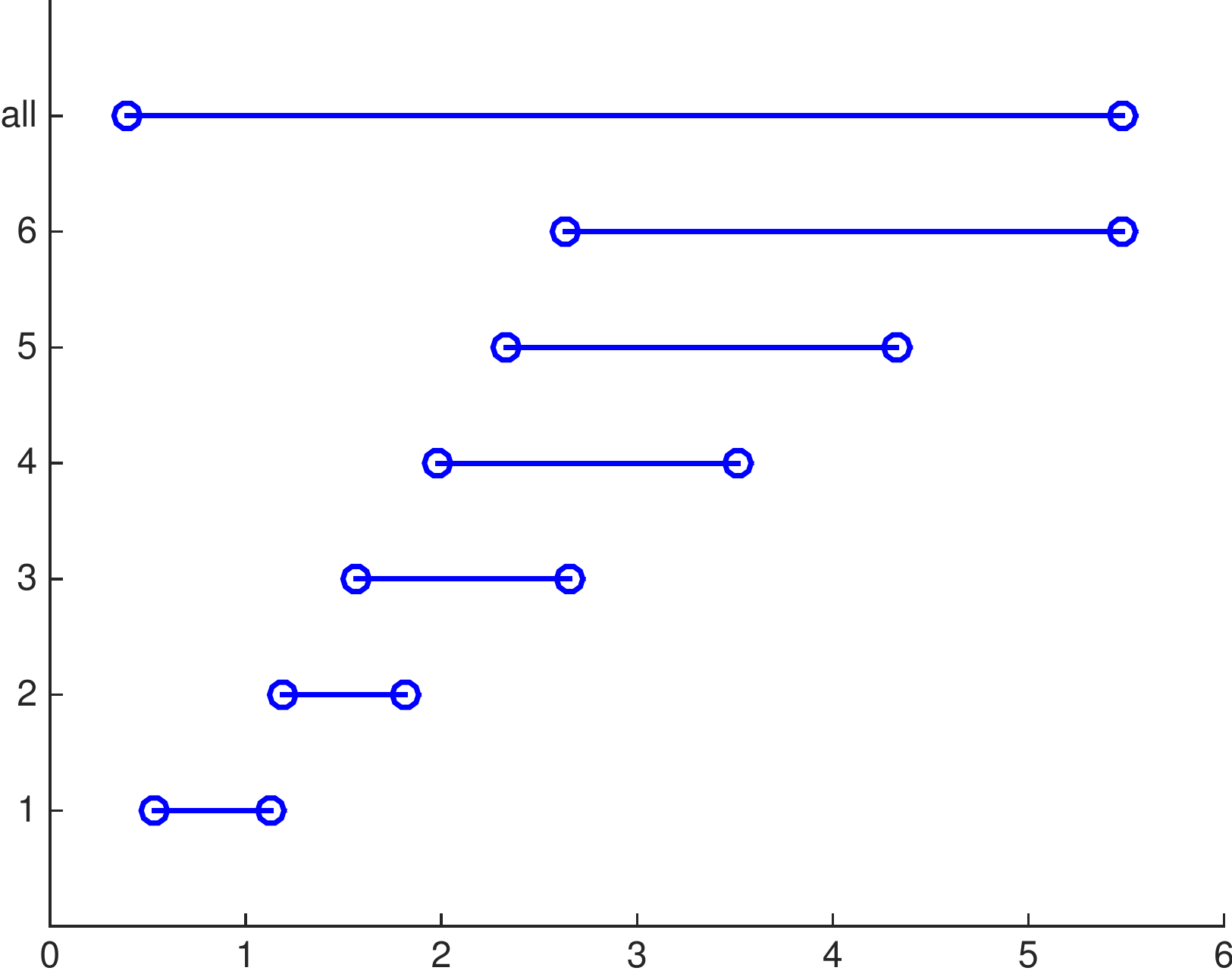}
\caption{Ranges of eigenvalues in $\V$ and $\W^{(k)}$ ($\zeta=\infty$) for $k=1,\ldots,6$ and $a(x)$ defined as in \eqref{ega}.}
\label{fig:eigs}
\end{figure}

\begin{figure}[H]
\centering
\includegraphics [width=6cm]{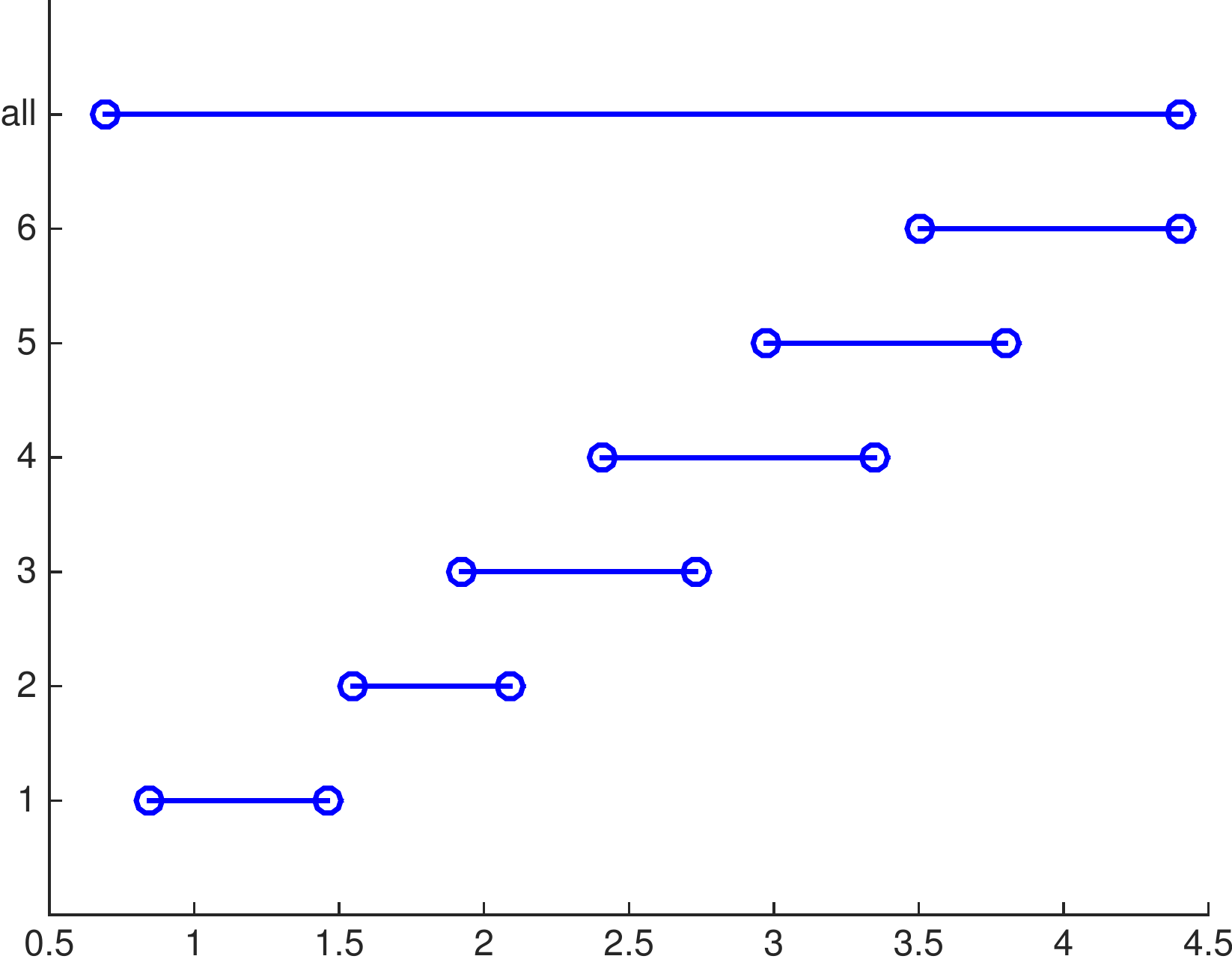}
\caption{Ranges of eigenvalues in $\V$ and $\W^{(k)}$ ($\zeta=\infty$) for $k=1,\ldots,6$ and $a(x)=I_d$ (the Laplacian)}
\label{fig:eigs_Lap}
\end{figure}

\subsection{Algorithms}

\subsubsection{The gamblet transform}
We will now describe the gamblet transform for the discrete operator obtained from the numerical approximation of \eqref{eqnscalarzeta}.
Consider the finite-element  solution of \eqref{eqnscalarzeta} over a basis $(\varphi_i)_{i\in \N}$ of fine-scale  elements.
To facilitate the presentation, assume that $(\varphi_i)_{i\in \N}$ is obtained from
$\T_\h$,  a regular fine mesh  discretization of $\Omega$ of resolution $\h$ with $0<\h \ll 1$. Let $\N$ be the set of interior nodes $z_i$ and $N=|\mathcal{N}|$ be the number of interior nodes ($N= \mathcal{O}(\h^{-d})$) of $\T_\h$. Write $(\varphi_i)_{i\in \N}$ a set of conforming nodal basis elements (of $H^1_0(\Omega)$) constructed from $\T_\h$ such that
for each $i\in \N$,  $\supp(\varphi_i)\subset B(z_i, C_0 \h)$, for $y\in \R^N$,
\begin{equation}\label{eqhhgfff65f}
\ubar{\gamma} \h^d |y|^2  \leq \|\sum_{i\in \N} y_i \varphi_i \|_{L^2(\Omega)}^2 \leq \bar{\gamma} \h^d |y|^2
\end{equation}
and
\begin{equation}\label{eqhhgfff65fip}
\|\nabla v\|_{L^2(\Omega)}\leq C_1 \h^{-1} \| v\|_{L^2(\Omega)}
\end{equation}
 for $v\in \Span\{\varphi_i\mid i\in \N\}$,
for some constants $\ubar{\gamma}, \bar{\gamma}, C_0, C_1\approx \mathcal{O}(1)$.

In addition to  properties \eqref{eqhhgfff65f} and  \eqref{eqhhgfff65fip} we assume that  $\T_\h$ is such that: (1) $\h=H^\q$ and (2) each set $\tau_i^{(\q)}$ ($i\in \I^{(\q)}$) contains one and only one interior node of $\T_\h$. Using this one to one correspondence we use the elements of $\I^{\q}$ to relabel the interior nodes $z_i$ of $\T_\h$ and their respective nodal elements $\varphi_i$.

 Write $\V:=\operatorname{span}\{\varphi_i \mid i\in \I\}$.
Given the matrices $\pi^{(k,k+1)}$ defined as in constructions \ref{constpiQQ} and \ref{eqndkjndnfj}, and the matrices $W^{(k)}$ obtained as in Definition \ref{defwk}, Algorithm  \ref{gambletsolve} computes elements of $\V$ corresponding discrete gamblets $(\psi_i^{(k),\z})_{i\in \I^{(k)}}$ and $\{\chi_i^{(k),\z}\}_{i\in \J^{(k)}}$. As in Theorem  \ref{thmgugyug0}, these discrete gamblets induce the following $\<\cdot,\cdot\>_\z$ orthogonal decomposition of the solution space, corresponding to a diagonalization of the stiffness matrix $A_{i,j}=\<\varphi_i,\varphi_j\>_z$ into blocks of  uniformly bounded condition numbers (which can be approximated by  truncated blocks thanks to the exponential decay of gamblets).
\begin{equation}\label{eqdedhhiuhed3dd}
\V=\V^{(1),\z}\oplus_\z \W^{(2),\z} \oplus_\z  \cdots \oplus_\z \W^{(\q),\z},
\end{equation}
Observe that the measurement functions $\phi_i^{(k)}$ do not appear explicitly in Algorithm  \ref{gambletsolve} (which depends only on the interpolation matrices  $\pi^{(k,k+1)}$ defined in Constructions \ref{constpiQQ} and \ref{eqndkjndnfj}).
Note that at the finest scale, level $\q$ gamblets $\psi_i^{(\q),\z}$ are simply the basis elements $\varphi_i$ used to discretize the PDE  \eqref{eqnscalarzeta} (as discussed in \cite{gamblet17}, writing $\bar{M}$ the mass matrix  $\bar{M}^\varphi_{i,j}=\int_{\Omega} \varphi_i \varphi_j  $, selecting $\psi_i^{(\q),\z}=\varphi_i$
 is equivalent to using $\phi_i^{(\q)}=\sum_{j\in \I^{(\q)}} \bar{M}_{i,j}^{-1}\varphi_j$ as level $\q$ measurement functions and defining $\phi_i^{(k)}$ via aggregation as in \eqref{eqndkjndnfj}).

\begin{algorithm}[!ht]
\caption{Gamblet Transform.}\label{gambletsolve}
\begin{algorithmic}[1]
\STATE\label{step2} For $i,j\in \I^{(\q)}$, $A^{\varphi,\z}_{i,j}=\<\varphi_i, \varphi_j\>_\z$ \COMMENT{Stiffness matrix}
\STATE\label{step3} For $i\in \I^{(\q)}$, $\psi^{(\q),\z}_i=\varphi_i$\COMMENT{Level $\q$ gamblets}
\STATE\label{step5} For $i,j\in \I^{(\q)}$, $A^{(\q),\z}_{i,j}= \< \psi_i^{(\q),\z},  \psi_j^{(\q),\z}\>_\z$   \COMMENT{$A^{(\q),\z}=A^{\varphi,\z}$}
\FOR{$k=\q$ to $2$}
\STATE\label{step9}  For $i\in \J^{(k)}$, $\chi^{(k),\z}_i=\sum_{j \in \I^{(k)}} W_{i,j}^{(k)} \psi_j^{(k),\z}$ \COMMENT{Level $k$, $\chi$ gamblets}
\STATE\label{step7} $B^{(k),\z}= W^{(k)}A^{(k),\z}W^{(k),T}$ \COMMENT{$B^{(k),\z}_{i,j}=\<\chi_i^{(k),\z},\chi_j^{(k),\z}\>_\z$}
\STATE\label{step11}  $ D^{(k,k-1),\z}= -B^{(k),\z,-1}W^{(k)}A^{(k),\z}\bar{\pi}^{(k,k-1)}$ \COMMENT{$B^{(k),\z,-1}=$matrix inverse of $B^{(k),\z}$}
\STATE\label{step12} $R^{(k-1,k),\z}=\bar{\pi}^{(k-1,k)}+D^{(k-1,k),\z}W^{(k)}$ \COMMENT{Interpolation/restriction operator}
\STATE\label{step14} For $i\in \I^{(k-1)}$, $\psi^{(k-1),\z}_i=\sum_{j \in  \I^{(k)}} R_{i,j}^{(k-1,k),\z} \psi_j^{(k),\z}$ \COMMENT{Level $k-1$, $\psi$ gamblets}
\STATE\label{step13} $A^{(k-1),\z}= R^{(k-1,k),\z}A^{(k),\z}R^{(k,k-1),\z}$ \COMMENT{$A^{(k-1),\z}_{i,j}=\<\psi_i^{(k-1),\z},\psi_j^{(k-1),\z}\>_\z$}
\ENDFOR
\end{algorithmic}
\end{algorithm}

\subsubsection{Linear solve with gamblets}

Given $g=\sum_{i\in \N} g_i \varphi_i$, Algorithm \ref{gambletsolveexlinsys} computes $u\in \Span\{\varphi_i \mid i\in \N\}$ such that,
\begin{equation}\label{eqhuiuhihyuhuiu}
 \<\varphi_j,u\>_\z=\int_{\Omega} \varphi_j g, \text{ for all } j\in \N
\end{equation}
Algorithm \ref{gambletsolveexlinsys} is exact and $u=u^{(1),\z}+(u^{(2),\z}-u^{(1),\z})+\cdots+(u^{(\q),\z}-u^{(\q-1),\z})$ obtained in Line \ref{step18} is the orthogonal decomposition of the solution $u$ of \eqref{eqhuiuhihyuhuiu} over $\V=\V^{(1),\z}\oplus_\z \W^{(2),\z} \oplus_\z  \cdots \oplus_\z \W^{(\q),\z}$.

\begin{algorithm}[!ht]
\caption{Linear solve with exact gamblets.}\label{gambletsolveexlinsys}
\begin{algorithmic}[1]
\STATE\label{step4} For $i\in \I^{(\q)}$, $g^{(\q),\z}_i=g_i$      \COMMENT{$g^{(\q),\z}_i=\int_{\Omega} \psi_i^{(\q),\z} g$ with $g=\sum_{i\in \I^{(\q)}} g_i \varphi_i$}
\FOR{$k=\q$ to $2$}
\STATE\label{step8} $w^{(k),\z}=B^{(k),\z,-1} W^{(k)} g^{(k),\z}$
\STATE\label{step10} $u^{(k),\z}-u^{(k-1),\z}=\sum_{i\in \J^{(k)}}w^{(k),\z}_i \chi^{(k),\z}_i$
\STATE\label{step15} $g^{(k-1),\z}=R^{(k-1,k),\z} g^{(k),\z}$
\ENDFOR
\STATE\label{step16} $ U^{(1),\z}=A^{(1),\z,-1}g^{(1),\z}$
\STATE\label{step17} $u^{(1),\z}=\sum_{i \in \I^{(1)}} U^{(1),\z}_i \psi^{(1),\z}_i$
\STATE\label{step18} $u=u^{(1),\z}+(u^{(2),\z}-u^{(1),\z})+\cdots+(u^{(\q),\z}-u^{(\q-1),\z})$
\end{algorithmic}
\end{algorithm}

\subsection{Fast gamblet transform}
Algorithms \ref{gambletsolve} and \ref{gambletsolveexlinsys} can be modified to operate in linear complexity. This near linear complexity is possible thanks to three main properties, (1) Nesting: level $k$ gamblets and stiffness matrices can be computed from level $k+1$ gamblets and stiffness matrices; (2) Bounded condition numbers: It follows from Theorem \ref{thmodhehiudhehd} that the linear systems involved in Algorithms \ref{gambletsolve} and \ref{gambletsolveexlinsys} have uniformly bounded condition numbers; (3) Localization: gamblets can be localized as a function of the desired accuracy.
The resulting modified algorithms are  \ref{fastgambletsolve} and \ref{gambletsolveexlinsysloc}.

\begin{algorithm}[!ht]
\caption{Localized Gamblets.}\label{fastgambletsolve}
\begin{algorithmic}[1]
\STATE\label{line3} For $i\in \I^{(\q)}$, $\psi^{(\q),\z,\loc}_i=\varphi_i$
\STATE\label{line5} For $i,j\in \I^{(\q)}$, $A^{(\q),\z,\loc}_{i,j}= \< \psi_i^{(\q),\z,\loc}, \psi_j^{(\q),\z,\loc}\>_\z$
\FOR{$k=\q$ to $2$}
\STATE\label{line7} $B^{(k),\z,\loc}= W^{(k)}A^{(k),\z,\loc}W^{(k),T}$
\STATE\label{line9}  For $i\in \J^{(k)}$, $\chi^{(k),\z,\loc}_i=\sum_{j \in \I^{(k)}} W_{i,j}^{(k)} \psi_j^{(k),\z,\loc}$
\STATE\label{line11}  $\Inv(B^{(k),\z,\loc} D^{(k,k-1),\z,\loc}=   -W^{(k)}A^{(k),\z,\loc}\bar{\pi}^{(k,k-1)},\rho_{k-1})$ \COMMENT{Def.~\ref{defb}, Thm.~\ref{tmdiscreteaccuracy}}
\STATE\label{line12} $R^{(k-1,k),\z,\loc}=\bar{\pi}^{(k-1,k)}+D^{(k-1,k),\z,\loc}W^{(k)}$ \COMMENT{Def.~\ref{defb}}
\STATE\label{line13} $A^{(k-1),\z,\loc}= R^{(k-1,k),\z,\loc}A^{(k),\z,\loc}R^{(k,k-1),\z,\loc}$
\STATE\label{line14} For $i\in \I^{(k-1)}$, $\psi^{(k-1),\z,\loc}_i=\sum_{j \in  \I^{(k)}} R_{i,j}^{(k-1,k),\z,\loc} \psi_j^{(k),\z,\loc}$
\ENDFOR
\end{algorithmic}
\end{algorithm}

Algorithm \ref{fastgambletsolve}  achieves $\mathcal{O}\big(N \ln^{3d} N\big)$ complexity in computing approximate gamblets (sufficient to achieve a  given level of accuracy). This fast algorithm is obtained by localizing/truncating the linear systems corresponding to Line ref{line11} in Algorithm \ref{gambletsolve}. The approximation error induced by these localization/truncation steps is controlled by the exponential decay of gamblets $\psi_i^{(k),\z}$ and $\chi_i^{(k),\z}$ and the uniform bound on the condition numbers of the matrices $B^{(k),\z}$ and $A^{(1),\z}$.
We define these localization/truncation steps as follows. For $k\in \{1,\ldots,q\}$ and $i\in \I^{(k)}$  define $i^{\rho}$ as the subset of indices $j\in \I^{(k)}$ whose corresponding subdomains $\tau_j^{(k)}$ are at distance at most $H_k \rho$ from $\tau_i^{(k)}$.

Note that level $\q$  gamblets  $\psi^{(\q),\z,\loc}_i$  are simply the finite-elements $\varphi_i$ used to discretize the operator. (Line \ref{line3} of Algorithm \ref{fastgambletsolve}).
Line  \ref{line11} of Algorithm \ref{fastgambletsolve} is defined  as follows.
\begin{Definition}\label{defb}
Let $k\in \{2,\ldots,\q\}$ and $B$ be the positive definite $\J^{(k)}\times \J^{(k)}$ matrix $B^{(k),\z,\loc}$ computed in Line \ref{line7} of Algorithm \ref{fastgambletsolve}.
For $i\in \I^{(k-1)}$, let $\rho=\rho_{k-1}$ and let $i^\chi$ be the subset of indices $j\in \J^{(k)}$ such that $j^{(k-1)}\in i^{\rho}$   (recall that if $j=(j_1,\ldots,j_k)\in \J^{(k)}$ then $j^{(k-1)}:=(j_1,\ldots,j_{k-1})\in \I^{(k-1)}$).
$B^{(i,\rho)}$ be the $i^{\chi}\times i^{\chi}$ matrix defined by $B^{(i,\rho)}_{l,j}=B_{l,j}$ for $l,j \in i^{\chi}$.
 Let $b^{(i,\rho)}$ be the $|i^{\chi}|$-dimensional vector defined by $b_j^{(i,\rho)}=-(W^{(k)}A^{(k),\z,\loc}\bar{\pi}^{(k,k-1)})_{j,i}$ for $j\in i^{\chi}$. Let $y^{(i,\rho)}$ be the $|i^{\chi}|$-dimensional vector solution of
$B^{(i,\rho)} y^{(i,\rho)}=b^{(i,\rho)}$.
We define  the solution $D^{(k,k-1),\z,\loc}$ of the localized linear system $\Inv(B^{(k),\z,\loc} D^{(k,k-1),\z,\loc}=   -W^{(k)}A^{(k),\z,\loc}\bar{\pi}^{(k,k-1)},\rho_{k-1})$
 as the $\J^{(k)}\times \I^{(k-1)}$ sparse matrix  given by $D^{(k,k-1),\z,\loc}_{j,i}=0$ for $j\not \in i^{\chi}$ and $D^{(k,k-1),\z,\loc}_{j,i}= y_j^{(i,\rho)}$ for $j\in i^{\chi}$. $D^{(k-1,k),\z,\loc}$ (Line \ref{line12} of Algorithm \ref{fastgambletsolve}) is then defined as the transpose of $D^{(k,k-1),\z,\loc}$.
\end{Definition}

\begin{Remark}\label{rmklocalgfast} Definition \ref{defb} (Line \ref{line7} of Algorithm \ref{fastgambletsolve}) is equivalent to localizing the computation of each gamblet $\psi_i^{(k-1),\z}$ to a subdomain of size $H_{k-1} \rho_{k-1}$, i.e., the   gamblet  $\psi^{(k-1),\z,\loc}_i$ computed in Line \ref{line14} of Algorithm \ref{fastgambletsolve} is  the solution of (1) the problem of finding $\psi$ in the affine space $\sum_{j\in \I^{(k)}}\bar{\pi}_{i,j}^{(k-1,k)}\psi_j^{(k),\z,\loc}+\Span\{\chi_j^{(k),\z,\loc} \mid j^{(k-1)}\in i^{\rho_{k-1}}\}$  such that $\psi$ is $\<\cdot,\cdot\>_\z$ orthogonal to $\Span\{\chi_j^{(k),\z,\loc} \mid j^{(k-1)}\in i^{\rho_{k-1}}\}$, and (2)
 the problem of minimizing $\|\psi\|_\z$ in $\Span\{\psi_l^{(k),\z,\loc} \mid l^{(k-1)}\in i^{\rho_{k-1}}\}$ subject to constraints $\int_{\Omega} \phi_j^{(k-1)} \psi =\delta_{i,j}$ for $j\in i^{\rho_{k-1}}$.
\end{Remark}

We will (occasionally) write $H_k$ for $H^k$ to emphasize that, as in \cite{OwhadiMultigrid:2015}, the essentially property of the sequence $H_k$  is that $H_{k}/H_{k+1}$ remains uniformly bounded away from $1$ and $\infty$.

 To simplify the presentation,  we will  write $C$ any constant that depends only on $d, \Omega, \lambda_{\min}(a), \lambda_{\max}(a), \delta, \ubar{\gamma}, \bar{\gamma}$, $C_0, C_1, \mu_{\min},\mu_{\max}, \delta$ but not on $h, \z$ nor $H$  (e.g., $2 C \z H^2\lambda_{\max}(a)$ will  be written $C \z H^2$).

 The following theorem shows that the condition numbers of the localized stiffness matrices $B^{(k),\z,\loc}$ remain uniformly bounded provided that the computation of level $k$ gamblets is localized to subdomains of size $\displaystyle H_k \ln \frac{1}{H_k}$.

\begin{Theorem}\label{tmdiscreteaccuracy0}
Let $\epsilon \in (0,1)$. Assume that
\begin{enumerate}
\item $\rho_k\geq C \big((1+\frac{1}{\ln(1/H)})\ln \frac{1}{H_k}+\ln \frac{1}{\epsilon}\big)$ for $k\in \{1,\ldots,\q\}$.
\item For $k\in \{2,\ldots,\q\}$ and each $i\in \I^{(k-1)}$, the localized linear system $B^{(i,\rho)} y=b$ of Definition \ref{defb} and Line \ref{line11} of Algorithm \ref{fastgambletsolve} is solved up to accuracy
$|y-y^{\app}|_{B^{(i,\rho)}} \leq C^{-1} H^{3-k+kd/2}\epsilon/k^2$ (using the notation  $|e|_A^2:=e^T A e$, and writing $y^{\app}$ the approximation of $y$).
\end{enumerate}
 Then it holds true that  $\operatorname{Cond}(A^{(1),\z,\loc})\leq C H^{-2}$ and
 for $k\in \{2,\ldots,\q\}$,
$\operatorname{Cond}(B^{(k),\z,\loc})\leq C H^{-2}$. Furthermore for
 $k\in \{1,\ldots,\q\}$ and  $\z=\infty$, $\displaystyle \frac{1}{C} \leq \lambda_{\min}(A^{(k),\z,\loc})$  and $\lambda_{\max}(A^{(k),\z,\loc})\leq C H^{-2k}$,
and for $k\in \{2,\ldots,\q\}$ and $\z=\infty$, $\displaystyle\frac{1}{C} H^{-2(k-1) }   \leq \lambda_{\min}(B^{(k),\z,\loc})$  and $\lambda_{\max}(B^{(k),\z,\loc})\leq  C H^{-2k }$.  Additionally  the functions $(\psi_i^{(1),\z,\loc})_{i\in \I^{(1)}}$ and $(\chi_i^{(k),\z,\loc})_{k\in \{2,\ldots,\q\}, i\in \J^{(k)}}$ are linearly independent and form a basis of $\V$.
\end{Theorem}

We now present Algorithm \ref{gambletsolveexlinsysloc}, which computes an approximation of the solution of \eqref{eqhuiuhihyuhuiu} using localized gamblets (up to $\epsilon$ accuracy in $H^1_0(\Omega)$-norm).

\begin{algorithm}[!ht]
\caption{Linear solve with localized gamblets.}\label{gambletsolveexlinsysloc}
\begin{algorithmic}[1]
\STATE\label{linstep4} For $i\in \I^{(\q)}$, $g^{(\q),\z,\loc}_i=g_i$      \COMMENT{$g=\sum_{i\in \I^{(\q)}} g_i \varphi_i$}
\FOR{$k=\q$ to $2$}
\STATE\label{linstep8} $w^{(k),\z,\loc}=B^{(k),\z,\loc,-1} W^{(k)} g^{(k),\z,\loc}$
\STATE\label{linstep10} $u^{(k),\z,\loc}-u^{(k-1),\z,\loc}=\sum_{i\in \J^{(k)}}w^{(k),\z,\loc}_i \chi^{(k),\z,\loc}_i$
\STATE\label{linstep15} $g^{(k-1),\z,\loc}=R^{(k-1,k),\z,\loc} g^{(k),\z,\loc}$
\ENDFOR
\STATE\label{linstep16} $ U^{(1),\z,\loc}=A^{(1),\z,\loc,-1}g^{(1),\z,\loc}$
\STATE\label{linstep17} $u^{(1),\z,\loc}=\sum_{i \in \I^{(1)}} U^{(1),\z,\loc}_i \psi^{(1),\z,\loc}_i$
\STATE\label{linstep18} $u^{\loc}=u^{(1),\z,\loc}+(u^{(2),\z,\loc}-u^{(1),\z,\loc})+\cdots+(u^{(\q),\z,\loc}-u^{(\q-1),\z,\loc})$
\end{algorithmic}
\end{algorithm}

\begin{Theorem}\label{tmdiscrete}
Let $u$ be the solution of the discrete system \eqref{eqhuiuhihyuhuiu}. Let $u^{(1),\z,\loc}$, $u^{(k),\z,\loc}-u^{(k-1),\z,\loc}$, $u^{\loc}$, $A^{(k),\z,\loc}$ and $B^{(k),\z,\loc}$ be the outputs of algorithms \ref{fastgambletsolve} and \ref{gambletsolveexlinsysloc}. Let $u^{(1),\z}$ and $u^{(k),\z}-u^{(k-1),\z}$  be the outputs of Algorithm \ref{gambletsolveexlinsys}.
For $k\in \{2,\ldots,\q\}$, write $u^{(k),\z,\loc}:=u^{(1),\z,\loc}+\sum_{j=2}^k (u^{(j),\z,\loc}-u^{(j-1),\z,\loc})$.
Let $\epsilon \in (0,1)$, it holds true that if $\rho_k\geq C \big((1+\frac{1}{\ln(1/H)})\ln \frac{1}{H_k}+\ln \frac{1}{\epsilon}\big)$ for $k\in \{1,\ldots,\q\}$ then
\begin{enumerate}
\item for $k\in \{1,\ldots,\q-1\}$ we have
$\|u^{(k),\z} - u^{(k),\z,\loc}\|_{\z} \leq   \epsilon \|g\|_{H^{-1}(\Omega)}$ and $\|u^{(k),\z} - u^{(k),\z,\loc}\|_{\z} \leq  C (H_k+\epsilon) \|g\|_{L^2(\Omega)}$
\item $\|u^{(k),\z}-u^{(k-1),\z}-(u^{(k),\z,\loc}-u^{(k-1),\z,\loc})\|_{\z} \leq \frac{\epsilon}{2 k^2} \|g\|_{H^{-1}(\Omega)}$.
\item Furthermore, $\|u - u^{\loc}\|_{\z} \leq  \epsilon \|g\|_{H^{-1}(\Omega)}$.
\end{enumerate}
\end{Theorem}

\begin{Theorem}\label{tmdiscreteaccuracy}
The results of Theorem \ref{tmdiscrete} remain true if
\begin{enumerate}
\item $\rho_k\geq C \big((1+\frac{1}{\ln(1/H)})\ln \frac{1}{H_k}+\ln \frac{1}{\epsilon}\big)$ for $k\in \{1,\ldots,\q\}$
\item For $k\in \{2,\ldots,\q\}$ and each $i\in \I^{(k-1)}$, the localized linear system $B^{(i,\rho)} y=b$ of Definition \ref{defb} and Line \ref{line11} of Algorithm \ref{fastgambletsolve} is solved up to accuracy $|y-y^{\app}|_{B^{(i,\rho)}} \leq C^{-1} H^{3-k+kd/2}\epsilon/k^2$ (using the notation $|e|_A^2:=e^T A e$, and writing $y^{\app}$ the approximation of $y$).
\item For $k\in \{2,\ldots,\q\}$ the linear system $ B^{(k),\z,\loc}y=  W^{(k)} g^{(k),\z,\loc}$ of Line \ref{linstep8} of  Algorithm \ref{gambletsolveexlinsysloc} is solved up to accuracy $|y-y^{\app}|_{B^{(k),\z,\loc}}\leq  \epsilon \|g\|_{H^{-1}(\Omega)}/(2\q)$.
\end{enumerate}
\end{Theorem}

Observe that theorems \ref{tmdiscrete} and  \ref{tmdiscreteaccuracy} imply that (1) the complexity of Algorithm \ref{fastgambletsolve} is $\mathcal{O}\big(N \big(\ln \max (\frac{1}{\epsilon},N^\frac{1}{d})\big)^{3d}\big)$ (2) the complexity of Algorithm \ref{gambletsolveexlinsysloc} is $\mathcal{O}\big(N \big(\ln \max (\frac{1}{\epsilon},N^\frac{1}{d})\big)^{d} \ln \frac{1}{\epsilon}\big)$. Therefore if $\epsilon$ corresponds to a grid size accuracy $N^{-1/d}$ (in $H^1$-norm) then the complexity of Algorithm \ref{fastgambletsolve}  is $\mathcal{O}(N \ln^{3d} N)$ and that of Algorithm \ref{gambletsolveexlinsysloc} is $\mathcal{O}(N \ln^{d+1} N)$. In Figure \ref{fig:complexity}, we show the elapsed time of localized Gamblet transform and localized gamblet linear solve for fixed $\rho_k = 3$ with respect to the degrees of freedom $N$. Although our implementation is in Matlab and it is not optimal in terms of efficiency, we can still observe close to linear complexity. 

\begin{figure}[H]
\begin{center}
\includegraphics [scale=0.3]{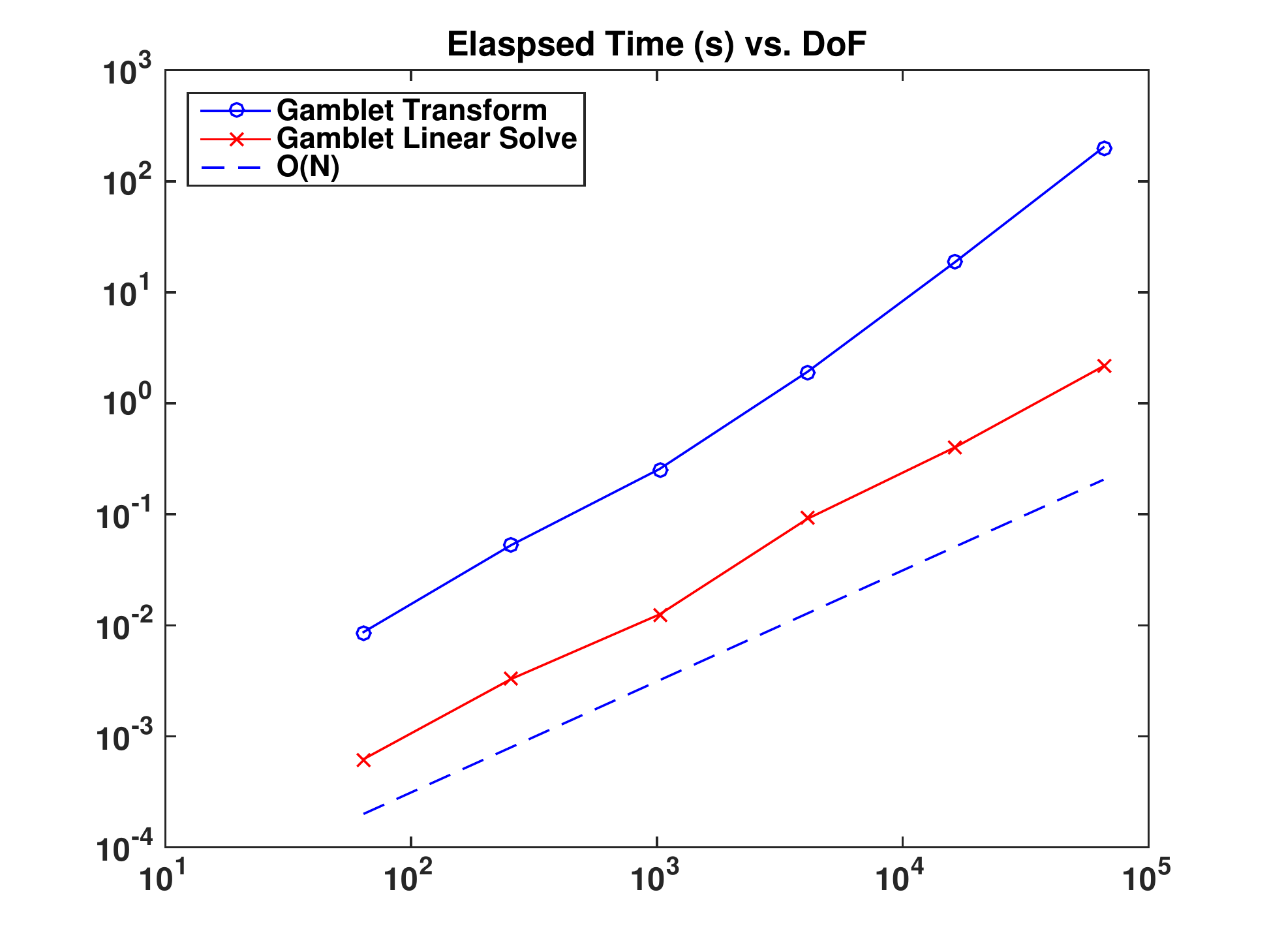}
\end{center}
\caption{Elapsed time (sec) vs. DoF ($N$) for localized gamblet transform and localized gamblet linear solve, $\rho_k=3$ for all $k$.}
\label{fig:complexity}
\end{figure}

\section{The wave PDE with rough coefficients}
\label{secwavepde}

Consider the following prototypical wave PDE with rough coefficients,
\begin{equation}\label{eqn:scalarwave}
\begin{cases}
    \mu(x)\partial_t^2 u(x,t) -\diiv \big(a(x)  \nabla u(x,t)\big)=g(x,t) \quad  x \in \Omega; \\
    u(x,0)=u_0(x) \quad \text{on}\quad \Omega,\\
    \partial_t u(x,0)=v_0(x) \quad \text{on}\quad \Omega,\\
    u(x,t)=0 \quad \text{on}\quad \partial\Omega\times [0,T]
    \end{cases}
\end{equation}
where the domain $\Omega$ and the coefficients $\mu(x)$ and $a(x)$ are as in \eqref{eqnscalarzeta}
(i.e. in $L^\infty(\Omega)$ and satisfy \eqref{eqaineq} and \eqref{eqineqmu}).

Let  $(\varphi_i)_{i\in \N}$ be a finite-dimensional (finite-element) basis of $H^1_0(\Omega)$ and write $\V=\Span\{\varphi_i\mid i\in \N\}$. Let $\tilde{u}(x,t)=\sum_{i\in \N} q_i(t)\varphi_i(x)$ be the finite-element solution of \eqref{eqn:scalarwave} in $\V$  and assume that the elements $(\varphi_i)_{i\in \N}$ are chosen to satisfy \eqref{eqhhgfff65f} and \eqref{eqhhgfff65fip} and so that $(\tilde{u}(x,t))_{0\leq t \leq T}$ is a \textit{good enough} approximation of the solution $(u(x,t))_{0\leq t \leq T}$  of \eqref{eqn:scalarwave}. Let $N=|\N|$ be the cardinal of $\N$ (and the dimension of $\V$). Let $M$ and $K$ be the $\N\times \N$ mass and stiffness matrices $M_{i,j}=\int_{\Omega}\varphi_i \varphi_j \mu$ and $K_{i,j}=\int_{\Omega}(\nabla \varphi_i)^T a \nabla \varphi_j$.

Recall that the vector $q\in \R^\db$ is the solution of the forced Hamiltonian system
\begin{equation}\label{eqhdsgjhdgdfi}
\begin{cases}
\dot{q}&=M^{-1}p\\
\dot{p}&=-K q+f
\end{cases}
\end{equation}
where for $i\in \N$, $f_i(t):=\int_{\Omega}\varphi_i g(x,t)$, $q_0=q(0)$ corresponds to the coefficients of $u_0$ in the $\varphi_i$ basis and $p_0=p(0)$ is the $\N$-vector defined by $p_{0,i}:=\int_{\Omega} \varphi_i v_0(x) \mu$.

\subsection{Implicit midpoint rule}
\label{sec:implicitmidpoint}
A popular time-discretization of \eqref{eqhdsgjhdgdfi} is the implicit midpoint rule \cite{HairerLubichWanner06}, which is unconditionally stable (A-stable, i.e. its region of absolute stability includes the entire complex half-plane with negative real part), symplectic, symmetric (time-reversible) and preserves quadratic invariants  exactly \cite{HairerLubichWanner06}.
For example, when $f=0$, (exactly preserved) quadratic invariants of \eqref{eqhdsgjhdgdfi} include the total energy ($E=\frac{1}{2}p^T M^{-1}p+\frac{1}{2}q^T K q$) and the energy of each vibration mode ($E_i=|Q_i^T M^{-1}p|^2\frac{1}{2}Q_i^T M Q_i+| Q_i^T q|^2\frac{1}{2}Q_i^T K Q_i$ with $\lambda_i M Q_i=K Q_i$).

Writing $q_n$ the numerical approximation of $q(n \dt)$, $p_n$ the numerical approximation of $p(n \dt)$, and $f_n:=f(n \dt)$, recall that the implicit midpoint time discretization of \eqref{eqhdsgjhdgdfi} is
\begin{equation}\label{implicitmidpoint}
\begin{cases}
q_{n+1}&=q_n+\dt M^{-1}\frac{p_n+p_{n+1}}{2}\\
p_{n+1}&=p_n- \dt K \frac{q_n+q_{n+1}}{2}+\dt f_{n+\frac{1}{2}}
\end{cases}
\end{equation}

Note that  \eqref{implicitmidpoint}  can be written
\begin{equation}\label{eqimeximpl}
\begin{cases}
(M+\frac{(\dt)^2}{4} K) q_{n+1} &=(M-\frac{(\dt)^2}{4} K)q_n+\dt p_n+\dt^2 \frac{f_{n+\frac{1}{2}}}{2} \\
p_{n+1}&=p_n- \dt K \frac{q_n+q_{n+1}}{2}+\dt f_{n+\frac{1}{2}}
\end{cases}
\end{equation}
Let $u_n(x):=\sum_{i\in \N} q_{n,i} \varphi_i(x)$
and
$v_n(x):= \sum_{i\in \N} (M^{-1} p_n)_i  \varphi_i$
be the corresponding approximations of $u(x,n\dt)$ and $\partial_t u(x,n\dt)$.
 Observe that $p_{n,i}=\int_{\Omega}\varphi_i v_n \mu$ and
 solving \eqref{eqimeximpl}  is equivalent to obtaining the finite element solution (in $\V$) of
\begin{equation}\label{eqimeximplzlocfem}
\begin{cases}
\frac{4}{\dt^2}\mu u_{n+1} -\diiv \big(a  \nabla u_{n+1}\big) &=\frac{4}{(\dt)^2} \mu u_{n}+\diiv \big(a  \nabla u_{n}\big)+\frac{4}{\dt} \mu v_n+2 g_{n+\frac{1}{2}} \\
\mu v_{n+1}&=\mu v_n^{\z,\loc}+\dt \diiv \big(a  \nabla\frac{u_n+u_{n+1}}{2}\big)+\dt g_{n+\frac{1}{2}}
\end{cases}
\end{equation}
with $g_n(x):=g(x,n\dt)$.

\subsection{Acceleration of the midpoint rule with $\z$-gamblets}\label{subsecaccmidp}
To achieve near linear complexity in the implementation of the midpoint rule we will perform the inversion of the implicit system in \eqref{eqimeximpl} or \eqref{eqimeximplzlocfem} in a localized $\z$-gamblet basis with $\z=\dt$. Write $(q_n^{\app},p_n^{\app})$ the output of the corresponding Algorithm
 \ref{wavequalgfemaltmix}.

\begin{algorithm}[!ht]
\caption{Implicit midpoint rule with localized gamblets.}\label{wavequalgfemaltmix}
\begin{algorithmic}[1]
\STATE Set $\z=\dt$ and  $\epsilon$ as in Theorem \ref{thmdiedhduw}.
\STATE Compute $\chi_i^{(k),\z,\loc}, \psi_i^{(k),\z,\loc}, B^{(k),\z,\loc}, R^{(k,k-1),\z,\loc}$ with Algorithm \ref{fastgambletsolve}.
\STATE  $q_0^{\app}:=q_0$ and $p_{0}^{\app}:=p_0$.
\FOR{$n=0$ to $T/\dt-1$}
\STATE\label{linfemwavegx} Solve $(M+\frac{(\dt)^2}{4} K) q_{n+1}^{\app} =(M-\frac{(\dt)^2}{4} K)q_n^{\app}+\dt p_n^{\app}+\frac{\dt^2}{2} f_{n+\frac{1}{2}}$ with Algorithm \ref{gambletsolveexlinsysloc}  \COMMENT{$f_{n+\frac{1}{2},i}:= \int_{\Omega}\varphi_i g(x,(n+\frac{1}{2})\dt)$}
\STATE\label{linfemwavegx2} $p_{n+1}^{\app} =p_n^{\app}- \dt K \frac{q_n^{\app}+q_{n+1}^{\app}}{2}+\dt f_{n+\frac{1}{2}}^{\app}$.
\ENDFOR
\end{algorithmic}
\end{algorithm}

Write $u^\app_n:= \sum_{i\in \N} q^{\app}_{n,i} \varphi_i$ and $v^\app_n:=\sum_{i\in \N} (M^{-1} p_n^{\app})_i  \varphi_i$. Observe that
Line \ref{linfemwavegx} of Algorithm  \ref{wavequalgfemaltmix} is  equivalent to solving
\begin{equation}\label{eqequiua}
\frac{4}{\dt^2}\mu u_{n+1}^\app -\diiv \big(a  \nabla u_{n+1}^\app\big) =\frac{4}{(\dt)^2} \mu u_{n}^\app+\diiv \big(a  \nabla u_{n}^\app\big)+\frac{4}{\dt} \mu v_n^\app+2 g_{n+\frac{1}{2}}
\end{equation}
in $\V$ with Algorithm \ref{gambletsolveexlinsysloc}. Note that $u_{n+1}^\app\in \V$ and the equality \eqref{eqequiua} is defined in the finite-element sense after integration against $\varphi\in \V$. Note also that (1)  $p_{n,i}^{\app}=\int_{\Omega} \varphi_i v_n^\app \mu$,
(2) $v^\app_n$ is an approximation of $\partial_t u(x,n\dt)$, (3)  $(q_n^{\app})^T K q_n^{\app}=\int_{\Omega}(\nabla u_n^\app)^T a \nabla u_n^\app$,
(4) $(p_n^{\app})^T M^{-1} p_n^{\app}=\int_{\Omega} (v_n^\app)^2 \mu$.

The following theorem, whose proof is given in Subsection \ref{secproof1} of the appendix,  provides a priori error estimates on the accuracy of Algorithm \ref{wavequalgfemaltmix} using the exact midpoint rule solution as a reference. We assume for the clarity of those error estimates, without loss of generality, that $\Delta t \leq 1$ and write $\|g\|_{L^\infty(0,T,L^2(\Omega))})$ the essential supremum of $\|g(\cdot,t)\|_{L^2(\Omega)}$ over $t\in (0,T)$.
\begin{Theorem}\label{thmdiedhduw}
Let $u^\app_n$ and $v^\app_n$ be the output of Algorithm \ref{wavequalgfemaltmix} and
$u_n$, $v_n$ be the solution of the implicit midpoint time discretization of
\eqref{eqimeximplzlocfem} (or equivalently \eqref{eqimeximpl}) with time-step $\dt$.
If $\epsilon$ in Algorithm \ref{wavequalgfemaltmix} satisfies $\displaystyle\epsilon \leq C^{-1}\frac{1}{T}\dt^3 h (\dt+h) $, then for $n\dt \leq T$ we have
\begin{equation}
\begin{split}
\|u_n-u^\app_n\|_{H^1_0(\Omega)}&+\|v_n-v^\app_n\|_{L^2(\Omega)}\leq C (\Delta t)^2 \\&
\big(  \|u_0\|_{H^1_0(\Omega)}+\|v_0\|_{L^2(\Omega)}+ T   \|g\|_{L^\infty(0,T,L^2(\Omega))})\big)
\end{split}
\end{equation}
\end{Theorem}

\subsection{Numerical experiments}
Let $a(x)$ be defined as in Example \ref{egadef}, $g(x,t) = \sin(2\pi(t+x_1))\cos(2\pi(t+x_2))$, $u(x,0) = 0$, and $u_t(x,0) = \sin(2\pi x_1)\cos(2\pi x_2)$. The reference solution is computed using bilinear finite-elements $\{\varphi_i\mid i \in \N\}$, and \textit{Matlab} built-in integrator \textit{ode15s} with time step $dt = 1/1280$. We test the performance/accuracy of exact gamblets and localized gamblets adapted to the implicit midpoint rule. We compute  numerical solutions up to time $T = 1$. Figure \ref{fig:sol:wave} shows the reference solution, the numerical solution and the error of the numerical solution. The numerical solution is computed using localized gamblet for 2 stages Gauss-Legendre scheme (which will be introduced in Section \S~\ref{sec:gl2}) with 2 layers, namely, we take $nl:=\rho_k=2$ in Algorithm \ref{fastgambletsolve} for $k=1,\ldots,\q$ and we will keep using the notation $nl$ (number of layers) for the values of $\rho_k$, and $\dt = 0.1$.
\begin{figure}[H]
\begin{center}
\includegraphics [width=14cm, height=4cm]{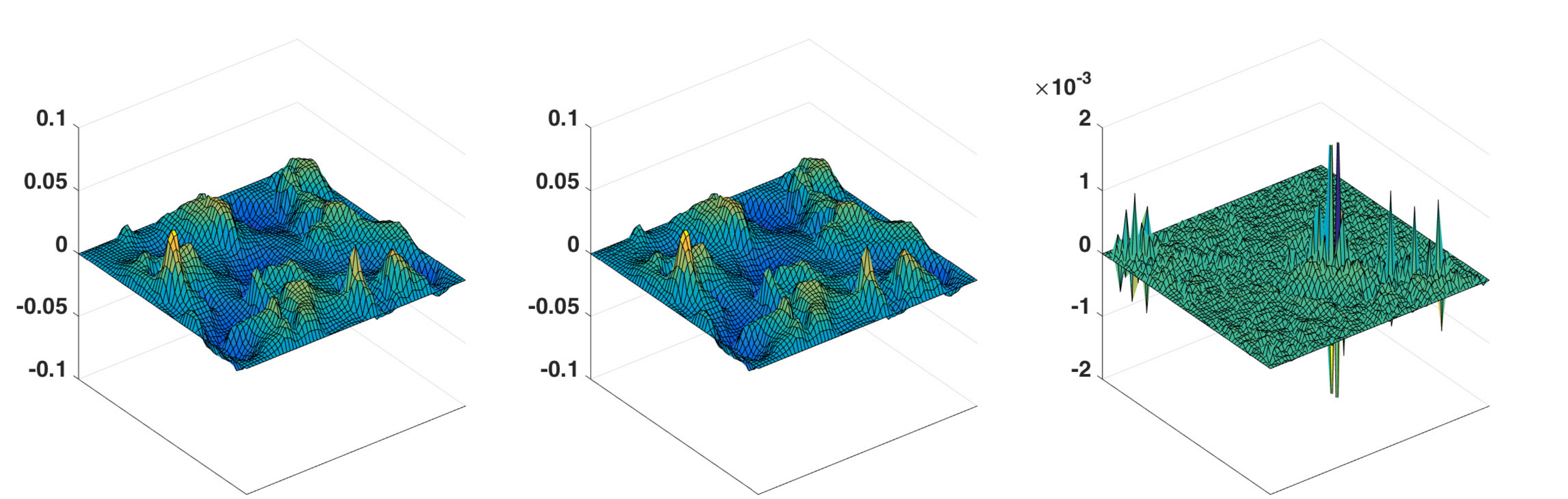}
\end{center}
\caption{Solutions at $T=1$. Left: Reference solution with the $\{\varphi_i\mid i \in \N\}$ basis, $dt = 1/1280$. Middle: numerical solution using localized gamblets with 3 layers for 2 stages Gauss-Legendre scheme, $\dt=0.1$. Right: The error of numerical solution.}
\label{fig:sol:wave}
\end{figure}
Figure \ref{fig:energy:wave} shows the relative error of the energy ($E=\frac{1}{2}p^T M^{-1}p+\frac{1}{2}q^T K q$) of gamblet solutions with respect to time and localization. The error for the gamblet solutions appears to be stable with respect to time for both the implicit midpoint scheme and the 2 stages Gauss-Legendre scheme if $nl>1$. When $nl\geq 3$ for implicit midpoint and $nl\geq 4$ for 2 stages Gauss Legendre, the localized gamblet solutions are almost as
accurate as the exact gamblet solutions.
\begin{figure}[H]
\centering
\includegraphics [width=6cm]{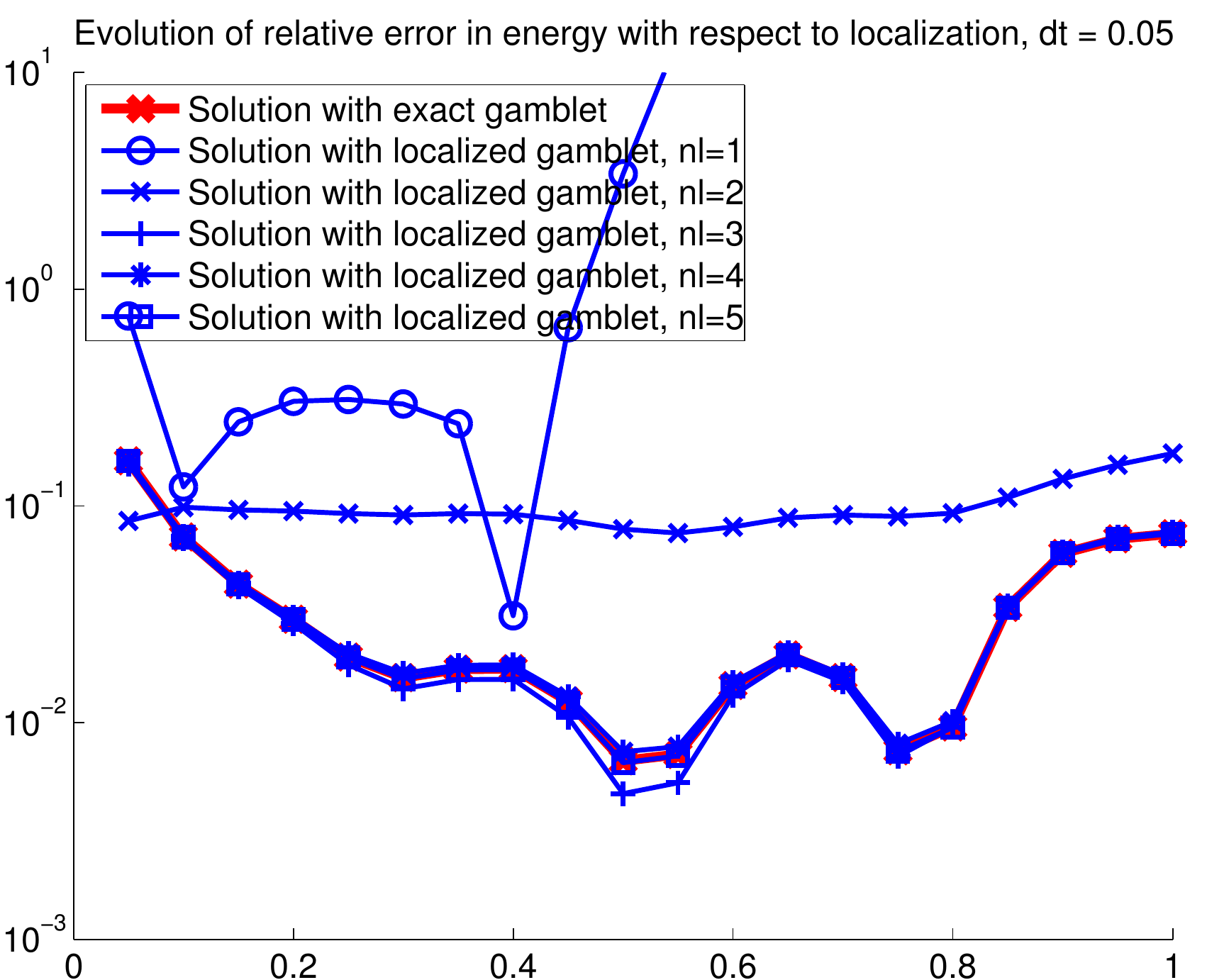}\quad
\includegraphics [width=6cm]{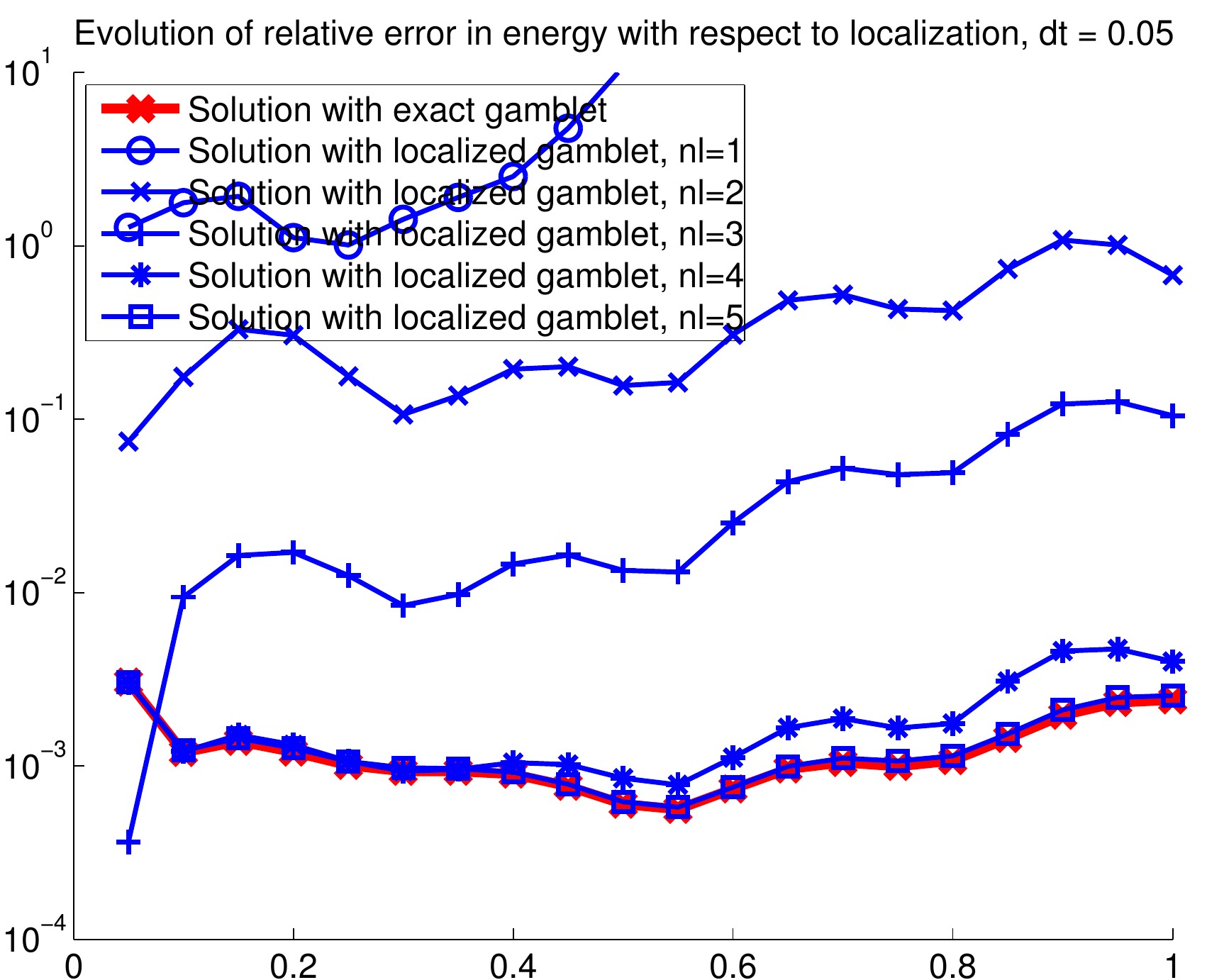}
\caption{Evolution of the relative error of the energy w.r.t localization: Left, implicit midpoint scheme; Right, 2 stages Gauss-Legendre scheme.}
\label{fig:energy:wave}
\end{figure}

Figure \ref{figure:band_i1wave_nl3} and Figure \ref{figure:band_i2wave_nl3} compare the components of the reference solution and the localized gamblet solution in each subband $\W^{(k)}$, with implicit midpoint scheme and 2 stages Gauss-Legendre scheme respectively. A phase shift error can be observed at high frequencies, and 2 stages Gauss-Legendre scheme has smaller phase error, even after localization.
\begin{figure}[H]
\centering
\includegraphics [width=10cm, height=6cm]{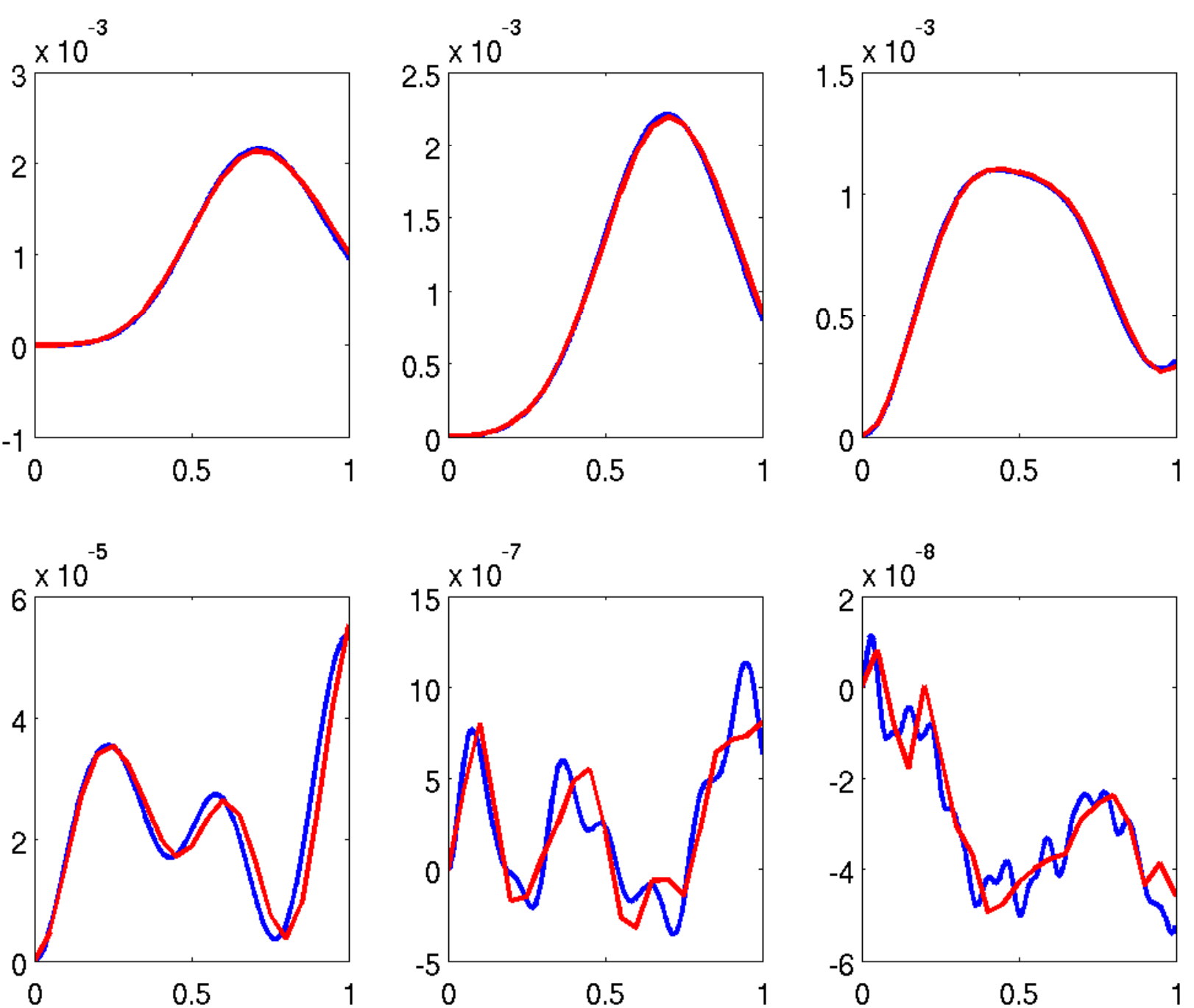}
\caption{Evolution of $\chi_1^{(k)}$ component in subband $\W^{(k)}$, $k = 1, \cdots, 6$, with localization parameter $nl=3$, using implicit midpoint scheme. The blue curve is for the reference solution, and the red curve is for the localized gamblet solution.}
\label{figure:band_i1wave_nl3}
\end{figure}
\begin{Remark}
Figures \ref{figure:band_i1wave_nl3} and \ref{figure:band_i2wave_nl3} show that the multiresolution decomposition of the solution space performed by gamblets is analogous to a eigensubspace decomposition: the coefficients of the solution in $\W^{(k),\z}$ evolve slowly for $k$ small and fast for $k$ large. Furthermore, these coefficients are robust to perturbations in initial conditions and dispersion errors for $k$ small and sensitive for $k$ large. Therefore gamblets decompose the the solution into components characterized by a hierarchy of levels of robustness.
Furthermore, as done with wavelets \cite{cohen2002adaptive},
gamblet refinement could  be used as an alternative to adaptive mesh refinement near singularities.
\end{Remark}

\begin{figure}[H]
\centering
\includegraphics [width=10cm, height=6cm]{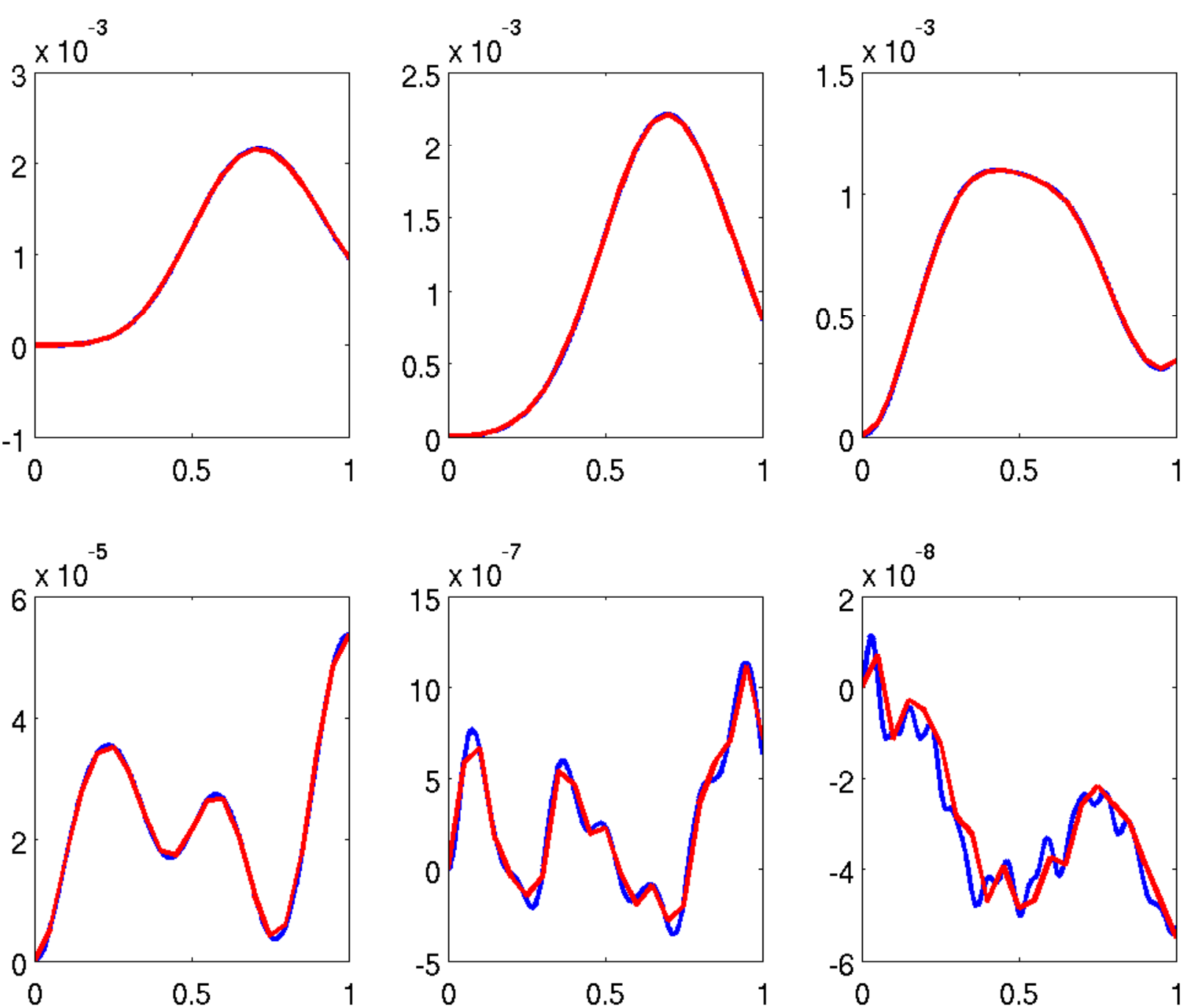}
\caption{Evolution of $\chi_1^{(k)}$ component in subband $\W^{(k)}$, $k = 1, \cdots, 6$, with localization parameter $nl=3$, using 2 stages Gauss Legendre scheme. The blue curve is for the reference solution, and the red curve is for the localized gamblet solution.}
\label{figure:band_i2wave_nl3}
\end{figure}

 Figure \ref{fig:h1errorwrttime:wave}  shows the $H^1$ and $L^2$ errors at $T=1$ with respect to localization for  localized gamblet solutions, with time step $\dt = 0.025$. This shows that for fixed spatial resolution, the errors get saturated after reaching a critical $nl$. For example, we can choose $nl=3$ as this critical value for the results shown in Figure \ref{fig:h1errorwrttime:wave} .
\begin{figure}[H]
    \begin{subfigure}[b]{0.5\textwidth}
        \includegraphics[width=\textwidth]{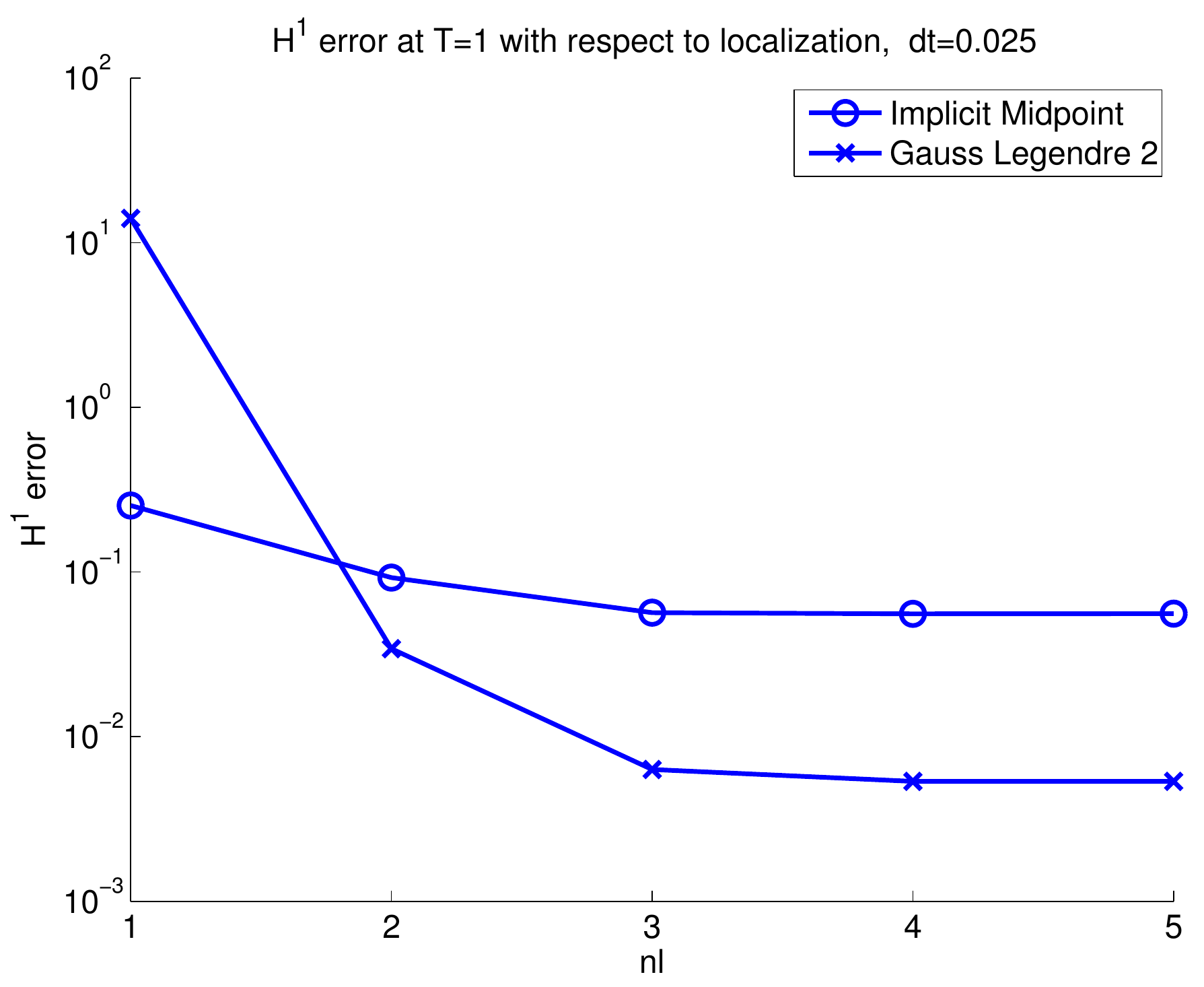}
    \end{subfigure}
    \begin{subfigure}[b]{0.5\textwidth}
        \includegraphics[width=\textwidth]{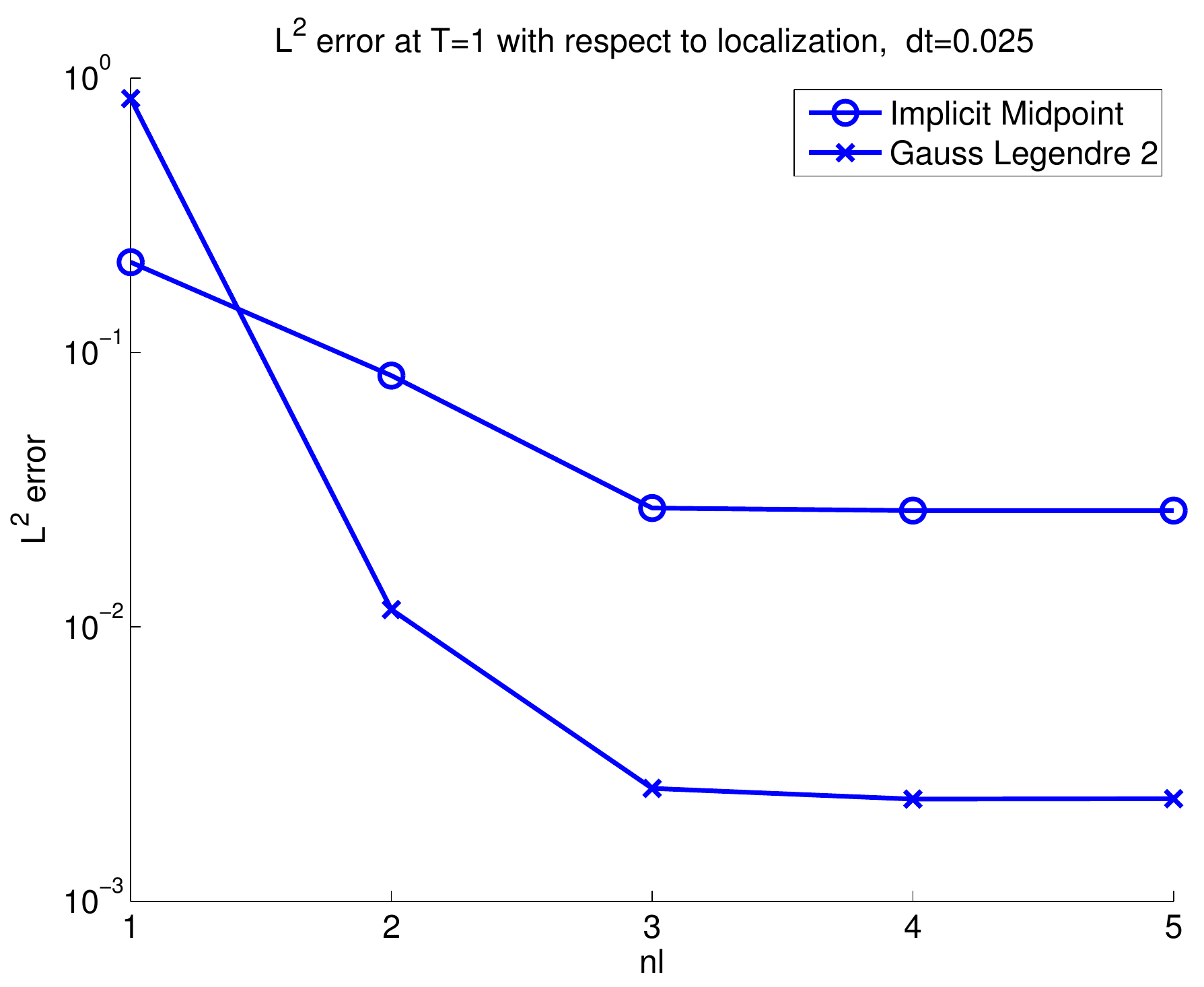}
    \end{subfigure}
    \caption{Error for localized gamblet solutions at $T=1$: Left, $H^1$ error w.r.t localization; Right, $L^2$ error w.r.t. localization.}\label{fig:h1errorwrttime:wave}
\end{figure}
 Figure \ref{fig:h1errorat1:wave} shows the  $H^1$ and $L^2$ errors  for the solution with exact gamblets and localized gamblets ($nl=2$ or $3$) using implicit midpoint scheme and 2 stages Gauss Legendre scheme at time $T=1$, and time steps  $\dt = 1/10, 1/20, 1/40, 1/80, 1/160, 1/320$.
 \begin{figure}[H]
    \begin{subfigure}[b]{0.5\textwidth}
        \includegraphics[width=\textwidth]{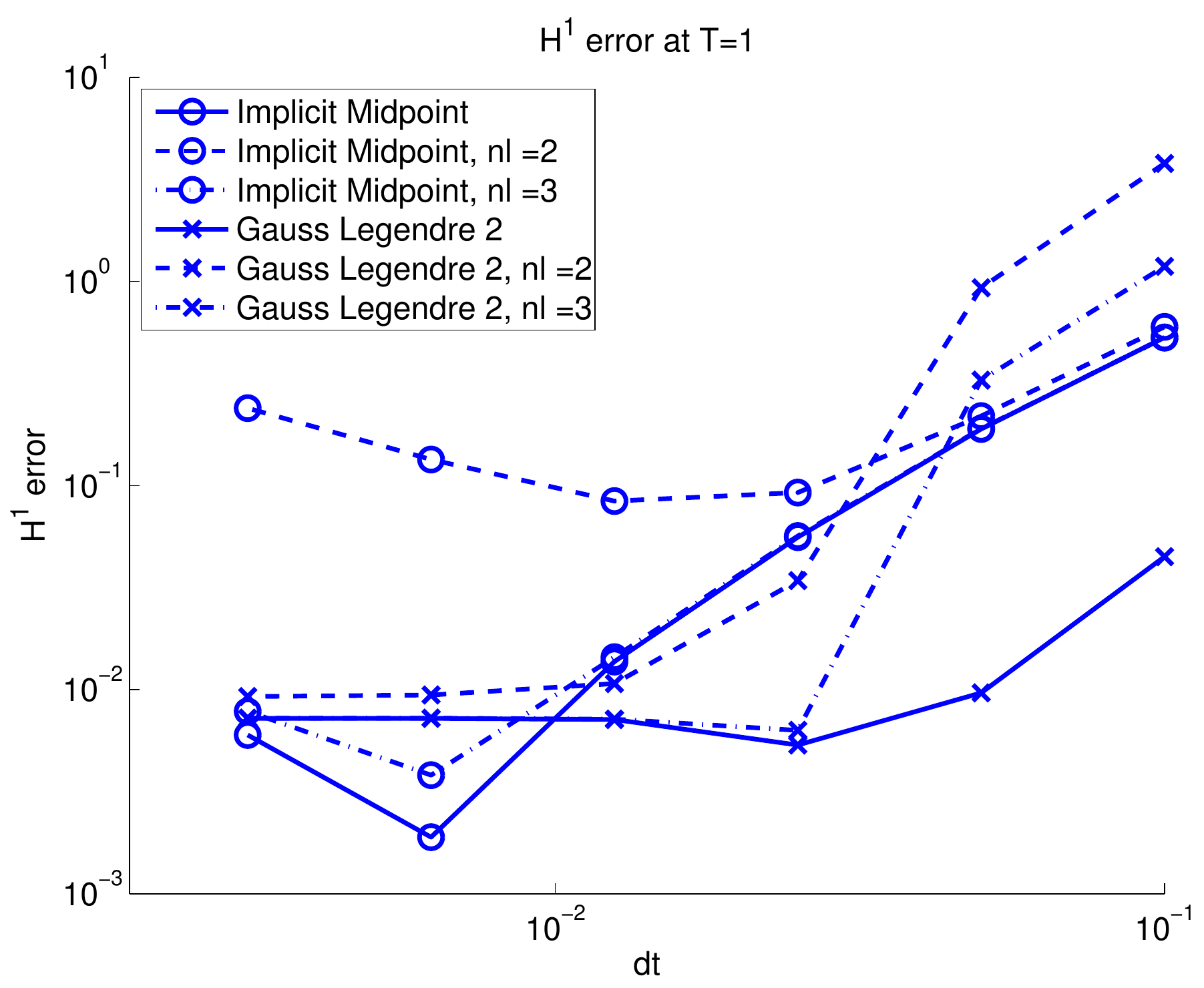}
    \end{subfigure}
    \begin{subfigure}[b]{0.5\textwidth}
        \includegraphics[width=\textwidth]{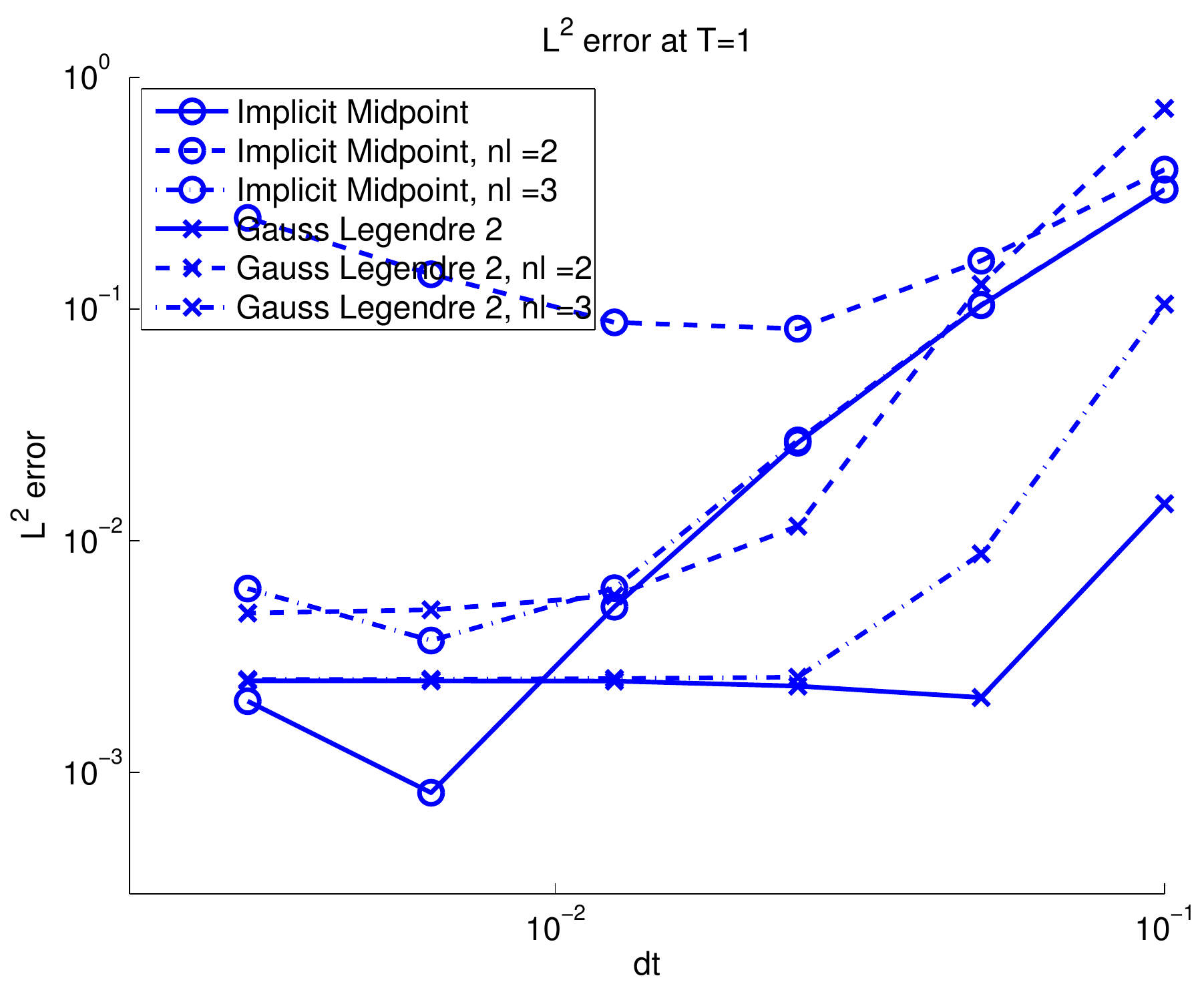}
    \end{subfigure}
    \caption{Left: $H^1$ error at time $T=1$. Right: $L^2$  error at time $T=1$.}
    \label{fig:h1errorat1:wave}
\end{figure}

The 2 stages Gauss Legendre scheme shows better performance for larger time steps and the localized gamblet solution is  close to the exact
gamblet solution for $nl=3$.

\section{The parabolic PDE with rough coefficients}
\label{secparabolic}
Consider the following prototypical example of the parabolic PDE with rough coefficients
\begin{equation}\label{eqn:parabolic}
\begin{cases}
    \mu(x)\partial_t u(x,t) -\diiv \big(a(x)  \nabla u(x,t)\big)=g(x,t) \quad  x \in \Omega; \\
    u(x,0)=u_0(x) \quad \text{on}\quad \Omega,\\
    u(x,t)=0 \quad \text{on}\quad \partial\Omega\times [0,T],
    \end{cases}
\end{equation}
where the domain $\Omega$ and the coefficients $\mu(x)$ and $a(x)$ are as in \eqref{eqnscalarzeta}
(i.e. in $L^\infty(\Omega)$ and satisfy \eqref{eqaineq} and \eqref{eqineqmu}).

As in Section \ref{secwavepde}, we consider $(\varphi_i)_{i\in \N}$  a finite-dimensional (finite-element) basis of $H^1_0(\Omega)$,  write $\tilde{u}(x,t)=\sum_{i\in \N} q_i(t)\varphi_i(x)$  the finite-element solution of \eqref{eqn:parabolic} in $\V$ (using the notations of Section \ref{secwavepde})
 and assume that the elements $(\varphi_i)_{i\in \N}$ are chosen to satisfy \eqref{eqhhgfff65f} and \eqref{eqhhgfff65fip} and so that $(\tilde{u}(x,t))_{0\leq t \leq T}$ is a \textit{good enough} approximation of the solution  of \eqref{eqn:parabolic}.

Recall that  $q$ is the solution of the ODE
\begin{equation}\label{eqn:parabolicFE}
    M \dot{q} + K q =f
\end{equation}
where $f$ is defined as in Section \ref{secwavepde} and $q_0=q(0)$ corresponds to the coefficients of $u_0$ in the $\varphi_i$ basis.

\subsection{Implicit-Euler time discretization}
The implicit Euler time-discretization of \eqref{eqn:parabolicFE} is
\begin{equation}\label{eqkjddkhsthimeul}
(M +\dt K) q_{n+1}=M q_n +\dt f_{n+1}
\end{equation}
Recall that implicit Euler is first-order accurate and, in
addition to being A-stable, is also  L-stable and B-stable (see Definition \ref{def:b-stability} of Section \ref{sec:appendix}) which are desirable for dissipative systems.

Let $u_n(x):=\sum_{i\in \N} q_{n,i} \varphi_i(x)$, be the corresponding approximation of $u(x,n \Delta t)$.
Observe that
 solving \eqref{eqkjddkhsthimeul}  is equivalent to obtaining the finite element solution (in $\V$) of
\begin{equation}\label{eqimeximplzlocfemp}
\frac{1}{\dt}\mu u_{n+1} -\diiv \big(a  \nabla u_{n+1}\big) =\frac{1}{\dt} \mu u_{n}+ g_{n+1}
\end{equation}
with $g_n(x):=g(x,n\dt)$.

As in Subsection \ref{subsecaccmidp}, to  achieve near linear complexity in the implementation of the implicit Euler method we will perform the inversion of implicit system in \eqref{eqkjddkhsthimeul} or  \eqref{eqimeximplzlocfemp}  in a localized $\z$-gamblet basis with $\z=2 \sqrt{\dt}$. Write $q_n^{\app}$  the output of the corresponding Algorithm
 \ref{alg:implicitEulerpara}.
\begin{algorithm}[!ht]
\caption{Implicit Euler with localized gamblets.}\label{alg:implicitEulerpara}
\begin{algorithmic}[1]
\STATE Set $\z=\dt$ and $\epsilon=\dt^3$.
\STATE Compute $\chi_i^{(k),\z,\loc}, \psi_i^{(k),\z,\loc}, B^{(k),\z,\loc}, R^{(k,k-1),\z,\loc}$ with Algorithm \ref{fastgambletsolve}.
\STATE  $q_0^{\app}:=q_0$  \COMMENT{$u_0^{\app}=u_0=\sum_{i\in \N} q_{0,i}\varphi_i$}.
\FOR{$n=0$ to $T/\dt-1$}
\STATE\label{linfempargx} Solve $(M+\dt K) q_{n+1}^{\app} =M q_n^{\app}+ f_{n+1}$ with Algorithm \ref{gambletsolveexlinsysloc}  \COMMENT{$f_{n,i}:= \int_{\Omega}\varphi_i g(x,n\dt)$}
\ENDFOR
\end{algorithmic}
\end{algorithm}
Write $u^\app_n:= \sum_{i\in \N} q^{\app}_{n,i} \varphi_i$. Observe that
Line \ref{linfempargx} of Algorithm  \ref{alg:implicitEulerpara} is  equivalent to solving
$\frac{1}{\dt}\mu u_{n+1}^{\app} -\diiv \big(a  \nabla u_{n+1}^{\app}\big) =\frac{1}{\dt} \mu u_{n}^{\app}+ g_{n+1}$
in $\V$ with Algorithm \ref{gambletsolveexlinsysloc}.

\begin{Theorem}
\label{thm:implicitEuler}
Let $u_n^\app$ be the output of Algorithm \ref{alg:implicitEulerpara}, and $u_n$ be the solution of implicit Euler time discretization of \eqref{eqn:parabolicFE} with time-step $\dt$. It holds true that for $n \leq T/\dt$,
\begin{equation}
	\|u_n^\app - u_n\|_{H^1_0(\Omega)}\leq C (\frac{T}{\dt})^2 \epsilon  \|g\|_{L^\infty(0, T, H^{-1}(\Omega))}
\end{equation}
where $\epsilon$ is the localization parameter in Algorithm \ref{alg:implicitEulerpara}.
\end{Theorem}
We refer to Subsection \ref{subsecproofthmimplicitEuler} for the proof of Theorem \ref{subsecproofthmimplicitEuler}.

\begin{Remark}
The proof of Theorem \ref{thm:implicitEuler} is based on Inequality \eqref{eqn:veineq} which is  the B-stability condition of Definition \ref{def:b-stability}. It is easy to show that this inequality is sufficient for validity of Theorem \ref{thm:implicitEuler}. Similar results to Theorem \ref{thm:implicitEuler} holds true for B-stable methods and they can also be accelerated  by the localized gamblets (e.g.,  the DIRK methods and SDIRK methods presented in Subsection \ref{subsecdirk3}).
\end{Remark}

\begin{Remark}\label{rmksublinear}
By truncating the propagation of the solution at higher frequencies (large $k$) in the generalized gamblet decomposition one obtains a numerical homogenization of the wave or parabolic equations (as in \cite{OwZh06c, OwZh:2007b, OwZh:2011}) with sub-linear complexity under sufficient regularity of initial conditions and source terms.
\end{Remark}

\subsection{TR-BDF2 time discretization}
The TR-BDF2  \cite{Bank:1985} time-discretization of \eqref{eqn:parabolicFE} is
\begin{equation}\label{eqnparatrbdf2}
\begin{cases}
    (M+\frac{\gamma\dt}{2} K) q_{n+\gamma} =(M-\frac{\gamma\dt}{2} K) q_{n}+\dt \frac{f_{n}+f_{n+\gamma}}{2}\\
		(M+\frac{1-\gamma}{2-\gamma}\dt K) q_{n+1} =\frac{M q_{n+\gamma}}{\gamma(2-\gamma)}-\frac{(1-\gamma)^2}{\gamma(2-\gamma)}M q_{n}+\frac{1-\gamma}{2-\gamma}\dt f_{n+1}
\end{cases}
\end{equation}
Recall that TR-BDF2 is  A-stable, L-stable   but  neither algebraically stable nor B-stable \cite{Hosea1996}.  It is  second order accurate and belongs to the category of diagonally implicit Runge-Kutta (DIRK) methods.
We  select $\gamma=2-\sqrt{2}$ to  minimize the local error \cite{Bank:1985} and ensure $\displaystyle\frac{\gamma}{2}=\frac{1-\gamma}{2-\gamma}$ (under that choice, \eqref{eqnparatrbdf2}   requires solving two systems of the same form $(M+\frac{\gamma\dt}{2} K)Q=b$ at each time step).
Let $u_n(x):=\sum_{i\in \N} q_{n,i} \varphi_i(x)$ be the corresponding approximation of $u(x,n \Delta t)$.
Observe that
 solving \eqref{eqnparatrbdf2}  is equivalent to obtaining the finite element solution (in $\V$) of
\begin{equation}\label{eqnparatrbdf2pde}
\begin{cases}
\frac{2}{\gamma\dt}\mu u_{n+\gamma} -\diiv \big(a  \nabla u_{n+\gamma}\big) =\frac{2 \mu u_{n}}{\gamma\dt} + \diiv \big(a  \nabla u_{n}\big)+ \frac{g_{n}+g_{n+\gamma}}{\gamma} \\
\frac{2}{\gamma\dt}\mu u_{n+1} -\diiv \big(a  \nabla u_{n+1}\big) =\frac{1}{\gamma(1-\gamma)\dt} \mu u_{n+\gamma}-  \frac{1-\gamma}{\gamma \dt }\mu u_{n}  + g_{n+1}
\end{cases}
\end{equation}
As in Subsection \ref{subsecaccmidp}, to  achieve near linear complexity  we will perform the inversion of implicit systems in \eqref{eqnparatrbdf2pde}   in a localized $\z$-gamblet basis with $\zeta= \sqrt{2\gamma \dt}$.

\subsection{DIRK3 and SDIRK3}\label{subsecdirk3}
Other popular time-discretion methods for \eqref{eqn:parabolicFE} are DIRK3  \cite{BoomZingg:2015} and  SDIRK3  \cite[p262]{Butcher:2008}. DIRK3 (3-stages Diagonally Implicit Runge-Kutta)  is L-stable and B-stable \cite{BoomZingg:2015}, and its  Butcher tableau is given in Table \ref{tabDIRK3}.
\begin{table}[h]
  \centering
  \scriptsize
  \begin{tabular}{l|lllll}
    0.0585104413419415   & 0.0585104413426586 & 0.0                &   0.0 \\
    0.8064574322792799   & 0.0389225469556698 & 0.7675348853239251 &   0.0 \\
    0.2834542075672883   & 0.1613387070350185 & -0.5944302919004032  & 0.7165457925008468 \\\hline
					 & 0.1008717264855379 & 0.4574278841698629 & 0.4417003893445992
  \end{tabular}
  \caption{Butcher tableau for DIRK3}
  \label{tabDIRK3}
\end{table}
The implementation of DIRK3 requires solving 3 equations
$\frac{1}{\dt A_{i,i}}\mu w_i -\diiv \big(a  \nabla w_i\big) = b_i$ using finite-elements in $\V$, where $A_{1,1},\ldots,A_{3,3}$ are the diagonal entries of
 the Runge-Kutta matrix $A$ of DIRK3.

SDIRK3 (3-stage Singly Diagonally Implicit Runge-Kutta) is L-stable \cite[P 262]{Butcher:2008} and its Butcher tableau is given in
Table \ref{tabSDIRK3} which has identical diagonal entries. The implementation of SDIRK3 requires solving 3 equations
$\frac{1}{\dt \lambda}\mu w_i -\diiv \big(a  \nabla w_i\big) = b_i$ using finite-elements in $\V$, where $\lambda$ is defined in Table \ref{tabSDIRK3}.
\begin{table}[h]
  \centering
  \begin{tabular}{l|lllll}
    $\lambda$     & $\lambda$    & 0   & 0         \\
    $\frac12(1+\lambda)$    & $\frac12(1-\lambda)$  & $\lambda$        & 0   \\
    1      & $\frac{1}{4}(-6\lambda^2+16\lambda-1)$ & $\frac{1}{4}(6\lambda^2-20\lambda+5)$	& $\lambda$  \\\hline
					 & $\frac{1}{4}(-6\lambda^2+16\lambda-1)$ & $\frac{1}{4}(6\lambda^2-20\lambda+5)$	& $\lambda$
  \end{tabular}
  \caption{Butcher tableau for SDIRK3 where  $\lambda\simeq 0.4358665215$  (identified as a root of $\frac{1}{6}-\frac{3}{2}\lambda + 3\lambda^2-\lambda^3=0$) ensures  L-stability.}
  \label{tabSDIRK3}
\end{table}
As in Subsection \ref{subsecaccmidp}, to  achieve near linear complexity  we will perform the inversion of linear systems of DIRK3 using 3 localized $\z$-gamblets with $\zeta= \sqrt{ A_{i,i} \dt/2}$, and the inversion of linear systems of SDIRK3 using only 1 localized $\z$-gamblets with $\zeta= \sqrt{ \lambda \dt/2}$.

\subsection{Numerical experiments}
Let $a(x)$ be defined as in Example \ref{egadef} and $\q = 6$, $g(x,t) = \sin(2\pi(t+x_1))\cos(2\pi(t+x_2))$, $u(x,0) = \sin(2\pi x_1)\cos(2\pi x_2)$.
The reference solution is the finite element solution with piecewise bilinear elements $\{\varphi_i\mid i\in \N\}$, and
\textit{Matlab} built-in integrator \textit{ode15s} with time step $dt = 1/1280$. We test the performance of exact and localized
gamblets  adapted to implicit Euler, TR-BDF2, DIRK3, SDIRK3, and fully implicit Runge-Kutta methods Radau IIA and Lobatto IIIC (which will be
introduced in \S~\ref{sec:irk}). We compute  numerical solutions up to time $T = 1$.
Figure \ref{fig:sol:para} shows  the reference solution, the DIRK3 solution with localized gamblet ($nl=3$) and its numerical error with respect to the reference solution.
\begin{figure}[H]
\includegraphics [width=13cm, height=4cm]{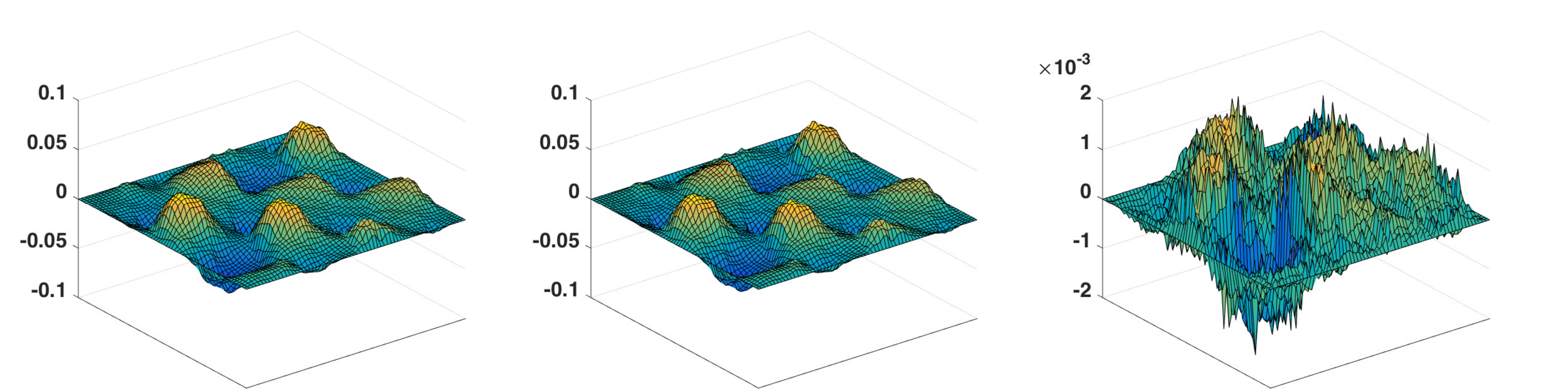}
\caption{Solutions at $T=1$. Left: Reference solution; Middle: numerical solution with localized gamblet, DIRK3, $nl=3$ and $\dt = 0.05$; Right: error of the localized solution.}
\label{fig:sol:para}
\end{figure}

\begin{figure}[H]
\centering
\includegraphics [width=6cm]{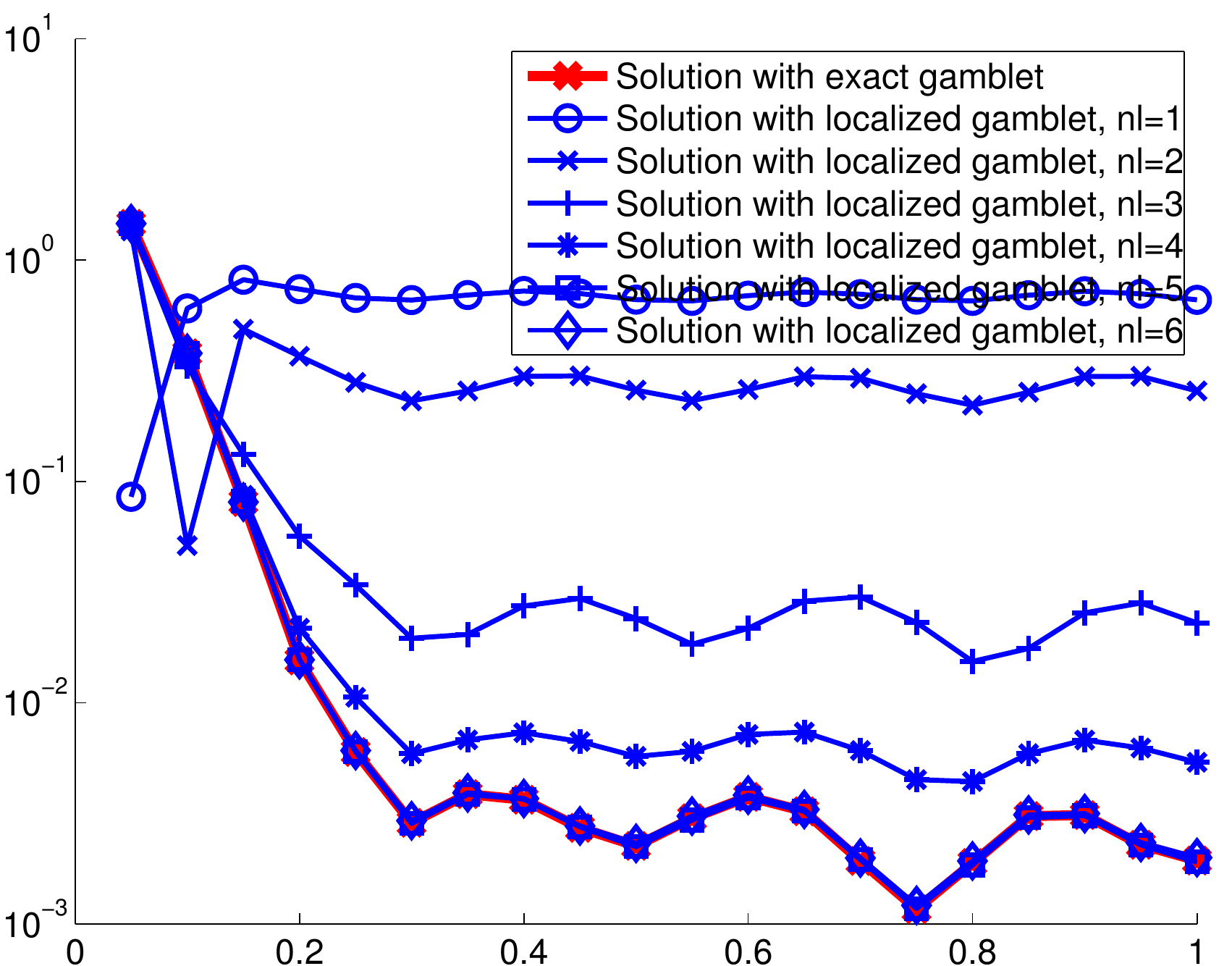}\quad
\includegraphics [width=6cm]{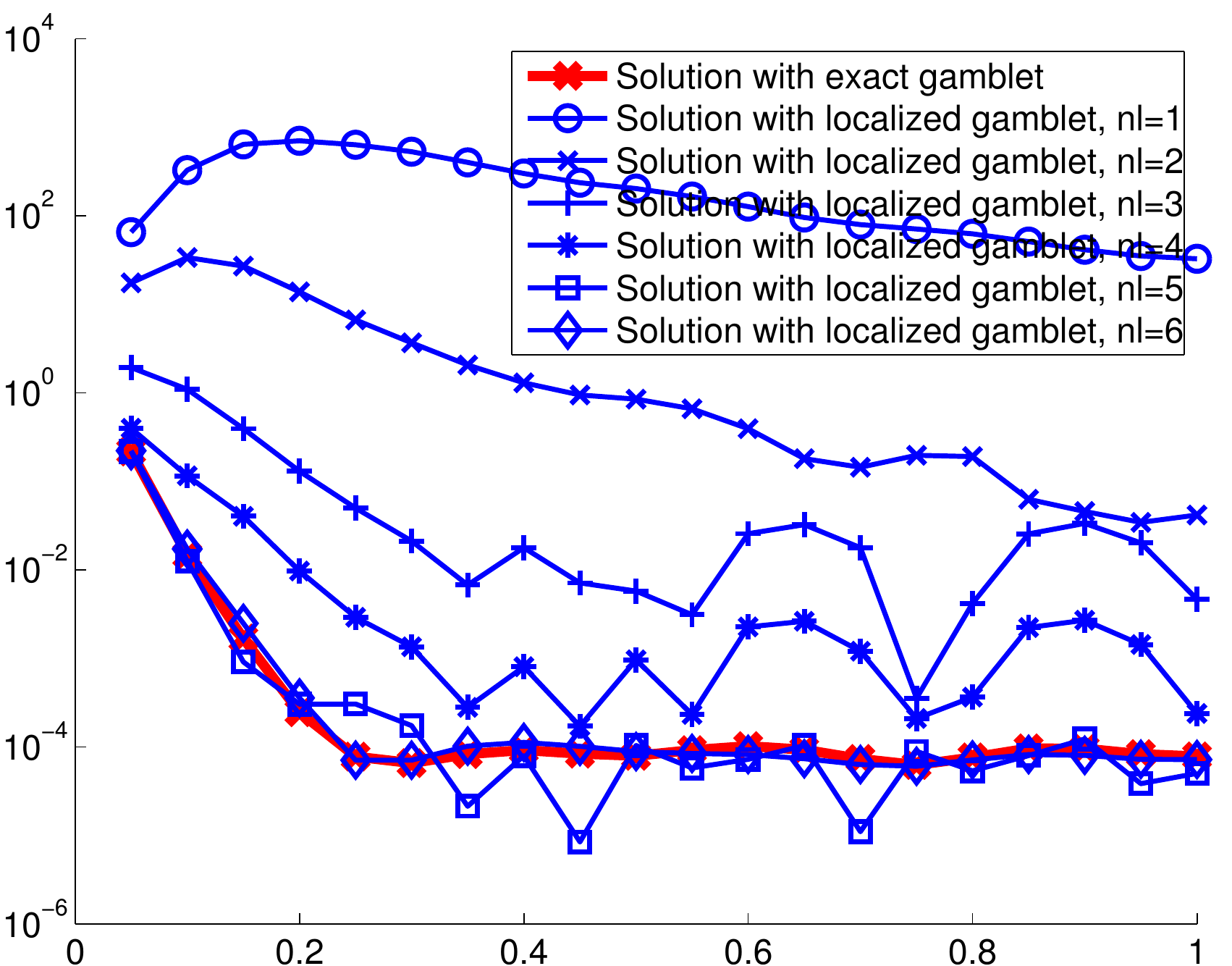}
\caption{Evolution of the relative error of energy w.r.t localization ($\dt=0.05$): Left, TR-BDF2; Right, Radau IIA scheme.}
\label{fig:energy:para}
\end{figure}

Figure \ref{fig:energy:para} shows the relative error of the energy of gamblet solutions with respect to time and localization. The errors decays when the number of layers increases, and Radau IIA scheme has better accuracy compared to TR-BDF2 method, note that Radau IIA is a 5th order method and TR-BDF2 is a 2nd order method.

 \begin{figure}[H]
    \begin{subfigure}[b]{0.5\textwidth}
        \includegraphics[width=\textwidth]{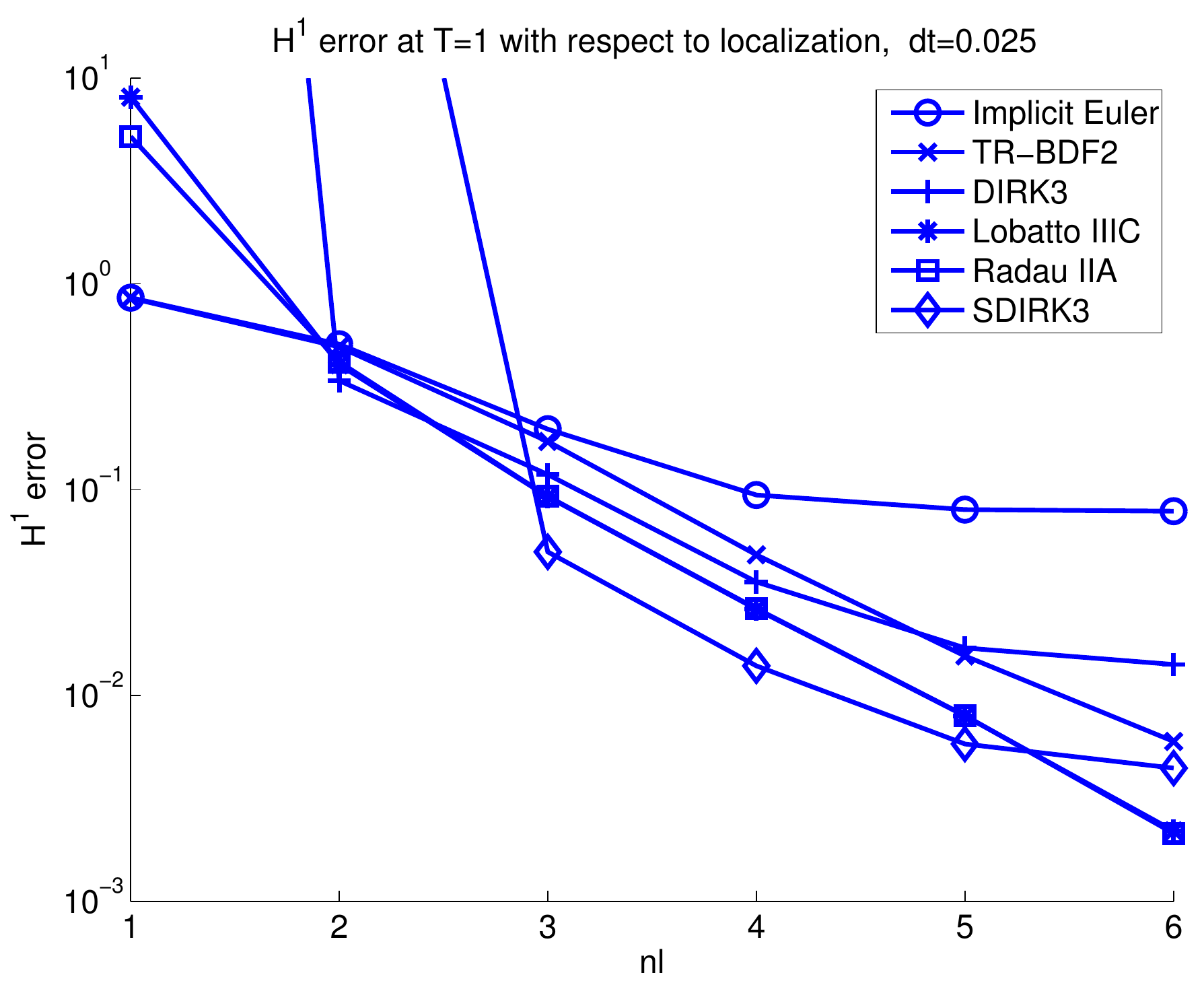}
    \end{subfigure}
        \begin{subfigure}[b]{0.5\textwidth}
        \includegraphics[width=\textwidth]{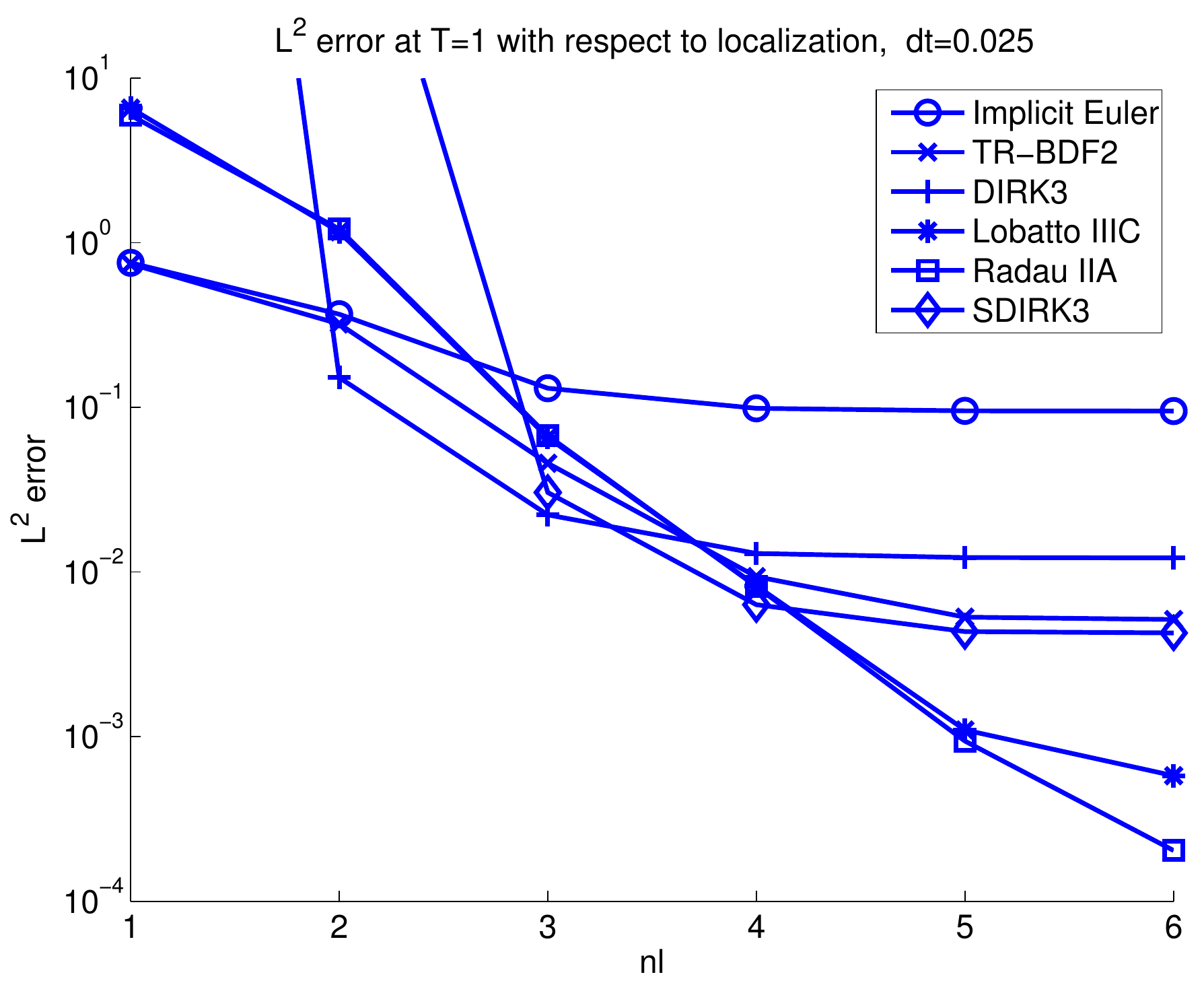}
    \end{subfigure}
    \caption{Left: $H^1$ error w.r.t localization parameter $nl$; Right: $L^2$ error w.r.t localization parameter $nl$. $\dt = 0.025$}
    \label{fig:errorat1wrtnl:para}
\end{figure}

Figure \ref{fig:errorat1wrtnl:para} shows the $H^1$ and $L^2$ errors at $T=1$ for all the 6 numerical schemes with respect to different localization parameters $nl=1,\cdots, 6$ for exact and localized gamblet solutions. High order methods such as Radau IIA and Lobatto IIIC have best accuracy if more localizaton layers are used. When $nl=2$ or $3$, it appears that simpler methods such as TR-BDF2, DIRK3 or SDIRK3 achieve a better balance between accuracy and computational cost.

Figure  \ref{figure:band_i4para_nl3} compares the components of the reference solution and localized gamblet solutions in each subband $\W^{(k),\z}$, computed with the DIRK3 scheme (we observe similar results for other schemes). Most of the error occurs in the first subband and at early time, and  gets damped quickly.


\begin{figure}[H]
\centering
\includegraphics [width=10cm, height=6cm]{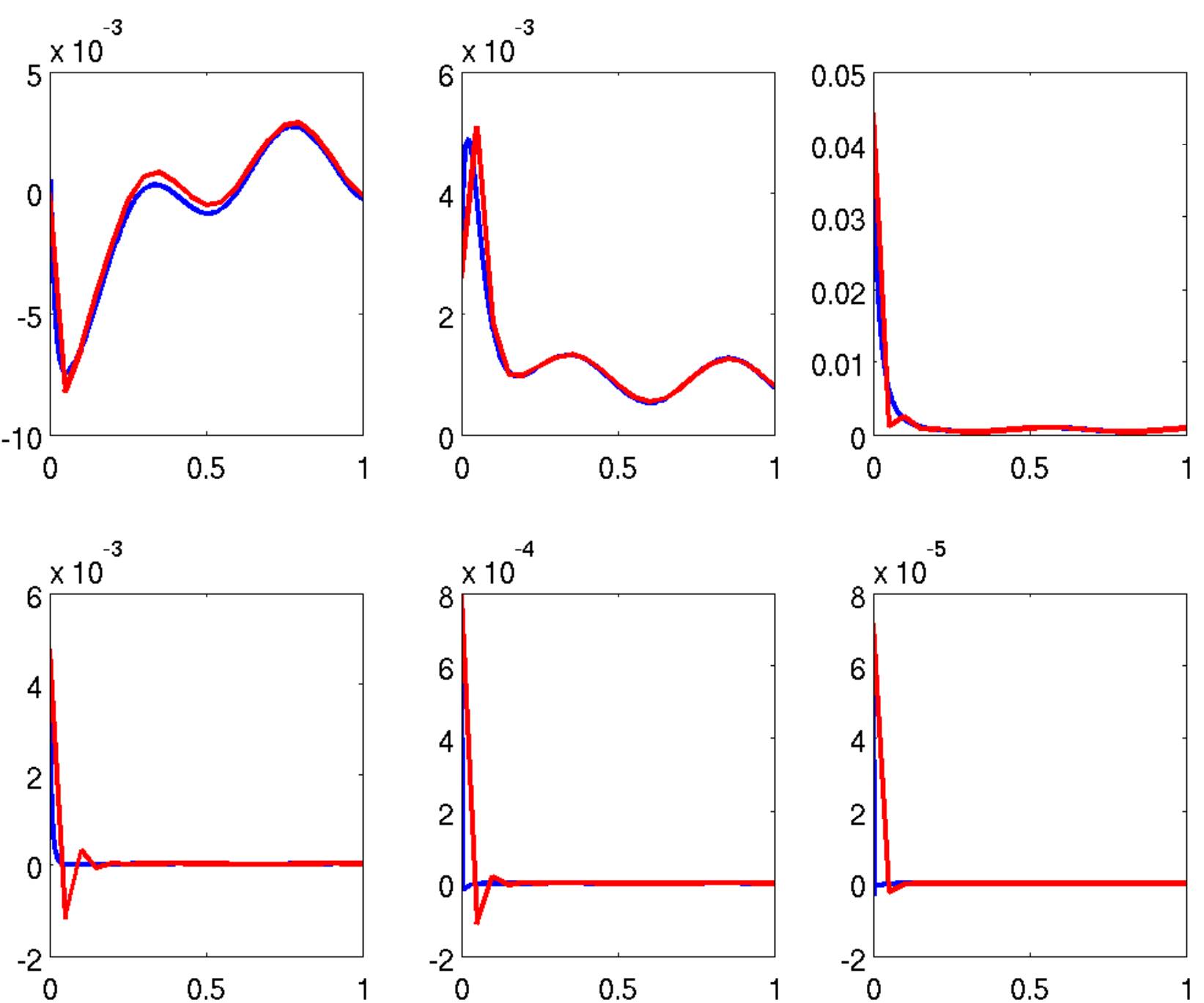}
\caption{Evolution of $\chi_1^{(k)}$ component in subband $\W^{(k), \z}$, $k = 1, \cdots, 6$, with localized gamblet ($nl=3$) for DIRK3 scheme. The blue curve is obtained from the reference solution, and the red curve is obtained from the localized gamblet solution.
}
\label{figure:band_i4para_nl3}
\end{figure}

Figure \ref{fig:errorat1:para} shows $H^1$ and $L^2$ errors with exact gamblets and localized gamblets ($nl=3$) at time $T=1$ with time steps  $\dt = 1/10, 1/20, 1/40, 1/80, 1/160, 1/320$. All 6 methods are tested: implicit Euler, TR-BDF2, DIRK3, Lobatto IIIC, Radau IIA and SDIRK3. In general, the higher order methods such as Radau IIA and Lobatto IIIC are more accurate for coarser time steps, refinement of time steps does not reduce the error further due to the fixed spatial resolution. Localized gamblet solutions converge as time steps decrease. Localized TR-BDF2 and Implicit Euler have a slower convergence rate compared to higher order methods.

 \begin{figure}[H]
    \begin{subfigure}[b]{0.5\textwidth}
        \includegraphics[width=\textwidth]{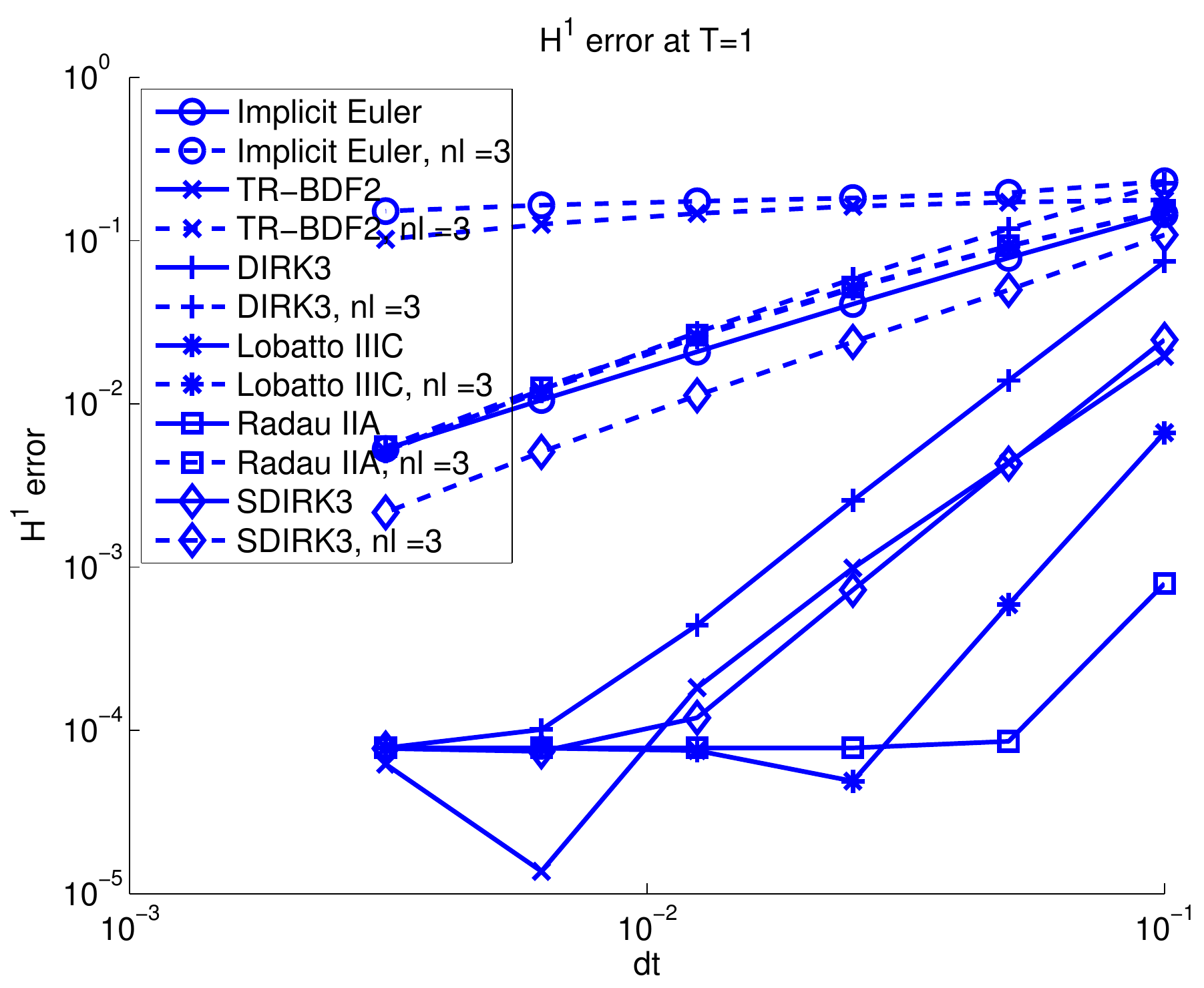}
    \end{subfigure}
    \begin{subfigure}[b]{0.5\textwidth}
        \includegraphics[width=\textwidth]{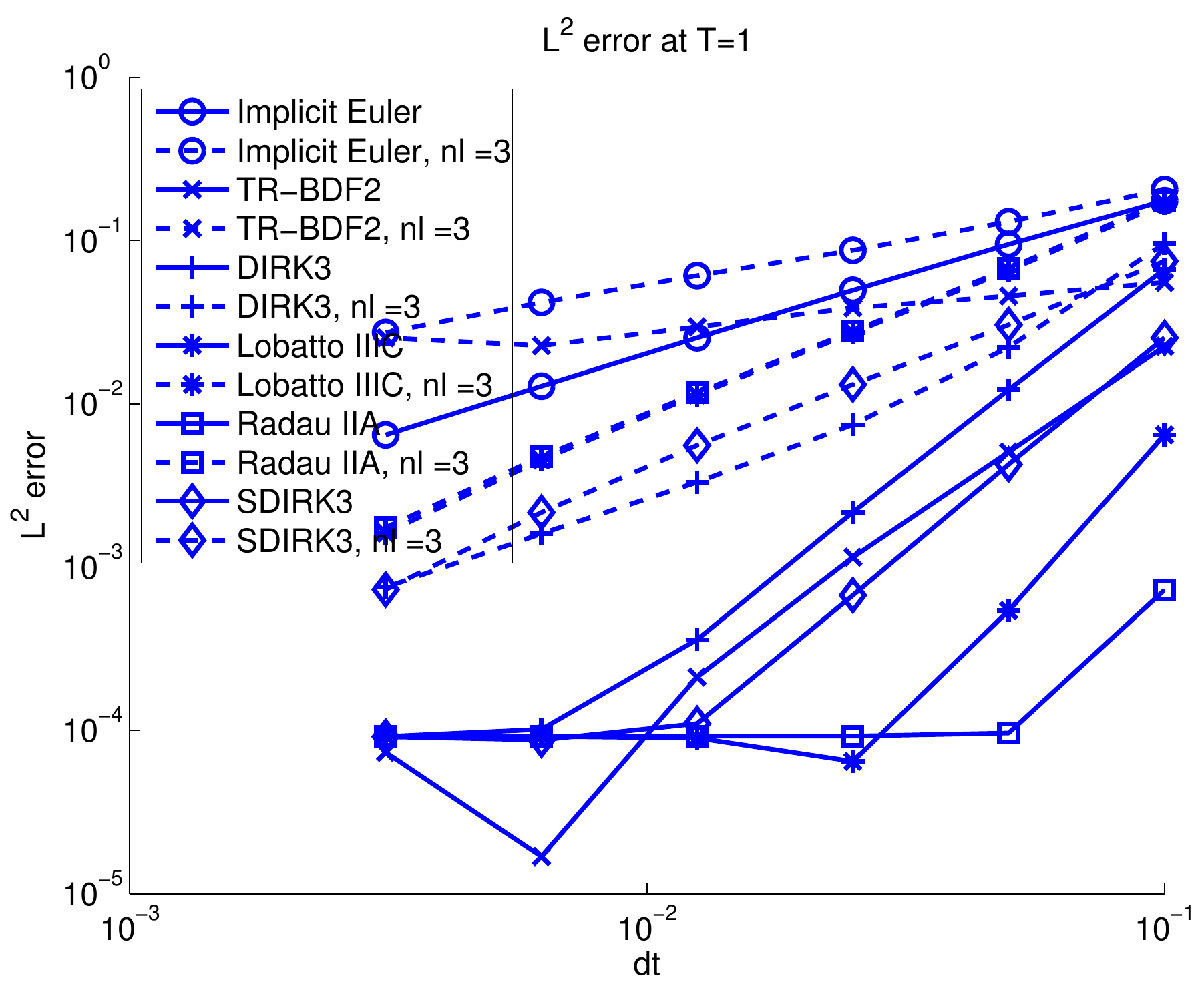}
    \end{subfigure}
    \caption{Left: $H^1$ error at time $T=1$; Right: $L^2$ error at time $T=1$.}
    \label{fig:errorat1:para}
\end{figure}

\section{Complex gamblets for  higher order implicit schemes}\label{seccomplexgamblets}

\subsection{Solving wave equation with 2 stages Gauss-Legendre scheme}
\label{sec:gl2}
The implicit midpoint scheme introduced in section \S~\ref{sec:implicitmidpoint} is a 1 stage Gauss-Legendre method. To obtain higher order method for the wave equation, we can use higher order Gauss-Legendre methods . Here we describe the
implementation of 2 stages Gauss-Legendre scheme, which is 4th order accurate, unconditionally stable, symplectic, symmetric (time-reversible) and preserves quadratic invariants exactly \cite{HairerLubichWanner06}. In particular, we will show how to use gamblets to achieve near linear complexity.

The Butcher tableau for the 2 stages Gauss-Legendre scheme is as follows.
\begin{table}[h]
  \centering
  \begin{tabular}{l|llll}
    $\frac12-\frac{\sqrt{3}}{6}$   & $\frac{1}{4}$                &  $ \frac{1}{4}-\frac{\sqrt{3}}{6} $\\
    $\frac12+\frac{\sqrt{3}}{6} $  & $\frac{1}{4}+\frac{\sqrt{3}}{6} $&   $\frac{1}{4} $\\ \hline
					 &  $\frac12$ & $\frac12$
  \end{tabular}
  \caption{Butcher tableau for GL2}
  \label{tabGL2}
\end{table}
Namely, the Runge-Kutta matrix, weights and nodes are $A = \begin{pmatrix} \frac{1}{4}  &   \frac{1}{4}-\frac{\sqrt{3}}{6} \\ \frac{1}{4}+\frac{\sqrt{3}}{6} &  \frac{1}{4} \end{pmatrix}$, $\displaystyle b = \big(\frac12, \frac12\big)$, and $\displaystyle c=\big(\frac12-\frac{\sqrt{3}}{6}, \frac12+\frac{\sqrt{3}}{6}\big)^T$.

\def\Z{\mathbf 0}
\def\I{\mathbf I}
Using notations in \S~\ref{secwavepde}, let $y_n = (q_n; p_n)$ be a column vector of length $2N$, $f_n^1 = (\Z; f(t_n+c_1 h))$, and $f_n^2 = (\Z; f(t_n+c_2 h))$, where $\Z$ is the zero valued column vector of length $N$.  Apply the 2 stages Gauss-Legendre scheme to equation \eqref{eqhdsgjhdgdfi}, we have
\begin{align*}
	k_n^1 & = f_n^1 + \begin{pmatrix}\Z & M^{-1}\\ -K &\Z\end{pmatrix} (y_n+ \dt A_{11} k_n^1 + \dt A_{12}k_n^2) \\
	k_n^2 & = f_n^2 + \begin{pmatrix}\Z & M^{-1}\\ -K &\Z\end{pmatrix} (y_n+ \dt A_{21} k_n^1 + \dt A_{22}k_n^2) \\
	y_{n+1} & = y_{n} + \dt(b_1k_n^1 + b_2k_n^2)
\end{align*}
Define $H$ in the following tensor product format,
\begin{equation}
	H := \begin{pmatrix}1  & 0 \\ 0 & 1 \end{pmatrix}\otimes \begin{pmatrix}M  & 0 \\ 0 & \I \end{pmatrix} +
	        a\otimes \begin{pmatrix} 0 & -h \I \\ hK & 0 \end{pmatrix}
\end{equation}
where $\I$ is the identity matrix of size $N$.

Write $F_n := (p_n; -K q_n + f(t_n+c_1 \dt); p_n; -Kq_n+f(t_n+c_1 \dt))^T$, we need to solve the coupled linear system $H k_n = F_n$ for $k_n = (k_n^1; k_n^2)$, and then $y_{n+1}$.

Since  $A$ is diagonalizable, we can write $A = S\Lambda S^{-1}$ with $\Lambda = \diag(\lambda_1, \lambda_2)$. Use $T := S\otimes \begin{pmatrix}\I & \Z \\ \Z & \I \end{pmatrix}$  to block diagonalize $H$. That is,
\begin{equation}
	\tilde{H} := T H T^{-1} =
	\begin{pmatrix}
		M & -\dt\lambda_1 \I & \Z & \Z\\
		\dt\lambda_1 K & \I  & \Z & \Z\\
		\Z & \Z & M & - \dt\lambda_2 \I\\
		\Z & \Z & \dt\lambda_2 K & \I
	\end{pmatrix}
\end{equation}
where $T^{-1} = S^{-1}\otimes \begin{pmatrix}\I & \Z \\ \Z & \I \end{pmatrix}$.

Write $\tilde{F}_n := T^{-1} F_n$ and $\tilde{k}_n := T^{-1} k_n$, we have $\tilde{H} \tilde{k}_n = \tilde{F}_n$.   Therefore, instead of solving the coupled linear system with respect to $H$, we solve the decoupled linear systems with respect to $\tilde{H}$.  Let $\bar{H}_1 := \begin{pmatrix} M & -\dt\lambda_1 \I\\ \dt\lambda_1 K & \I \end{pmatrix}$ and $\bar{H}_2: = \begin{pmatrix} M & -\dt\lambda_2 \I\\ \dt\lambda_2 K & \I \end{pmatrix}$. Similar to the gamblet solution for implicit midpoint scheme \eqref{eqimeximpl}, we only need to introduce gamblets associated with the matrices $\dt^2\lambda_i^2K + M$ to solve the linear systems associated with $\bar{H}_i$, $i=1,2$.

\subsection{Solving parabolic equation with fully implicit Runge-Kutta methods}
\label{sec:irk}

Let $q$ be the solution of the semidiscrete ODE system derived from the parabolic PDE \eqref{eqn:parabolic},
\begin{equation}\label{eqn:parabolicODE}
    M \dot{q} + K q =f
\end{equation}
where $f$ is defined as in Section \S~\ref{secwavepde} and $q_0=q(0)$ corresponds to the coefficients of $u_0$ in the $\varphi_i$ basis. Higher order implicit Runge-Kutta methods can be used to solve  \eqref{eqn:parabolicODE} to achieve better stability  and accuracy.
Write $A$, $b$ and $c$  the Runge-Kutta matrix,  weights, and nodes. Let $s$ be the number of stages of the Runge-Kutta method. The Runge-Kutta method can be written as,
\begin{displaymath}
	q_{n+1} = q_n + \dt \sum_{i = 1}^s b_i k_n^i,
\end{displaymath}
with
\begin{displaymath}
	M k_n^i = -K (q_n + \dt\sum_{i=1}^s A_{ij}k_n^j)  + f(t_n+c_i\dt),\quad k=1,\dots, s
\end{displaymath}
Write
\begin{displaymath}
	H: = \I_s\otimes M + \dt A\otimes K\,.
\end{displaymath}
where $\I_s$ is the identity matrix of size $s$.

To obtain $q_{n+1}$, we need to solve the coupled linear system $H k_n = F_n$ for $k_n: = (k_n^1; \cdots; k_n^s)$, 
$F_n : = (-Kq_n + Mf(t_n + c_1 \dt); \dots; -Kq_n + Mf(t_n + c_s \dt))$.

Assume $A$ is diagonalizable, such that $A = S\Lambda S^{-1}$ with $\Lambda = \diag(\lambda_1, \dots, \lambda_s)$. Write $T := S\otimes \I$ (and $T^{-1} = S^{-1}\otimes \I)$, we have,
\begin{displaymath}
	\tilde{H} :=THT^{-1} = \I_s \otimes M + \dt\Lambda \otimes K
\end{displaymath}
Write $\tilde{k}_n := Tk_n$ and $\tilde{F}_n := TF_n$, we can use gamblets associated with the matrices $M+\dt \lambda_i K (i=1,\dots, s)$ to solve the decoupled linear system $\tilde{H} \tilde{k}_n = \tilde{F}_n$ for $\tilde{k}_n$, then obtain $k_n$ and $q_{n+1}$.

In our numerical illustrations in \S~\ref{secparabolic} we have considered the following fully implicit Runge-Kutta methods: Lobatto IIIC and Radau IIA.
Recall that
Lobatto IIIC \cite{HairerWanner:1996, Butcher:2008} is 4th order accurate, L-stable, B-stable,  stiffly accurate, and its Butcher tableau is in Table \ref{tabLobattoIIIC}.
\begin{table}[h]
  \centering
  \begin{tabular}{l|lllll}
    0     & 1/6    	& -1/3   	& 1/6         \\
    0.5  & 1/6    	& 5/12   	& -1/12     \\
    1     & 1/6   	& 2/3		& 1/6	  	\\\hline
	   & 1/6	& 2/3		& 1/6
  \end{tabular}
  \caption{Butcher tableau for Lobatto IIIC scheme.}
  \label{tabLobattoIIIC}
\end{table}

\def\ss{\sqrt{6}}
Recall also that
Radau IIA \cite{HairerWanner:1996, Butcher:2008} is 5th order accurate, A-stable and its Butcher tableau is in Table \ref{tabRadauIIA}.
\begin{table}[h]
  \centering
  \begin{tabular}{l|cllc lc l cl}
    $\frac{2}{5}-\frac{\ss}{10}$ & $\frac{11}{45}-\frac{7\ss}{360}$ & $\frac{37}{225}-\frac{169\ss}{1800}$ & $-\frac{2}{225}+\frac{\ss}{75}$       \\
    $\frac{2}{5}+\frac{\ss}{10}$ & $\frac{37}{225}+\frac{169\ss}{1800}$ & $\frac{11}{45}+\frac{7\ss}{360}$ & $-\frac{2}{225}-\frac{\ss}{75}$       \\
    1     & $\frac{4}{9}-\frac{\ss}{36}$	& $\frac{4}{9}+\frac{\ss}{36}$	& $\frac{1}{9}$  	\\\hline
	   & $\frac{4}{9}-\frac{\ss}{36}$	& $\frac{4}{9}+\frac{\ss}{36}$	& $\frac{1}{9}$
  \end{tabular}
  \caption{Butcher tableau for Radau IIA scheme.}
  \label{tabRadauIIA}
\end{table}

\subsection{Gamblet transformation for complex valued matrix}
As shown in \S~\ref{sec:gl2} and \S~\ref{sec:irk}, we need to  introduce generalized gamblets for matrices of the form $M+\dt^2\lambda^2K$ for hyperbolic equations and $M+\dt\lambda K$ for parabolic equations to open the complexity bottleneck of higher order implicit schemes, where $\lambda$ is the eigenvalue of the corresponding Runge-Kutta matrix. For implicit Runge-Kutta methods, such as Gauss-Legendre type, Lobatto type, and Radau type, $\lambda$ is in general a complex number.
\begin{Definition}
The definition of  complex valued $\zeta$-gamblets is algebraically identical to that of real-value $\zeta$-gamblets. We keep using the scalar product defined in \eqref{eqsp}, and the notion of orthogonality remains the same as in \eqref{eqsp}.   Algorithm \ref{gambletsolve} and \ref{fastgambletsolve} remain unchanged, in particular we do not replace  matrix transpose operations with complex conjugate transpose operations.
\end{Definition}

\begin{Remark}
It is a simple observation that the gamblet transform remains algebraically exact for complex valued matrices (this can be observed rigorously and numerically).
 Although, we loose the positivity of the scalar product \eqref{eqsp} when $\zeta$ is complex (the matrices $M+\dt^2\lambda^2K$ and $M+\dt\lambda K$ remain symmetric when $\lambda$ is a complex complex number but they are neither positive definite nor Hermitian), figures \ref{fig:condnum_para} and \ref{fig:eigrange_para}, show that, for Radau IIA (complex $\zeta$) and  SDIRK3 (real $\zeta$),  condition numbers and ranges of eigenvalues  remain similar. Although this is not proven in this paper, we suspect that the preservation of  uniformly bounded condition numbers and exponential decay with complex valued $\zeta$ is generic.
\end{Remark}

 Figures \ref{fig:realbasis} and \ref{fig:imagbasis}  illustrate the real part and imaginary parts of the gamblet basis associated with the first complex eigenvalue (0.1626+0.1849i) of Radau IIA and $\dt = .1$.

\begin{figure}[H]
        \includegraphics[height=6.5cm,width=\textwidth]{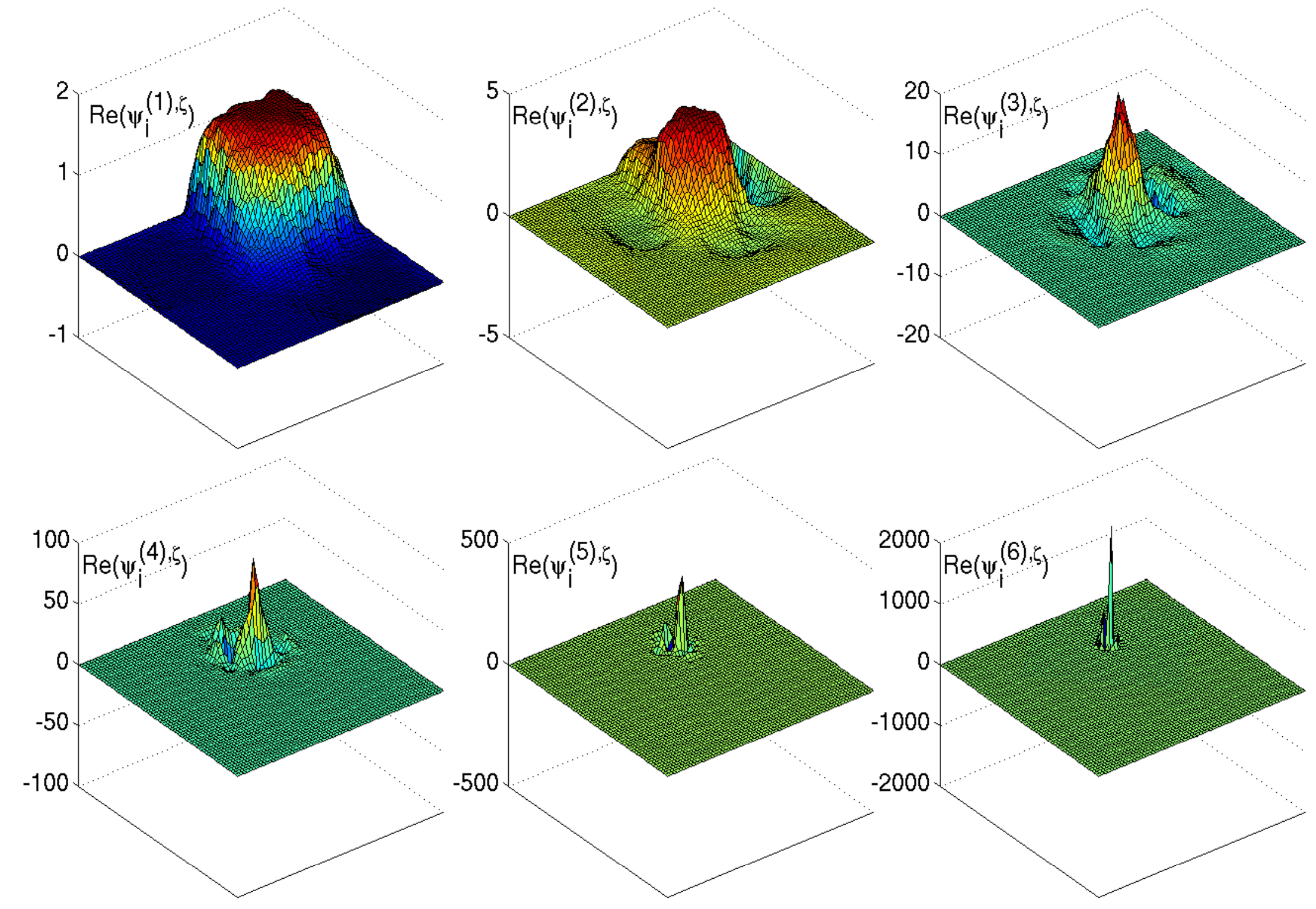}
        \caption{Real part of $\chi^{k}_i$ associated with the first complex eigenvalue of Radau IIA and $\dt=.1$ for parabolic equation.}
        \label{fig:realbasis}
\end{figure}

\begin{figure}[H]
        \includegraphics[height=6.5cm, width=\textwidth]{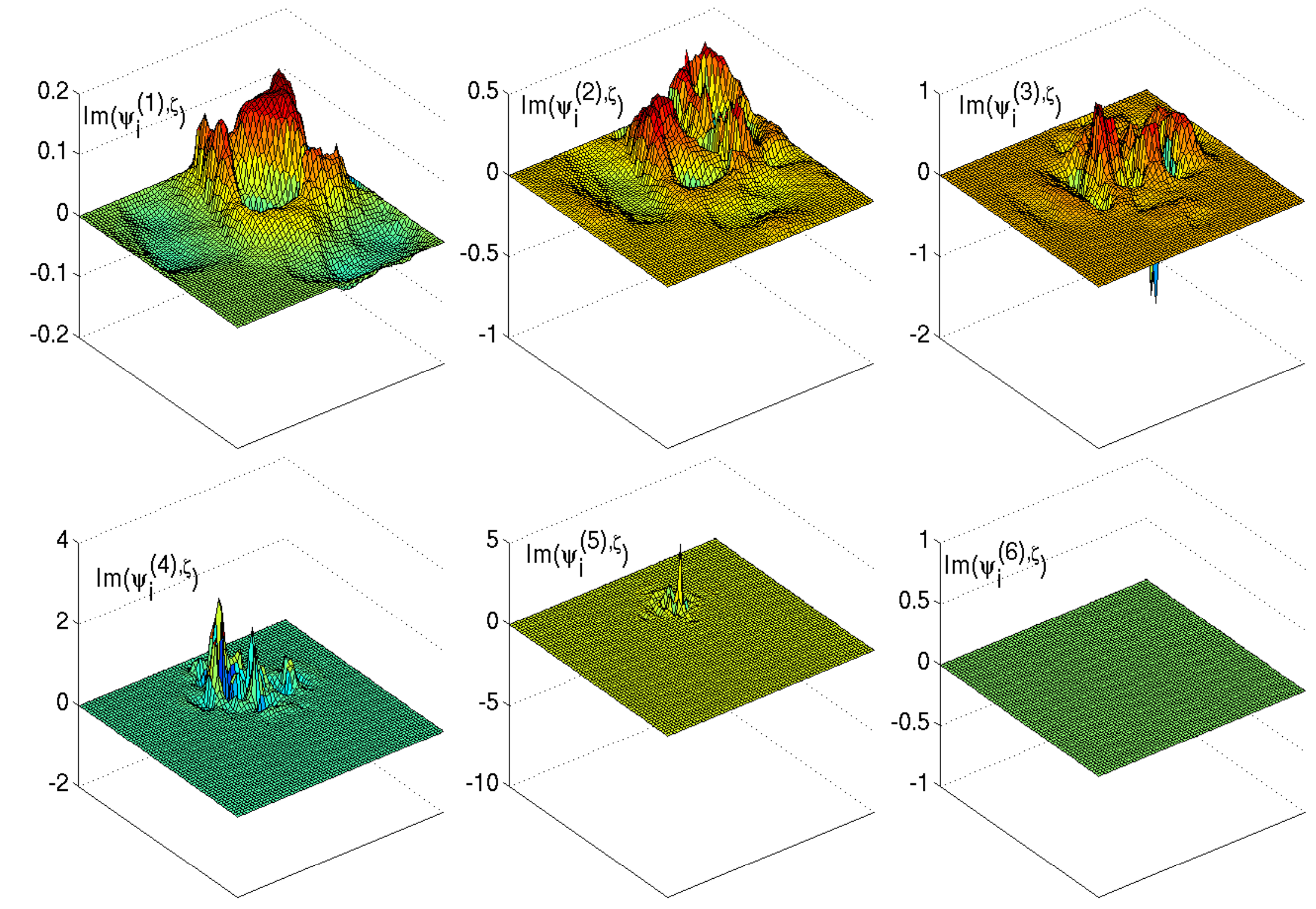}
        \caption{Imaginary part of $\chi^{k}_i$ associated with the first complex eigenvalue of Radau IIA and $\dt=.1$ for parabolic equation.}
        \label{fig:imagbasis}
\end{figure}

Figures \ref{fig:condnum_para} and \ref{fig:eigrange_para} compare the condition numbers and the ranges of eigenvalues in $\W^{(k)}$ of the complex gamblets  associated with the first complex eigenvalue (0.1626+0.1849i) of the RK matrix of Radau IIA and the real gamblets associated with the first eigenvalue (0.4359) of the RK matrix of SDIRK3. Note the similarity between   the condition numbers and ranges of eigenvalues of the  complex gamblets of Radau IIA and the  real gamblets of SDIRK3.
 \begin{figure}[H]
    \begin{subfigure}[b]{0.5\textwidth}
        \includegraphics[width=\textwidth]{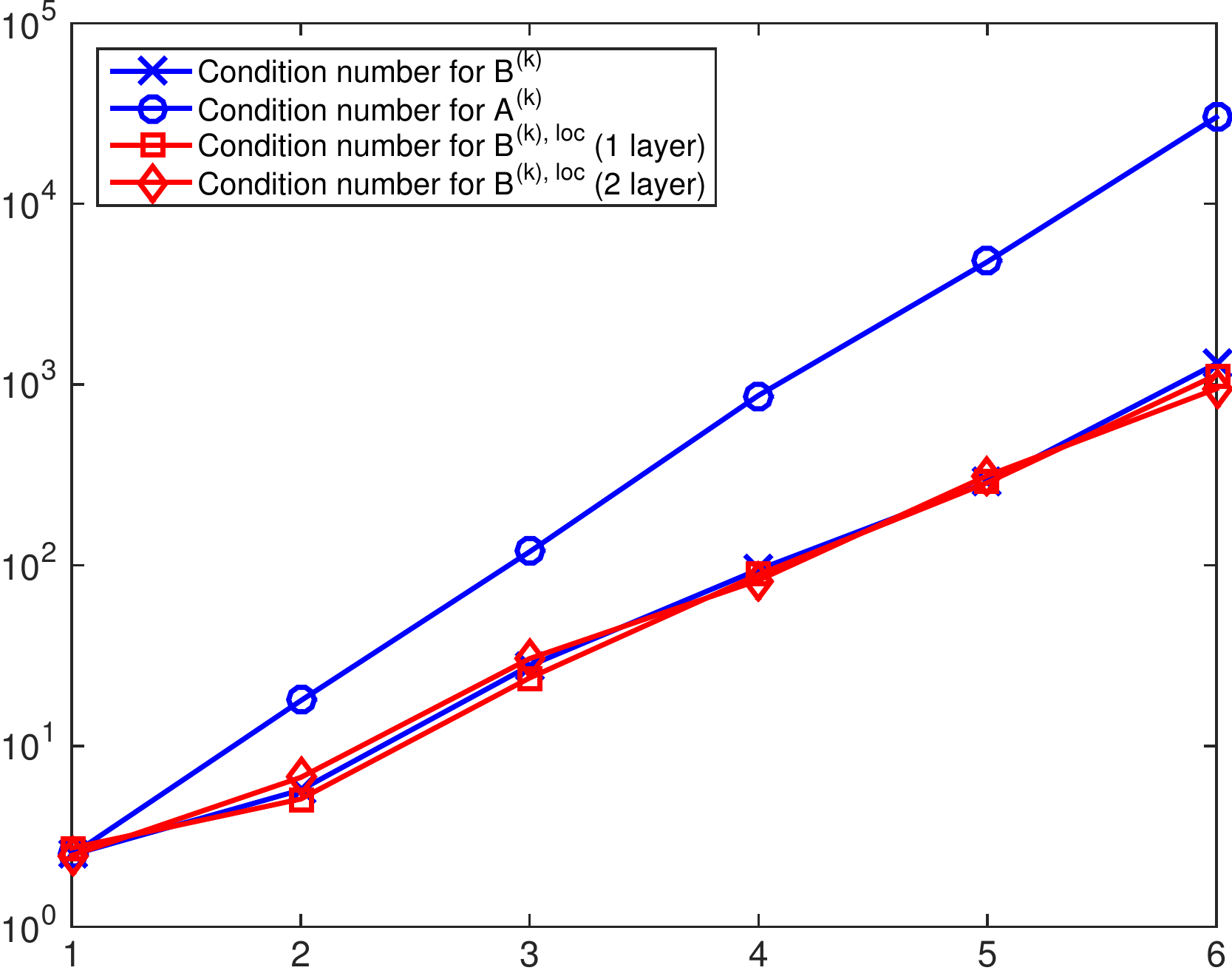}
    \end{subfigure}
    \begin{subfigure}[b]{0.5\textwidth}
        \includegraphics[width=\textwidth]{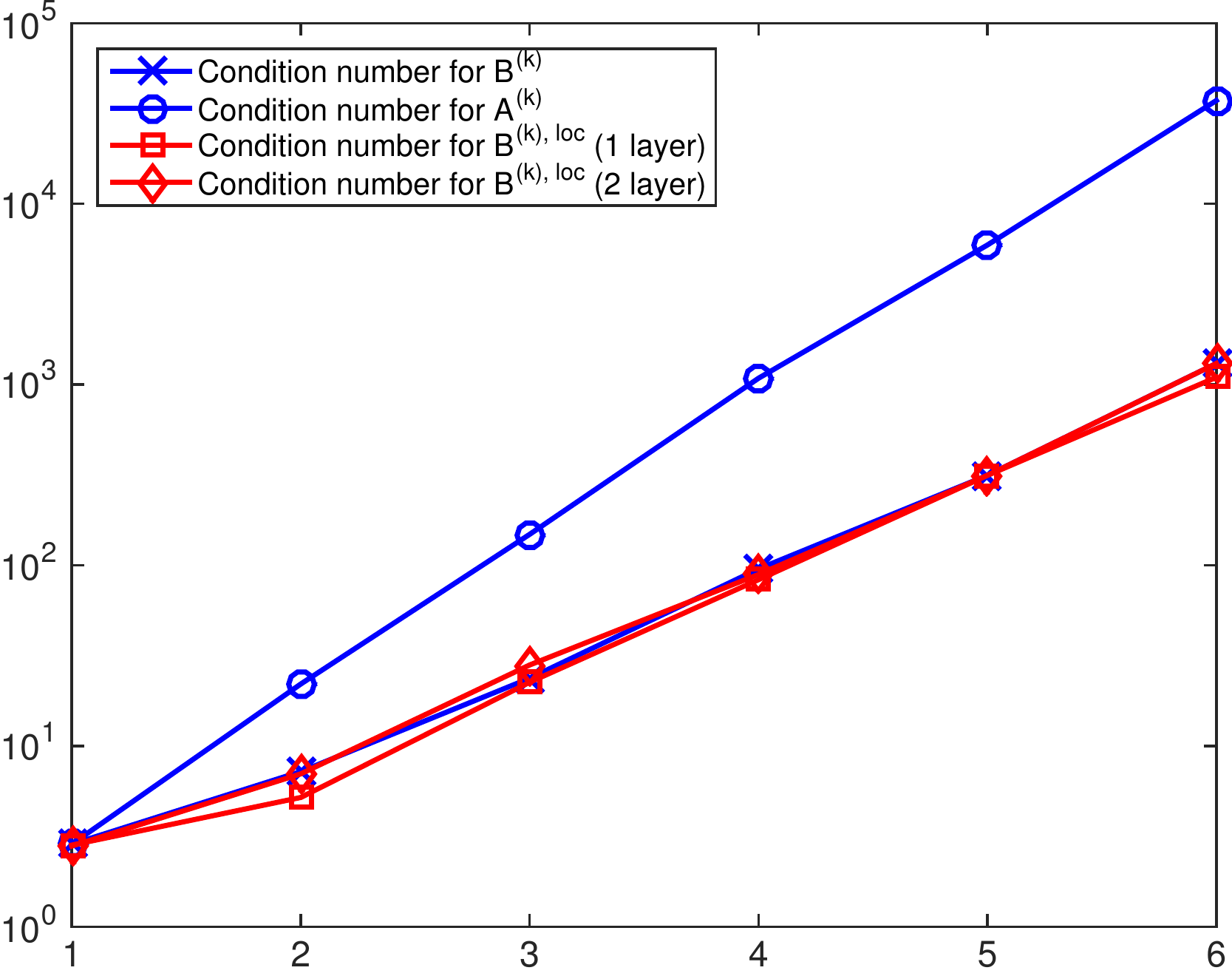}
    \end{subfigure}
    \caption{Left: Range of (the modulus of complex) eigenvalues in $\W^{k}$ associated with the first eigenvalue of Radau IIA and $\dt=1$. Right: Range of (real) eigenvalues in $\W^{k}$ associated with SDIRK3 and $\dt=1$.}
    \label{fig:condnum_para}
\end{figure}

\begin{figure}[H]
    \begin{subfigure}[b]{0.5\textwidth}
        \includegraphics[width=\textwidth]{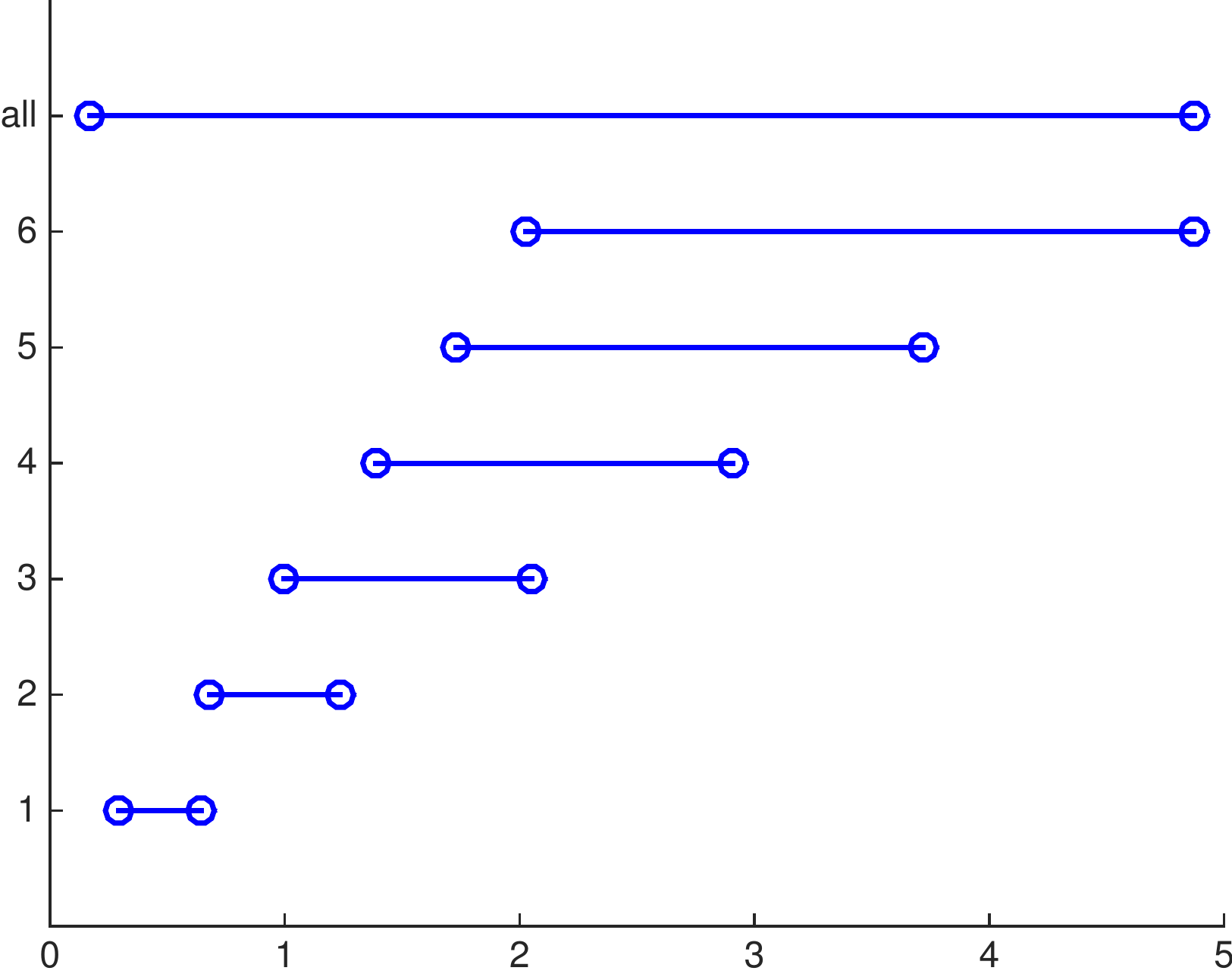}
    \end{subfigure}
    \begin{subfigure}[b]{0.5\textwidth}
        \includegraphics[width=\textwidth]{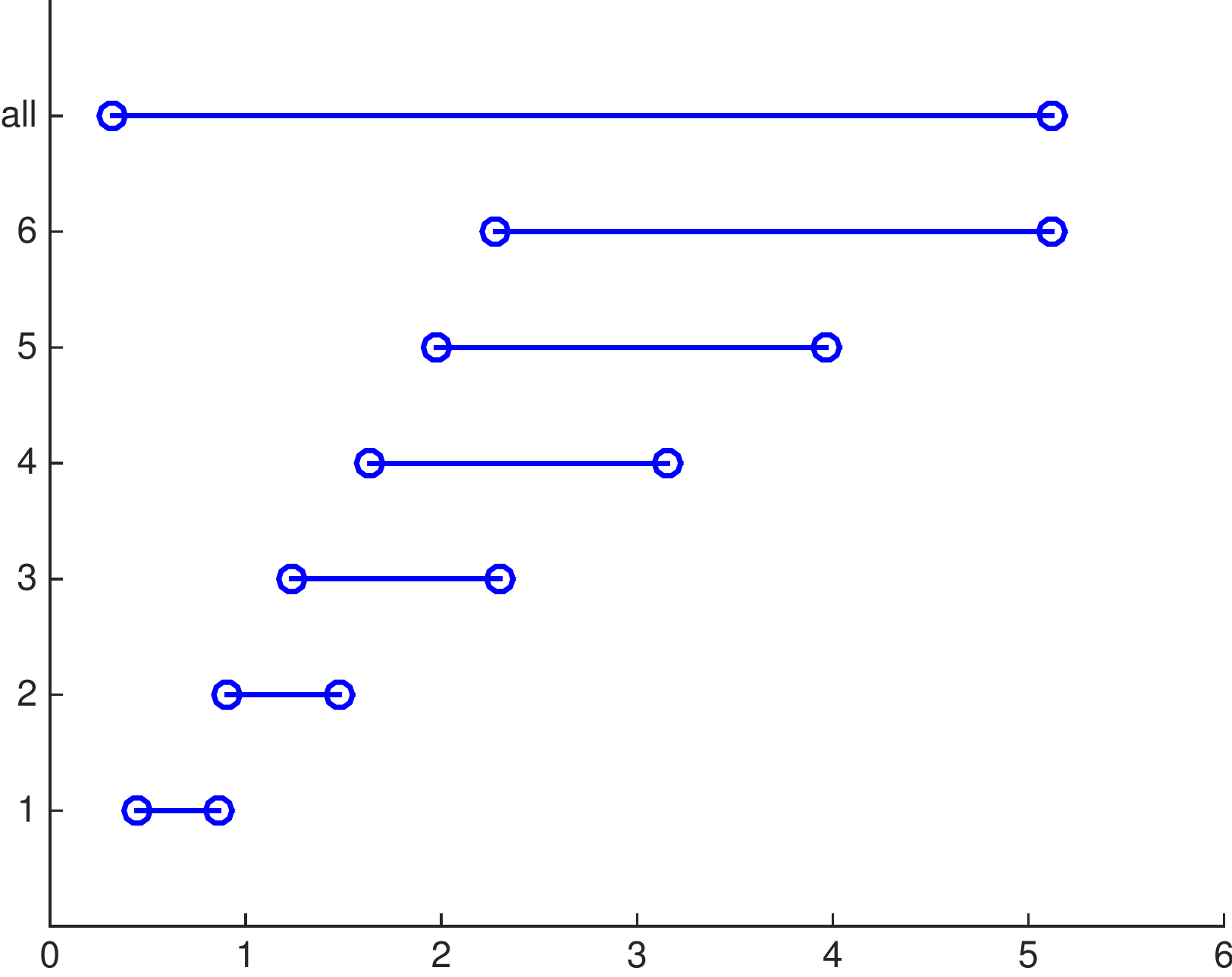}
    \end{subfigure}
    \caption{Left: Range of (the modulus of complex) eigenvalues in $\W^{k}$ associated with the first eigenvalue of Radau IIA and $\dt=1$. Right: Range of (real) eigenvalues in $\W^{k}$ associated with SDIRK3 and $\dt=1$.}
    \label{fig:eigrange_para}
\end{figure}

\section{Solving the parabolic equation in near-linear complexity with multi-time-stepping}\label{secmultitimestep}
As illustrated in Figure \ref{figure:band_i4para_nl3}, errors in finer subbands (corresponding to larger eigenvalues) decay quickly.
Therefore by refining time-steps close to the final stop time it is possible to lower the computational complexity and still preserve the accuracy.
Motivated by this observation we propose the following $\mathcal{O}(N \ln^{3d+1}N)$ complexity multi-time-stepping algorithm for solving
 \eqref{eqn:parabolicODE} (up to grid-size accuracy in energy norm). Let $T = M\dt$, prescribe an error threshold $\varepsilon<\dt$, then there exist
 $s\in \mathbb{N}$, such that $\displaystyle\frac{\dt}{2^{s}} \leq \varepsilon < \frac{\dt}{2^{s-1}}$.
For $n=1, 2, ..., M-1$, we use gamblets associated with time step $\dt$ to obtain $q_n$ up to $T-\dt$. For the last (coarse) time step from $T-\dt$ to $T$,
we subsequently choose time steps $\dt/2$, $\dt/4$, ..., $\dt/2^s$, $\dt/2^s$, and solve the implicit scheme with gamblets associated to those time steps.

 \begin{figure}[H]
    \begin{subfigure}[b]{0.5\textwidth}
        \includegraphics[width=\textwidth]{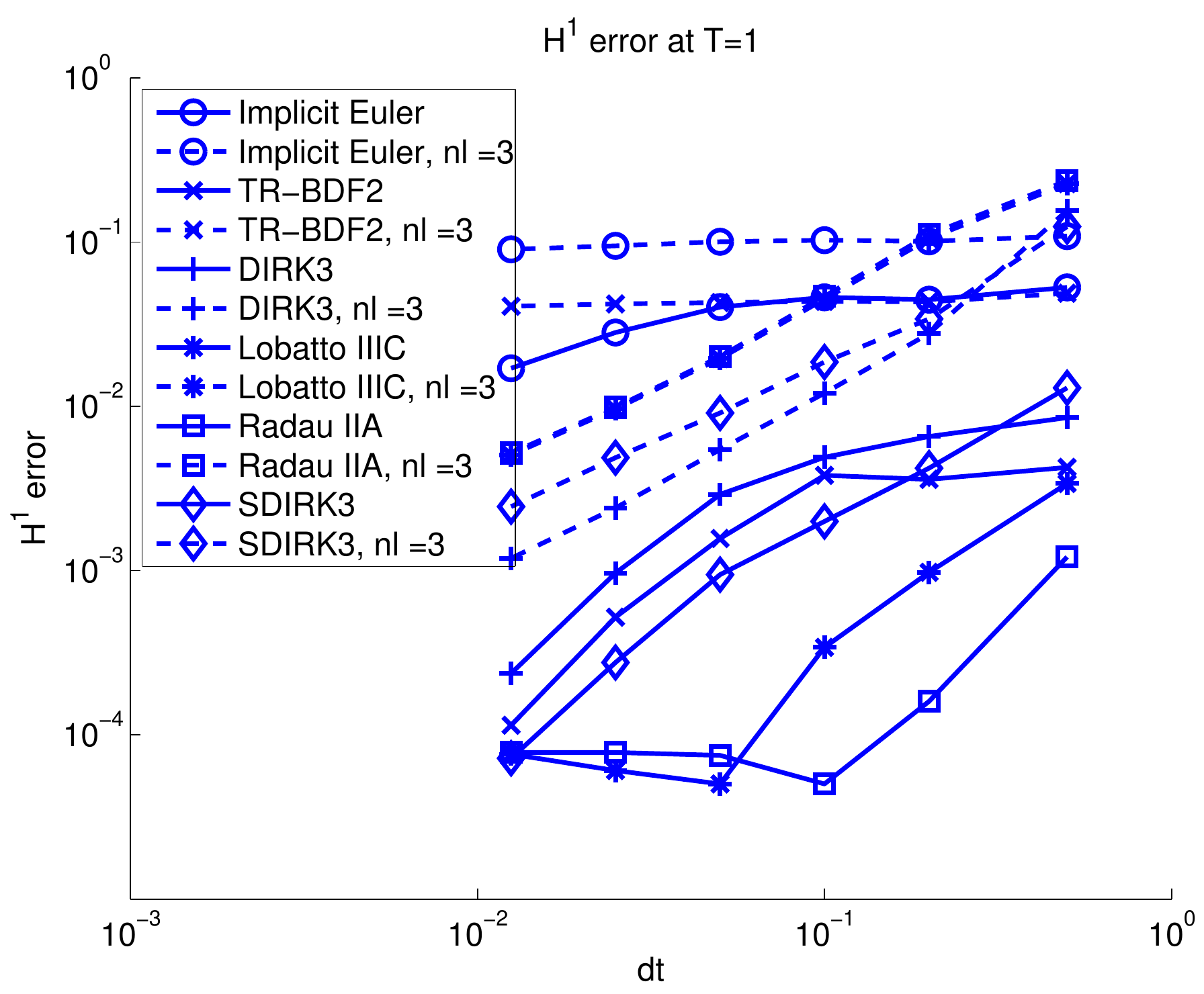}
    \end{subfigure}
    \begin{subfigure}[b]{0.5\textwidth}
        \includegraphics[width=\textwidth]{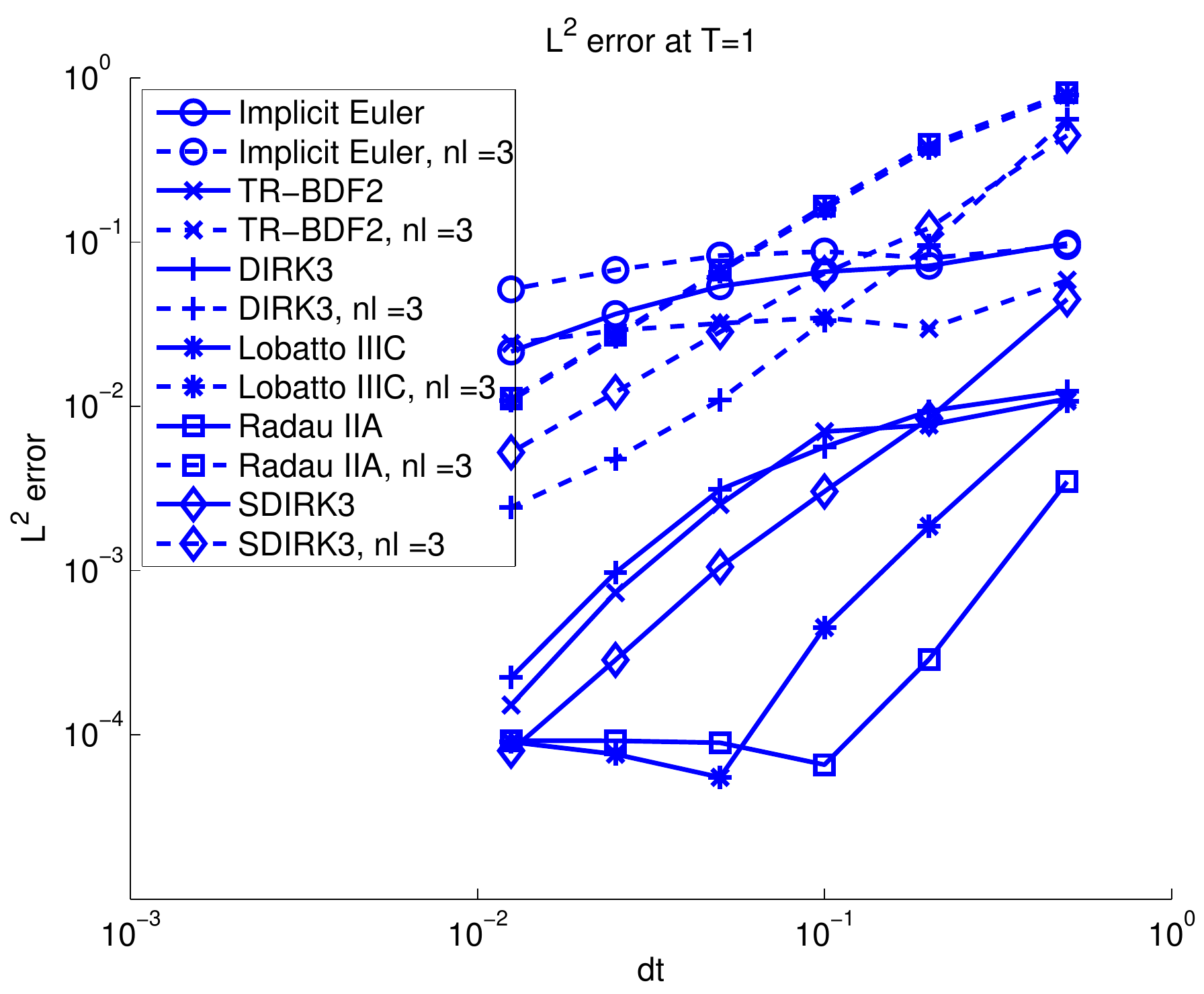}
    \end{subfigure}
    \caption{Solving parabolic equation with multi-time-stepping, $\varepsilon = 1/1280$, Left: $H^1$ error at time $T=1$; Right: $L^2$ error at time $T=1$.}\label{fig:h1errorat1_d1:para}
\end{figure}

Comparing Figures \ref{fig:h1errorat1_d1:para} and  \ref{fig:errorat1:para},   localized gamblets (with 3 layers) achieve $H^1$ error of $10^{-3}$ with multi-time-stepping and $\dt \simeq 0.02$,  we need $\dt \simeq 0.004$ to achieve the same accuracy with uniform time stepping.

\section{Appendix}
\label{sec:appendix}
\subsection{Proof of Theorem \ref{thmdiedhduw}}\label{secproof1}
We will need the following lemma,
\begin{Lemma}\label{lemguyg66g8}
Let $(q_n,p_n)$ be the solution of \eqref{implicitmidpoint}. Write $E_n:=\frac{1}{2}p_n^T M^{-1} p_n+\frac{1}{2}q_n^T K q_n$. Using the notation $|f|_{M^{-1}}:=\sqrt{f^T M^{-1} f}$ we have
\begin{equation}
|\sqrt{E_{n}}-\sqrt{E_0}|\leq \dt 2^{-1/2} \sum_{k=0}^{n-1}|f_{k+\frac{1}{2}}|_{M^{-1}}
\end{equation}
\end{Lemma}
\begin{proof}
Multiplying the first line of \eqref{implicitmidpoint} by $(q_{n+1}+q_n)^T K$, the second line by $(p_{n+1}+p_n)^T M^{-1}$ and summing together, we obtain that
\begin{displaymath}
E_{n+1}-E_n=\dt (p_{n+1}+p_n)^T M^{-1} \frac{f_{n+\frac{1}{2}}}{2}.
\end{displaymath}

Observe that
\begin{displaymath}
|(p_{n+1}+p_n)^T M^{-1} f_{n+\frac{1}{2}}| \leq  |p_{n+1}+p_n|_{M^{-1}} |f_{n+\frac{1}{2}}|_{M^{-1}} \leq \sqrt{2} (\sqrt{E_{n+1}}+\sqrt{E_{n}}) |f_{n+\frac{1}{2}}|_{M^{-1}}.
\end{displaymath}

We have $|\sqrt{E_{n+1}}-\sqrt{E_n}|\leq \dt 2^{-1/2}|f_{n+\frac{1}{2}}|_{M^{-1}}$, and we conclude the proof by induction.
\end{proof}
Let us now prove Theorem \ref{thmdiedhduw}.
Line \ref{linfemwavegx} of Algorithm \ref{wavequalgfemaltmix} and Theorems \ref{tmdiscrete} and \ref{tmdiscreteaccuracy} imply that
$(\frac{4}{(\dt)^2}M+ K) q_{n+1}^{\app} =b_n+(\frac{4}{(\dt)^2}M+ K) e_n$ with $b_n=(\frac{4}{(\dt)^2} M- K)q_n^{\app}+  \frac{4}{\dt} p_n^{\app}+2 f_{n+\frac{1}{2}}$ and
\begin{equation}\label{eqiidhud}
e_n^T (\frac{4}{(\dt)^2}M+ K)e_n \leq C \epsilon^2 b_n^T K^{-1} b_n
\end{equation}
Therefore, lines \ref{linfemwavegx} and \ref{linfemwavegx2} of Algorithm \ref{wavequalgfemaltmix}
can be written as,
\begin{equation}\label{implictitmihjhgdpret}
\begin{cases}
q_{n+1}^{\app}-q_n^{\app}&=\dt M^{-1}\frac{p_n^{\app}+p_{n+1}^{\app}}{2}+s_n\\
p_{n+1}^{\app}-p_n^{\app}&=- \dt K \frac{q_n^{\app}+q_{n+1}^{\app}}{2}+\dt f_{n+\frac{1}{2}}
\end{cases}
\end{equation}
with
\begin{equation}\label{eqidudhucmx}
s_n =\frac{(\dt)^2}{4}M^{-1}(\frac{4}{(\dt)^2}M+ K) e_n
\end{equation}
Write $q_{n}^\err:=q_{n}^{\app}-q_n$ and $p_{n}^\err:=p_{n}^{\app}-p_n$, together with \eqref{implictitmihjhgdpret} with \eqref{eqidudhucmx}, it leads to
\begin{equation}\label{eqimeximplsiuwszlocfem}
\begin{cases}
q_{n+1}^{\err}-q_n^{\err}&=\dt M^{-1}\frac{p_n^{\err}+p_{n+1}^{\err}}{2}+s_n\\
p_{n+1}^{\err}-p_n^{\err}&=- \dt K \frac{q_n^{\err}+q_{n+1}^{\err}}{2}
\end{cases}
\end{equation}
Write $E_n^\err:=\frac{1}{2}p_n^{\err,T} M^{-1} p_n^{\err}+\frac{1}{2}q_n^{\err,T} K q_n^{\err}$.
Left multiplying the first equation of \eqref{eqimeximplsiuwszlocfem} by $\frac{1}{2}(q_{n+1}^{\err}+q_n^{\err})^T K$ and the second equation by $\frac{1}{2} (p_{n+1}^{\err}+p_n^{\err})M^{-1}$, then adding the resulting equations we obtain that
\begin{equation}\label{eqimexswocfeer2m}
E^\err_{n+1}-E^\err_n=\frac{1}{2}(q_{n+1}^{\err}+q_n^{\err})^T K s_n
\end{equation}
Observing that $|(q_{n+1}^{\err}+q_n^{\err})^T K s_n|\leq \sqrt{2}(\sqrt{E^\err_{n+1}}+\sqrt{E^\err_n})|s_n|_{K }$, we have
\begin{equation}\label{eqimexswocfeesser2m}
|\sqrt{E^\err_{n+1}}-\sqrt{E^\err_n}|\leq 2^{-\frac{1}{2}}|s_n|_{K }
\end{equation}
Write $B:=\frac{4}{(\dt)^2}M+ K$, \eqref{eqidudhucmx} and \eqref{eqiidhud} imply that
\begin{equation}
(\frac{4}{(\dt)^2})^2   s_n^T M  B^{-1} M s_n \leq C \epsilon^2 b_n^T K^{-1} b_n\,.
\end{equation}
We also have
\begin{equation}
 s_n^T K s_n \leq \frac{\lambda_{\max}(K) \lambda_{\max}(B)}{(\lambda_{\min}(M))^2} s_n^T M  B^{-1} M s_n,
\end{equation} and
\begin{equation}
\begin{split}|b_n|_{K^{-1}}\leq &\big(\frac{4}{(\dt)^2} \frac{\lambda_{\max}(M)}{\lambda_{\min}(K)} +1)  |q_n^{\app}|_{K}+\frac{4}{\dt} \sqrt{\frac{\lambda_{\max}(M)}{\lambda_{\min}(K)}}|p_n^{\app}|_{M^{-1}}\\&+
2 \sqrt{\frac{\lambda_{\max}(M)}{\lambda_{\min}(K)}} \|g_{n+\frac{1}{2}}\|_{L^2(\Omega)}.
\end{split}\end{equation}
Poincar\'{e}'s inequality, \eqref{eqhhgfff65f} and \eqref{eqhhgfff65fip} lead to
\begin{align}
\lambda_{\max}(M) \leq C N^{-1}, \qquad & C^{-1} N^{-1} \leq \lambda_{\min}(M),\\
\lambda_{\max}(K) \leq C N^{-1} h^{-2}, \qquad & C^{-1} N^{-1} \leq \lambda_{\min}(K)\,.
\end{align}
Summarizing we have obtained that
$|s_n|_{K }^2 \leq C N h^{-2} (\frac{4}{(\dt)^2}N^{-1}+ h^{-2} N^{-1})$ $\dt^4 \epsilon^2 b_n^T K^{-1} b_n
$, which implies
\begin{equation}
|s_n|_{K } \leq C  h^{-1}( 1+\frac{\dt}{h})  \epsilon \big(\dt^{-1} \sqrt{E_n^{\app}} +\dt \|g_{n+\frac{1}{2}}\|_{L^2(\Omega)}\big)\,.
\end{equation}
where $E_n^\app:=\frac{1}{2}p_n^{\app,T} M^{-1} p_n^{\app}+\frac{1}{2}q_n^{\app,T} K q_n^{\app}$. Recall that $E_n:=\frac{1}{2}p_n^{T} M^{-1} p_n+\frac{1}{2}q_n^{T} K q_n$,
using $\sqrt{E_n^{\app}}\leq \sqrt{E_n}+\sqrt{E_n^{\err}}$, we deduce from \eqref{eqimexswocfeesser2m} that
\begin{equation}\label{eqimytfsser2djum}
\sqrt{E^\err_{n+1}}-\sqrt{E^\err_n}\leq (c_n +z \sqrt{E^\err_n})
\end{equation}
with $ c_n=C ( \frac{1}{h}+ \frac{\dt}{h^2})  \epsilon \big(\dt^{-1}   \sqrt{E_n} +\dt \|g_{n+\frac{1}{2}}\|_{L^2(\Omega)}\big)$ and
$ z = C  \frac{1}{h}( \frac{1}{\dt}+\frac{1}{h})  \epsilon$. Therefore (using $E_0^{\err}=0$) we obtain that $\sqrt{E^\err_{n}}\leq \sum_{k=0}^{n-1} c_k (1+z)^{n-k-1}$.
For $ \epsilon \leq \frac{1}{4}C^{-1}\frac{\dt}{T}h(\dt+h)$, we have $z\leq \dt/ T$, therefore $(1+z)^{n}\leq  e^1$ and $\sqrt{E^\err_{n}}\leq C \sum_{k=0}^{n-1} c_k$.
Using $n \dt \leq T$ and $\|g_{k+\frac{1}{2}}\|_{L^2(\Omega)} \leq  \|g\|_{L^\infty(0,T,L^2(\Omega))}$,
Lemma \ref{lemguyg66g8} implies that
\[
	\sqrt{E_{n}}\leq \sqrt{E_0}+ \dt 2^{-1/2} \sum_{k=0}^{n-1}\|g_{k+\frac{1}{2}}\|_{L^2(\Omega)}\leq \sqrt{E_0}+ T  2^{-1/2} \|g\|_{L^\infty(0,T,L^2(\Omega))}.
\]
Using
$\sum_{k=0}^{n-1} c_k  \leq C( \frac{1}{h}+ \frac{\dt}{h^2})  \epsilon\big( \dt^{-1} \sum_{k=0}^{n-1} \sqrt{E_k} +\dt
\sum_{k=0}^{n-1} \|g_{k+\frac{1}{2}}\|_{L^2(\Omega)}\big)$ we obtain that $\sum_{k=0}^{n-1} c_k
										 \leq C \dt^{-2} ( \frac{1}{h}+ \frac{\dt}{h^2}) \epsilon T \big(   \sqrt{E_0}+ T   \|g\|_{L^\infty(0,T,L^2(\Omega))}\big)$ and
\begin{equation}
\sqrt{E_{n}^\err}\leq \dt^2 \big(  \sqrt{E_0}+ T   \|g\|_{L^\infty(0,T,L^2(\Omega))}\big)
\end{equation}
for $\epsilon \leq C^{-1}\frac{1}{T}\dt^4\frac{h}{\dt}(\dt+h)$.
We conclude the proof by observing that
$2 E^\err_{n}=\int_{\Omega}(\nabla u_n- \nabla u^\app_n)^T a (\nabla u_n- \nabla u^\app_n) +\int_{\Omega}(v_n-v^\app_n)^2 \mu$.

\subsection{Stability}

\begin{Definition}
A function $f(t,x)$ is \textit{dissipative} if $(f(t,y)-f(t,z)), (y-z))\leq 0$ for all $y$ and $z$. An ODE is \textit{contractive} if $\|y(t)-z(t)\|\leq \|y(s)-z(s)\|$ for every pair of solutions $y$ and $z$ when $t\geq s$. Every ODE with a dissipative right-hand side $f$ is contractive.
\end{Definition}

It is easy to see that equation \eqref{eqn:parabolicFE} is contractive.

\begin{Definition}
\label{def:b-stability}
A numerical method is B-stable (or contractive) if every pair of numerical solutions $u$ and $v$ satisfy $\|u_{n+1} - v_{n+1}\|\leq \|u_n - v_n\|$ for all $n\geq 0$, when solving an IVP with a dissipative $f$.
\end{Definition}

\begin{Definition}
	A Runge-Kutta method is algebraically stable if the matrices
	\begin{displaymath}	
		B = \diag(b_1, \dots, b_s), \qquad M = BA + A^TB^T - bb^T
	\end{displaymath}
	are nonnegative semidefinite. An algebraically stable Runge Kutta method is B-stable.
\end{Definition}

\subsection{Proof of Theorem \ref{thm:implicitEuler}}\label{subsecproofthmimplicitEuler}

The implicit Euler scheme \eqref{eqkjddkhsthimeul} can be written as
\begin{displaymath}
	(\frac{M}{\dt}+K)q_{n+1} = \frac{M}{\dt}q_n + f_{n+1}.
\end{displaymath}
 Using localized gamblets, \eqref{eqkjddkhsthimeul} is solved up to error $e_n$, i.e.
\begin{displaymath}
	(\frac{M}{\dt}+K)q_{n+1}^\app = \frac{M}{\dt}q_n^\app + f_{n+1} + (\frac{M}{\dt}+K)e_n
\end{displaymath}	
and $e_n$ satisfies, $e_n^T (\frac{M}{\dt} + K) e_n \leq C\varepsilon^2 b_n^T K^{-1} b_n$. Write $\|q\|_\zeta^2:=q^T (\frac{M}{\dt}+K)q$.

\begin{Lemma}
\label{lem:implicitEuler}
It holds true that $\|q_{n+1}\|_\zeta \leq \|q_n\|_\zeta + \|f_{n+1}\|_{H^{-1}(\Omega)}$.
\end{Lemma}
\begin{proof}
	Multiplying \eqref{eqkjddkhsthimeul} by $q_{n+1}$ and using Young's inequality we obtain that
$q_{n+1}^T(\frac{M}{\dt}+K)q_{n+1}  = q_{n+1}^T \frac{M}{\dt} q_n + q^T_{n+1} f_{n+1}$ and $\|q_{n+1}\|_\zeta^2 \leq \frac12 q^T_{n+1} \frac{M}{\dt} q_{n+1} + \frac12 q_n^T \frac{M}{\dt} q_n + \frac12 q^T_{n+1} K q_{n+1} + \frac12 f^T_{n+1} K^{-1} f_{n+1}$.
Therefore,
\begin{displaymath}
	q_{n+1}^T(\frac{M}{\dt}+K)q_{n+1} \leq q_n^T (\frac{M}{\dt}+K) q_n + f^T_{n+1} K^{-1} f_{n+1}.
\end{displaymath}
which concludes the proof of Lemma \ref{lem:implicitEuler}.
\end{proof}
Let $\epsilon_n := q_n - q_n^\app$, then we have
$
	(\frac{M}{\dt}+K)\epsilon_{n+1} = \frac{M}{\dt}\epsilon_n + (\frac{M}{\dt}+K)e_n	
$
and
$(\frac{M}{\dt}+K)(\epsilon_{n+1}-e_n) = \frac{M}{\dt}\epsilon_n	
$.
Multiplying by $\epsilon_{n+1} - e_n$ and using Young's inequality, we obtain that
$(\epsilon_{n+1} - e_n)^T (\frac{M}{\dt} + K) (\epsilon_{n+1} - e_n)  = (\epsilon_{n+1} - e_n)^T (\frac{M}{\dt}) \epsilon_n$
and $\|\epsilon_{n+1}-e_n\|_\zeta^2 \leq \frac12 (\epsilon_{n+1} - e_n)^T (\frac{M}{\dt} + K) (\epsilon_{n+1} - e_n) + \frac12 \epsilon_n^T (\frac{M}{\dt} + K) \epsilon_n$. Therefore,
$
	\|\epsilon_{n+1}-e_n\|_\zeta \leq \|\epsilon_n\|_\zeta
$
and
\begin{equation}
\label{eqn:veineq}
	\|\epsilon_{n+1}\|_\zeta\leq \|\epsilon_n\|_\zeta + \|e_n\|_\zeta\,.
\end{equation}
Since $\|e_n\|^2_\zeta \leq C\varepsilon^2 b_n^T K^{-1} b_n$, and $b_n = \frac{M}{\dt} q_n^\app + f_{n+1}$, we have
$
	b_n^TK^{-1} b_n  = q_n^\app \frac{M}{\dt} K^{-1} \frac{M}{\dt} q_n^\app + f_{n+1}K^{-1} f_{n+1}$ and
	$b_n^TK^{-1} b_n\leq q_n^\app K q_n^\app (\frac{\lambda_{\max} (M)}{\lambda_{\min} (K)})^2 + f_{n+1}K^{-1} f_{n+1}$.
Therefore,
$\|b_n\|_{K^{-1}} \leq \|q_n^\app\|_\zeta + \|f_{n+1}\|_{K^{-1}(\Omega)}$. Using
 \eqref{eqn:veineq} we deduce that
$\|e_n\|_\zeta  \leq C \epsilon (\|q_n\|_\zeta + \|\epsilon_n\|_\zeta + \|f_{n+1}\|_{K^{-1}(\Omega)}) \leq C \epsilon \frac{T}{\dt} \|g\|_{L^\infty(0, T, H^{-1}(\Omega))} + C\epsilon \|\epsilon_n\|_\zeta
$. Hence,
$\|\epsilon_{n+1}\|_\zeta  \leq C \frac{T}{\dt} \epsilon \|g\|_{L^\infty(0, T, H^{-1}(\Omega))} + (1+C\epsilon)\|\epsilon_n\|_\zeta  \leq C (\frac{T}{\dt})^2 \epsilon  \|g\|_{L^\infty(0, T, H^{-1}(\Omega))}$,
which finishes the proof of Theorem \ref{thm:implicitEuler}.

\paragraph{Acknowledgements.}
H. Owhadi gratefully acknowledges the support of the Air Force Office of Scientific Research and the DARPA EQUiPS Program under
award  number FA9550-16-1-0054 (Computational Information Games). L. Zhang gratefully acknowledges the support of the National Natural Science Foundation of China grant 11471214 and the One Thousand Plan of China for young scientists. The authors also thank two anonymous referees for comments and suggestions.

\bibliographystyle{plain}
\bibliography{RPS}

\end{document}